\documentclass[11pt, reqno]{amsart}

\usepackage{amssymb,latexsym}
\usepackage[active]{srcltx}
\usepackage{amsmath}

\usepackage{amsthm}
\usepackage{amscd}
\usepackage{mathrsfs}
\usepackage{mathtools}
\usepackage{fancyhdr}
\usepackage{graphicx}
\usepackage{calc}
\usepackage{nccmath}
\usepackage{enumitem}
\setlist[enumerate]{label=(\arabic*), font=\normalfont}
\usepackage{fullpage}
\usepackage[cmtip,arrow,matrix,curve,tips,frame]{xy}
\usepackage{cite}     
\usepackage{hyperref}
\xyoption{matrix}

\usepackage{ifthen}
\usepackage{epsfig}
\usepackage{color}
\usepackage{amsxtra}
\usepackage{amstext}
\usepackage{amssymb}
\usepackage{stmaryrd}
\usepackage{mathrsfs}
\usepackage{pifont}
\usepackage{lmodern}
\usepackage{textcomp}
\usepackage{bbm}
\usepackage{comment}

\usepackage{tikz} 
\usetikzlibrary{matrix,arrows,calc,decorations.pathreplacing,fit}
\usepackage{tikz-cd}

\usepackage{xcolor}
\hypersetup{%
  colorlinks=false,
}

\numberwithin{equation}{section}

\theoremstyle{plain}
\newtheorem{thm}{Theorem}[section]
\newtheorem{lem}[thm]{Lemma}
\newtheorem{prop}[thm]{Proposition}

\theoremstyle{definition}
\newtheorem{defn}[thm]{Definition}

\newtheorem{rmk}[thm]{Remark}

\numberwithin{figure}{section}
\numberwithin{table}{section}

\def \AA {\mathbb{A}}

\def \CC {\mathbb{C}}

\def \NN {\mathbb{N}}

\def \RR {\mathbb{R}}
\def \SS {\mathbb{S}}
\def \TT {\mathbb{T}}

\def \ZZ {\mathbb{Z}}

\def \Bcal {\mathcal{B}}

\def \Dcal {\mathcal{D}}
\def \Ecal {\mathcal{E}}

\def \Ical {\mathcal{I}}

\def \Mcal {\mathcal{M}}

\def \Ocal {\mathcal{O}}
\def \Pcal {\mathcal{P}}

\def \Rcal {\mathcal{R}}
\def \Scal {\mathcal{S}}
\def \Tcal {\mathcal{T}}

\def \Vcal {\mathcal{V}}

\def \Dfr {\mathfrak{D}}

\def \Sfr {\mathfrak{S}}

\def \bfr {\mathfrak{b}}
\def \cfr {\mathfrak{c}}

\def \lfr {\mathfrak{l}}

\def \sfr {\mathfrak{s}}

\def \fhat {\widehat{f}}
\def \ghat {\widehat{g}}

\def \khat {\widehat{k}}

\def \Abar {\bar{A}}

\def \Ebar {\bar{E}}

\def \Kbar {\bar{K}}

\def \Mbar {\bar{M}}

\def \Sbar {\bar{S}}

\def \cbar {\bar{c}}

\def \hbar {\bar{h}}

\def \jbar {\bar{j}}

\def \Qtil {\widetilde{Q}}

\def \Xbf {\mathbf{X}}

\def \qbf {\mathbf{q}}

\newcommand{\abs}[1]{\left|#1\right|}

\newcommand{\norm}[1]{\left\|#1\right\|}

\newcommand{\jap}[1]{\left<#1\right>}

\newcommand{\wh}[1]{\widehat{#1}}

\DeclareMathOperator{\Vol}{Vol}

\DeclareMathOperator{\supp}{supp}

\DeclareMathOperator{\sgn}{sgn}

\DeclareMathOperator{\tr}{tr}
\DeclareMathOperator{\Ric}{Ric}

\let\ringaccent\r

\renewcommand{\a}{\alpha}
\renewcommand{\b}{\beta}
\newcommand{\g}{\gamma}
\newcommand{\G}{\Gamma}
\renewcommand{\d}{\delta}

\newcommand{\e}{\varepsilon}

\newcommand{\et}{\eta}
\renewcommand{\th}{\theta}

\renewcommand{\l}{\lambda}

\newcommand{\m}{\mu}
\newcommand{\n}{\nu}
\newcommand{\x}{\xi}

\renewcommand{\r}{\rho}

\renewcommand{\t}{\tau}

\renewcommand{\o}{\omega}
\renewcommand{\O}{\Omega}

\DeclareTextCommand{\aa}{OT1}{\ringaccent a}
\DeclareTextCommand{\AA}{OT1}{\ringaccent A}

\newcommand{\nb}{\nabla}
\newcommand{\rd}{\partial}

\newcommand{\rhotil}{\tilde{\rho}}

\newcommand{\header}[1]{\noindent\textbf{#1}}

\definecolor{green1}{rgb}{0,0.5,0.2}

\newcommand{\step}[1]{\noindent{\textit{Step#1}) }}

\renewcommand{\max}[1]{\mathrm{max}(#1)}

\author[Eo]{Saehoon Eo}
\address[Saehoon Eo]
{Department of Mathematics, Stanford University, 450 Jane Stanford Way, Stanford, CA 94305, USA
}
\email{eosehoon@stanford.edu}

\date{\today}
\newcommand{\brm}{\mathrm{b}}
\newcommand{\lin}{\mathrm{(lin)}}
\newcommand{\jt}{\jap{t}}
\newcommand{\Sch}{\mathrm{Sch}}
\newcommand{\init}{\mathrm{in}}
\title{Nonlinear Phase Mixing in Einstein--Vlasov System near Schwarzschild Spacetime}
\begin{document}

\begin{abstract}
We study the Einstein–-Vlasov system near the Schwarzschild spacetime in spherical symmetry.
In particular, we consider the case when the matter is supported on bounded geodesics. 
In this regime, the decay of the matter is driven by phase mixing rather than dispersion. 
For the linearized problem, we obtain quantitative phase mixing; the key new ingredient proved here is a monotonicity property of the period function.
For the nonlinear coupled problem, we prove phase mixing up to the timescale expected from the nonlinear echo mechanism; this is achieved by analysis on dynamical action angle variables and the integral estimates with the vector field method.  
This construction is made possible by our gauge choice, an isotropic coordinate system together with a foliation of prescribed mean curvature, in which the Einstein equations gain derivatives over the matter due to its elliptic nature.
\end{abstract}

\maketitle

\tableofcontents

\section{Introduction}\label{sec:Introduction}
\subsection{Einstein--Vlasov system}
The Einstein--Vlasov system describes the mesoscopic dynamics of a self-gravitating collisionless gas of which particles interact with each other through gravity described by the Einstein's general relativity.
The metric solves the Einstein equation with matter defined by the integral of the distribution function, and the distribution function solves the Vlasov equation on the curved spacetime defined by the metric.

We are considering a \((1+3)\)-dimensional Lorentzian manifold \((\Mcal, g)\) coupled with a distribution function \(f:\Pcal \to \RR_+\) defined on the mass shell 
\[
\Pcal = \{(x, p) \mid g_x(p, p) = -1, p \text{ future-directed}\} \subset \Tcal\Mcal
\] 
Those are coupled with the Einstein--Vlasov equations: 
The metric \(g\) satisfies the Einstein equations 
\begin{equation}\label{eq:g}
\Ric(g)_{\mu\nu} - \frac{1}{2} \Sbar(g) g_{\mu\nu} = 8\pi T_{\mu\nu},
\end{equation}
where \(\Ric(g)\) and \(\Sbar(g)\) denote the Ricci and scalar curvature of \(g\), and the energy-momentum tensor \(T_{\mu\nu}\) is given from the distribution function \(f\) as
\begin{equation}\label{eq:T}
T^{\mu\nu}(x) = \int_{\Pcal_x} f(x, p)p^\m p^\n \frac{\sqrt{\abs{\det g}}}{p^0} dp^1 dp^2 dp^3.
\end{equation}
Here, \(\Pcal_x\) denotes the mass shell at a fixed spacetime point \(x\), i.e.,
\[
\Pcal_x = \{p \in \Tcal_x\Mcal \mid g_x(p, p) = -1, p \text{ future-directed}\}.
\]
Further, the distribution function \(f\) satisfies the Vlasov equation on the curved spacetime \((\Mcal, g)\), which is given by 
\begin{equation}\label{eq:f}
\Xbf (f) = 0,
\end{equation}
for the vector 
\[
\Xbf = p^\m \rd_{x^\m} - p^\a p^\b\G^i_{\a\b}\rd_p^i.
\]
This is equivalent to say that \(f\) flows by the geodesic flow on the mass shell \(\Pcal\).
Here, we give a coordinate \((s = x^0, x^1, x^2, x^3, p^1, p^2, p^3)\) on \(\Pcal\), and view \(f\) as a function of these variables.
We use the usual summation convention, namely, the greek indices run from \(0\) to \(3\) and Latin indices run from \(1\) to \(3\).
The Cauchy problem for this system is given by prescribing the initial data on a spacelike hypersurface, which consists of the induced metric \(\ghat\), the second fundamental form \(\khat\) of the hypersurface, and the restriction of \(f\) on the mass shell over this hypersurface.

There are some results concerning this system, especially on the local and global existence, and long-time dynamics. 
In case when the particle is \textit{massless}, i.e. when particles follow null geodesics, the stability of Minkowski space, which can be understood as a zero solution, under spherical symmetry was proved by Dafermos \cite{Daf06}, whereas the \textit{massive} case was studied by Rein--Rendall \cite{RR92}.
Without symmetry assumption, there was a result on the global existence of small data by Taylor \cite{Tay17} for the massless case, and the stability of Minkowski space in massive case was proved in \cite{LT20,FJS21,Wan22}.

Compared to this stability of Minkowski problem for Einstein--Vlasov, or the stability of Schwarzschild spacetime for the vacuum Einstein equation, the stability of the Schwarzschild problem for the Einstein--Vlasov system presents  distinctive challenges.
Because the set of geodesics which stay in a compact subset of the Schwarzschild spacetime is stable, we cannot hope the matter to disperse naturally. 
Even worse, it is known that the Einstein--Vlasov space admits static solutions which are arbitrarily close to the Schwarzschild spacetime, and hence we cannot expect the solution to converge back to the Schwarzschild spacetime in general. (See \cite{Rei94, Jab21} and \cite{GRS25}.)
Therefore, the best that we can expect is that the solution converges to a static solution, or that the time derivative of the solution decays in time.

This is in stark contrast with the Schwarzschild stability problem for other matter models, where one can exploit dispersive behavior: for Einstein--Maxwell, see for instance \cite{Pas19b, Blu08, ST15}; for Einstein--scalar field, see \cite{DR05}.
For the Einstein--Vlasov system, Velozo Ruiz \cite{Vel22} treats the regimes where the matter is massless, or massive and dispersive, so that particles either escape to infinity or collapse into the black hole.

\subsection{Spherical symmetry and isotropic coordinate system}
In this paper, we are considering spherically symmetric solutions to the Einstein--Vlasov system. 
In particular, we are considering the isotropic coordinate system, where the metric takes the following form: 
\[
g(s, R, \th, \phi) = -\a^2(s, R) ds^2 + A^2(s, R)[(dR+\b(s, R) ds)^2 + R^2 d\O^2],
\]
where \(d\O^2\) is the standard round metric on \(\SS^2\).
Schwarzschild metric in this coordinate system has the following coefficients:
\begin{equation}\label{sch-metric}
\a_0 = \frac{2R-m}{2R+m}, \qquad
A_0 = \Big(1+\frac{m}{2R}\Big)^2, \qquad
\b_0 = 0.
\end{equation}
Notice that, given a spherically symmetric metric on a \((1+3)\)-dimensional Lorentzian manifold, there are many coordinate systems in which the metric takes the above form. 
For a given foliation consisting of spacelike hypersurfaces, we can choose the time function \(s\) to be constant on each leaf of the foliation, and then we can choose the radial coordinate \(R\) which makes the induced metric on each leaf conformal to the standard flat metric.
Note that, the coordinate \(R\) is different from the areal radius, which in this case given as \(r = AR\). 
Especially, \(R\) depends on the choice of the foliation, and hence is not a geometric quantity. 

The choice of the isotropic coordinate system is motivated by the nonlinear analysis. 
In particular, this coordinate system is chosen so that the Einstein equations can be written as a system of ODEs \eqref{absODE}, which allows us to gain two derivatives from the matter to the metric.
This is a distinctive feature of the isotropic coordinate system.
For example, in a double-null coordinate system, it does not seem that this kind of gain is possible, mainly because the second derivatives \(\rd_u\rd_u\) or \(\rd_v\rd_v\) do not appear in the equation. 
By contrast, in the isotropic coordinate system, the second derivative \(\rd_R\rd_R\) of \(A\) appears in the constraint equation \eqref{eq:Gauss}, and that of \(\a\) appears in \eqref{eq:SumEE}, so that we gain two derivatives.

Under symmetry, the distribution function \(f\) evaluated on the mass shell depends on its argument only through \((s, R, w, \ell)\), which are defined as follows: 
\[
R = \sqrt{(x^1)^2 + (x^2)^2 + (x^3)^2}
\]
as usual, and we use \(v\) defined by 
\[
v^0 = \a p^0, \qquad v^i = A\big[p^i+\frac{\b x^i}{R}p^0\big].
\]
Notice that we have 
\[
\abs{v}^2 \coloneqq (v^1)^2 + (v^2)^2 + (v^3)^2 = -1 + (v^0)^2.
\]
Then, we set \(w, \ell\) as 
\[
w(s, x, p) = \frac{v^i x^i}{R}, \qquad \ell(s, x, p) = A^2 (R^2\abs{v}^2 - (v^ix^i)^2).
\]
In this setting, we can also rewrite the integral that defines the energy-momentum tensor. 
\(\r, \tr T, j\), \(S_R\) are all coming from the energy-momentum tensor, which is an integral of \(f\) on each fiber. 
The energy-momentum tensor is given by 
\[
T(w_1, w_2) = \int f g(p, w_1)g(p, w_2) d\Vol_p,
\]
where \(d\Vol_p\) is the volume form on the tangent space of the spacetime at a given point endowed with the induced metric from \(g\). 
Since the coordinate given by \((v^1, v^2, v^3)\) has orthogonal coordinate vectors, we can write the volume form explicitly with Euclidean geometry, which is 
\[
d\Vol_p = \frac{1}{\sqrt{1 + \abs{v}^2}}dv^1 dv^2 dv^3.
\]
Since our coordinate for the tangent space \((w, \ell)\) is defined in terms of \(v\), this allows us to write the volume form in terms of \(w\) and \(\ell\). 
This is given by 
\[
d\Vol_p = \frac{\pi}{R^2A^2E}d\ell dw,
\]
after trivial integration along the spherical symmetry. 
Here, \(E\) is defined by \(E^2 \coloneqq 1 + w^2 + \frac{\ell}{A^2R^2}\). 
Therefore, we have the following equations for the density variables.
\begin{align}
\r \coloneqq T(\vec{n}, \vec{n}) &= \frac{\pi}{R^2A^2}\int_{-\infty}^\infty\int_0^\infty f(s, R, w, \ell)E \,d\ell dw, \label{rhodef}\\
\tr T &= -\frac{\pi}{R^2A^2}\int_{-\infty}^\infty\int_0^\infty f(s, R, w, \ell)\frac{1}{E} \,d\ell dw, \label{trTdef}\\
j \coloneqq -\frac{1}{A}T(\vec{n}, \rd_R) &= \frac{\pi}{R^2A^2}\int_{-\infty}^\infty\int_0^\infty f(s, R, w, \ell)w \,d\ell dw, \label{jdef}\\
S_R \coloneqq \frac{1}{A^2}T(\rd_R, \rd_R) &= \frac{\pi}{R^2A^2}\int_{-\infty}^\infty\int_0^\infty f(s, R, w, \ell)\frac{w^2}{E} \,d\ell dw.
\label{SRdef}
\end{align}

The Einstein--Vlasov system in spherical symmetry with isotropic coordinate system was first studied by Rendall \cite{R96}. 
There, Rendall proved that the `first' singularity (if exists) should occur at the center, and the global existence of small data is true in this regime. 
(See also the survey paper \cite{And11}.)

\subsection{Gauge choice}\label{subsec:Gauge}
The gauge that we will use for Einstein equation is given as follows: 
because we require our coordinate to make the outer region of the solution to be exactly Schwarzschild in the standard foliation (the reason will be explained later), this fixes the \(s\)-coordinate there.
Therefore, by declaring that we use the isotropic coordinate, the \(R\)-coordinate will be fixed as soon as we choose the \(s\)-foliation of the spacetime.
We set \(s = 0\) on the initial hypersurface, and choose the leaves to have the same mean curvature as the initial hypersurface.
To be precise, let \(\Psi(R)\) be the mean curvature of the initial hypersurface (note that this is determined by the initial data \(\ghat\) and \(\khat\)), then we require the mean curvature of the \(s\)-constant hypersurface to be \(\Psi(R)\) as well.
With this choice of gauge, we can finally write Einstein equation in the following form (See \cite{R96}).

The Einstein equations reduce to
\begin{align}
(R^2\sqrt{A}')' &= -\frac{A^{5/2}R^2}{8}\big(16\pi\r + \frac{3}{2}K^2 - K\Psi - \frac{1}{2}\Psi^2\big),
\label{eq:Gauss}\\
\a'' + \frac{2}{R}\a' + \frac{A'}{A}\a' &= \a A^2\big(4\pi\tr T + 8\pi\r + \frac{3}{2}K^2-K\Psi+\frac{1}{2}\Psi^2\big) + A^2\b\Psi', \label{eq:SumEE}\\
K' +  3\big(\frac{A'}{A} + \frac{1}{R}\big)K&= 8\pi A j + \frac{(AR\Psi)'}{AR}, \label{eq:Codazzi}\\
\b'-\frac{\b}{R} &= \frac{\a}{2}(3K-\Psi), \label{eq:MeanCurv}\\
\rd_s A &= -\a AK + (\b A)', \label{eq:Keq}\\
\rd_s K &= -\frac{\a''}{A^2} + \frac{\a'A'}{A^3}
+ \a\left[\frac{2A'^2}{A^4} -\frac{2A''}{A^3} - \frac{2A'}{RA^3} + 4\pi\tr T - 8\pi S_R + K\Psi\right] + \b K'. \label{eq:EE}
\end{align}
Here and for the following, \('\) denotes the derivative with respect to \(R\).
We use \(K\) to denote the component of the second fundamental form of the \(s\)-constant hypersurface in this coordinate system, given by 
\begin{equation}\label{Kdef}
K \coloneqq \frac{1}{A^2}g(\nb_{\rd_R}\rd_R, \vec{n}),
\end{equation}
where \(\vec{n}\) is the normal vector field to the \(s\)-constant hypersurface, and \(\nb\) is the Levi-Civita connection of \(g\).

The Vlasov equation is given by \(\Dfr f = 0\), where the transport operator \(\Dfr\) in \((s, R, w, \ell)\) coordinate can be written as
\begin{align}\label{eq:vlasov-rwl}
\Dfr \coloneqq \rd_s + \big(\frac{\a w}{AE}-\b\big)\rd_R
+ \big(\frac{\a}{A^3}\frac{\ell}{R^3E} + \a\frac{A'}{A^4}\frac{\ell}{R^2E}-\frac{\a'}{A}E+ \a Kw\big)\rd_w.
\end{align}
The quantity that we need to use for the definition of dynamical action angle variable is \((H, M)\), which are given by 
\begin{equation}\label{Hdef}
H \coloneqq \a E-\b A w, \qquad M = \ell,
\end{equation}
which is given by the inner product of the velocity of the geodesic with \(\rd_s\) and \(\rd_\phi\), respectively.

By choosing the isotropic coordinate, we can make the spatial part of Einstein equation on \(A\) \eqref{eq:Gauss}, as well as the equation on lapse \(\a\) \eqref{eq:SumEE}, to be elliptic.
We gain two derivatives whenever we try to bound the metric coefficients in terms of the energy-momentum tensor components, which is crucial for the nonlinear estimates.
In this context, it is more convenient to understand metric coefficients to be a solution of a system of ODEs \eqref{eq:Gauss}--\eqref{eq:MeanCurv}, but of course then we need to understand the boundary condition that we have. 

Because of the finite speed of propagation, the mass is supported in a compact region for any finite time.
Hence, the metric outside this region should be exactly Schwarzschild by Birkhoff's theorem, and the Hawking mass should be constant in this region.
We choose our foliation in the way that the metric coefficients \(\a, A, \b\) are exactly the same as those of Schwarzschild in the standard foliation:
\begin{equation}\label{bdrycond}
A(s, R) = \Big(1+\frac{m}{2R}\Big)^2, \qquad
\a(s, R) = \frac{2R-m}{2R+m}, \qquad
\b(s, R) = 0,    
\end{equation}
for \(R\) outside the mass support. 
This serves us a boundary condition. 
We will treat the equations \eqref{eq:Gauss}--\eqref{eq:MeanCurv} in an abstract manner as follows: 
\[
\vec{X} = (A, A', \a, \a', K, \b) \in \RR^6, \qquad 
\vec{S} = (\r, \tr T, j, S_R, \Psi-\Psi_0) \in \RR^5,
\]
\begin{align}
\vec{X}'(R) &= \vec{F}(\vec{X}(R), \vec{S}(R), R), \label{absODE} \\
\text{ where } 
\vec{F}(\vec{X}(R), \vec{S}(R), R) &= \begin{pmatrix}
A'(R) \\
F_2(\vec{X}(R), \vec{S}(R), R) \\
\a'(R) \\
F_4(\vec{X}(R), \vec{S}(R), R) \\
F_5(\vec{X}(R), \vec{S}(R), R) \\
F_6(\vec{X}(R), \vec{S}(R), R)
\end{pmatrix}
\end{align}
Note that this system is written in a fixed time \(s\), and \(\vec{F}\) is a smooth function in our domain, as long as \(A\) is away from \(0\). 
This should be understood as an ODE system with a source \(\vec{S}\), and we will analyze this system around the solution \(\vec{X}_0 = (A_0, A_0', \a_0, \a_0', K_0, \b_0)\) with \(\vec{S} = 0\), which is the metric coefficients of the hypersurface we chose in Schwarzschild spacetime.
The function \(\Psi_0\) in the definition of \(\vec{S}\) above is the mean curvature of this hypersurface, following this perspective. 
We need a stability of this ODE system to i) obtain the local existence result and ii) improve the bootstrap assumption for the nonlinear system; this will be discussed at each place. 

\subsection{The object we perturb}\label{subsec:inithypsurf}
The object that we perturb is a hypersurface of Schwarzschild.
The metric of Schwarzschild black hole with mass \(m\) is given as follows: 
\[
g_{\Sch} = -\big(1-\frac{2m}{r}\big)dt^2 + \big(1-\frac{2m}{r}\big)^{-1}dr^2 + r^2d\O^2.
\]
To extend this beyond the event horizon, we use the \((v, r)\) coordinate, which is defined by 
\[
v \coloneqq t + r^* = t + r + 2m\log(r-2m),
\]
in which the metric is written as 
\[
g_{\Sch} = -\big(1-\frac{2m}{r}\big)dv^2 + 2dv dr + r^2d\O^2.
\]
This metric is defined for \(v \in \RR\) and \(r > 0\), which includes the region inside the event horizon. 
Within this coordinate we can describe the hypersurface of Schwarzschild that we perturb, which is the zero level set of the following function: 
\[
s = v - r - 2m\chi(r)C - 2m(1-\chi(r))\log(r-2m),
\]
where \(\chi\) is a smooth cutoff function satisfying that 
\[
\chi(r) = \begin{cases}
    1 & \text{for } r \le \frac{5}{2}m, \\
    0 & \text{for } r \ge 3m,
\end{cases}
\]
and \(C = -\log(m/4)\).
With properly chosen \(\chi\), the hypersurface defined by \(s = 0\) is spacelike, matches with the \(t = 0\) hypersurface for \(r \ge 3m\), and touches the event horizon at \(r = 2m\). 
We will perturb this hypersurface and use it as the initial hypersurface for the Einstein--Vlasov system. 

We know that the Cauchy problem for the Einstein--Vlasov system is well-posed for the initial data given as a Riemannian manifold as in this case, in the following sense.
The celebrated result of Choquet-Bruhat \cite{CB71} states that given a Riemannian manifold \((\Sigma, \ghat)\) with a symmetric 2-tensor \(\khat\) and a distribution function \(f\) on the mass shell over \(\Sigma\), there exists a unique (up to diffeomorphism) maximal globally hyperbolic development (MGHD) of this initial data, which is a Lorentzian manifold \((\Mcal, g)\) with a distribution function \(f\) on the mass shell over \(\Mcal\), such that the Einstein--Vlasov system is satisfied and the initial data is recovered on \(\Sigma\).
This was based on the previous result on the vacuum Einstein equation \cite{CB52, CBG69}.
See also the textbook \cite{Rin13} for a detailed discussion on this topic.

Because we are using the isotropic coordinate system with \(s = 0\) on the initial hypersurface, we need to write the Schwarzschild metric in this coordinate system as well.
The metric coefficients are given as a solution of ODE (involving the cutoff function \(\chi\)), but for the exterior region \(r \ge 3m\) we can write them explicitly as in \eqref{sch-metric}.
Note further that there for the second fundamental form we have 
\[
K_0 = 0, \qquad \Psi_0 = 0.
\]

\subsection{Linearized Einstein--Vlasov system on Schwarzschild background}
In this paper, we examine the stability of the Schwarzschild spacetime in the Einstein--Vlasov system.
We can first consider the linearized problem, namely proving decay of the solutions to the linearized Einstein--Vlasov system on Schwarzschild background.
Here, the linearization is done in the following way: we consider a one-parameter family \(\{\Mcal_\n, g_\n, f_\n\}_{\n \in \RR}\) of solutions to the (nonlinear) Einstein--Vlasov system, where \((\Mcal_0, g_0)\) is the subset of Schwarzschild spacetime (which is the future of the hypersurface described in subsection \ref{subsec:inithypsurf}), and \(f_0 \equiv 0\).
Each of them has a coordinate system that is described in subsection \ref{subsec:Gauge}, on which the metric and the energy-momentum tensor component satisfies \eqref{eq:Gauss}--\eqref{eq:EE} and \eqref{eq:vlasov-rwl}.
In this setting, the linearized system is obtained by taking the \(\n\) derivative of each equation. 

The linearized system is upper triangular due to the structure of the Vlasov equation. 
Because the distribution function \(f\) is perturbed around zero, the linearized Vlasov equation collapses to the linear transport equation of which coefficients are written in the metric component of Schwarzschild spacetime. 
This means, the distribution function \(f\) is transported by the geodesic flow on the Schwarzschild spacetime.
Therefore, we can solve the linearized Vlasov equation first, calculate the energy-momentum tensor component \(\r, \tr T, j, S_R\), and then use it as the source term of the linearized Einstein equation. 

Under the Schwarzschild background \eqref{sch-metric}, the transport operator \(\Dfr_0^{\lin}\) can be written in \((s, R, w, \ell)\) coordinate as follows:
\begin{equation}\label{eq:vlasov-lin}
\Dfr^\lin_0 \coloneqq \rd_s + \frac{\a_0 w}{A_0E}\rd_R
+ \big(\frac{\a_0}{A_0^3}\frac{\ell}{R^3E} + \a_0\frac{A_0'}{A_0^4}\frac{\ell}{R^2E}-\frac{\a_0'}{A_0}E\big)\rd_w,
\end{equation}
where \(E\) here is given by \(E^2 = 1 + w^2 + \frac{\ell}{A_0^2R^2}\).
The density variables \(\r, \tr T, j, S_R\) are defined by the same formulas as above, but with the metric coefficients replaced by those of Schwarzschild.

For the linearized Einstein equation, it is much simpler to consider the abstract form \eqref{absODE}. 
The linearized solution \((\a_\lin, A_\lin, \b_\lin)\) of the Einstein equation is given as the solution of the following linear ODE.
\begin{equation}\label{eq:linODE}
\begin{aligned}
\vec{X}_\lin'(R) = \vec{F}_{\vec{X}}(\vec{X}_0(R), 0, R)\vec{X}_\lin(R)
+ \vec{F}_{\vec{S}}(\vec{X}_0(R), 0, R)\vec{S}(R), \\
\text{ where } 
\vec{X}_\lin = (A_\lin, A'_\lin, \a_\lin, \a_\lin', K_\lin, \b_\lin), \\
\text{ with the boundary condition }
\vec{X}(R) \equiv \vec{X}_0(R) \text{ for large } R.
\end{aligned}
\end{equation}

We will show that the density variables \(\r, \tr T, j, S_R\) show phase mixing phenomenon by solving the Vlasov equation explicitly and applying the non-stationary phase argument, and then show that the metric components also have decaying time derivatives by analyzing \eqref{eq:linODE}.

We will use the coordinate system defined by the action angle variables on the mass shell. 
(See \cite{RS18, RS20, RS23} for the relevant result.)
This uses two conserved quantities for the transport, which are given by 
\begin{equation}\label{H0def}
H_0 \coloneqq \a_0 E, \quad M \coloneqq \ell. 
\end{equation} 
Notice that, contrary to \(H\) defined in \eqref{Hdef}, \(H_0\) is indeed a conserved quantity on geodesics. 
To keep the orbits to be bounded in \(R\), we need to make a support condition on the initial data written in terms of \(H_0\) and \(M\). 
We want the initial data to be supported in the set \(\Scal_{\cfr}\), for a given small constant \(\cfr > 0\), which is defined by 
\begin{multline*}
\Scal_{\cfr} \coloneqq
\bigg\{(R, w, \ell) \bigg|\, R \ge \big(\frac{3}{2}+\sqrt{2}\big)m, \ell > 16m^2, \\
V_{\min}(\ell)+\cfr < H_0^2(R, w, \ell) < 1-\cfr, 
(H_0(R, w, \ell), M=\ell)\in \Sfr_{\cfr}\bigg\},
\end{multline*}
where \(V_{\min}(\ell)\) is the minimum of the potential curve defined by \(V_0(R, \ell) = \a_0(R)^2\left(1+\frac{\ell}{A_0(R)^2R^2}\right)\), for fixed \(\ell\). 

The qualitative property of \(V_0\) becomes much clearer when we use the areal radius \(r\) instead of \(R\). 
Note that for the Schwarzschild metric we have \(r = R(1+m/2R)^2\).
Then, \(V_0\) can be written as 
\[
V_0(R, \ell) = 1-\frac{2m}{r} + \frac{\ell}{r^2} - \frac{2m\ell}{r^3},
\]
hence at \(R\) which gives the minimum of this we have 
\[
r = \frac{\ell + \sqrt{\ell^2-12m^2\ell}}{2m}.
\]
From the explicit computation, we can deduce that \(V_{\min}(\ell)\) increases to \(1\) on \([16m^2, \infty)\), and \(R\) which gives the minimum of \(V_0\) increases to \(\infty\) as \(\ell\to \infty\).
In particular, if we restrict \(V_0(R, \ell) < 1-\cfr\) for some \(\cfr > 0\), then \(\ell\) is bounded above, hence \(\Scal_\cfr\) has to be a compact set. 

The set \(\Sfr_{\cfr}\) deserves further explanation. 
We first introduce the set of \((H_0, M)\) for which the corresponding characteristic is bounded.
This is given as follows: 
\[
\Sfr_{\brm} = \left\{(H_0, M) \middle| M > 16m^2, \, V_{\min}(M) \le H_0^2 < 1\right\}.
\]
This is a sufficient condition that there is a characteristic bounded in \(R\) with this \((H_0, M)\) values.
Note that when \(H_0^2 = V_{\min}(M)\), the characteristic is constant in \(R\), but throughout this paper we choose strictly positive \(\cfr\), hence such characteristics will not be considered. 
On this set, we will define \(\O_\Sch\) by \eqref{OSchdef}, which is inversely proportional to the period of the characteristic described in the phase space \((R, w)\). 
Then, \(\Sfr_{\cfr}\) is the subset of \(\Sfr_{\brm}\) given as follows: 
\begin{multline*}
\Sfr_{\cfr} \coloneqq \big\{(H_0, M)\in \Sfr_{\brm} \big| (\rd_{H_0}\O_\Sch)(H_0, M) < -\cfr, \rd_\ell (\O_\Sch(H_0(R, w, \ell), M = \ell)) < -\cfr, \\
\forall (R, w, \ell) \text{ with } H_0 = H_0(R, w, \ell), M = \ell, \text{ and } R \ge (3/2+\sqrt{2})m\big\}.
\end{multline*}
These are the conditions that ensure the phase mixes. 
Note that the second condition refers to the \(\rd_\ell\) derivative; here we regard \(\O_\Sch\) as a function of \((R, w, \ell)\) through \((H_0, M)\).
There are inclusions \(\Sfr_{\cfr} \subset \Sfr_{\brm}\) for all \(\cfr > 0\), and \(\Sfr_{\cfr_1} \subset \Sfr_{\cfr_2}\) if \(\cfr_1 > \cfr_2 > 0\), and \(\Sfr_0 = \bigcup_{\cfr > 0} \Sfr_{\cfr}\).

In Proposition \ref{prop:monotonic}, we prove that \(\Sfr_0\) contains a large subset.
In particular, as long as \(M\) is large enough, or \(H_0^2\) is close enough to \(V_{\min}(M)\), we have \((H_0, M)\in \Sfr_0\). 
This proves that the set \(\Sfr_{\brm}\setminus\Sfr_0\) is compact for small \(\cfr\).
As the set \(\Scal_{\cfr}\) has compact range in \((H_0, M)\) for any \(\cfr > 0\) (of course it is also compact in \((R, w, \ell)\)), and as \(\Sfr_0\) is an increasing union of \(\Sfr_{\cfr}\) with \(\cfr\to 0\), we can choose \(\cfr > 0\) to be small to include the compactly supported initial data that satisfies the condition of Proposition \ref{prop:monotonic}.
Still, it remains open whether the set \(\Sfr_{\brm}-\Sfr_0\) is empty or not.
This distinction between \(\Sfr_{\brm}\) and \(\Sfr_\cfr\) is a new feature of this problem compared to \cite{CL24}, where the explicit formula for the period yields monotonicity for all bounded characteristics.

\begin{thm}[]\label{thm:lin}
    Let \(m, \cfr\) be given positive constants. 
    Let \(f\colon [0, \infty)_s\times [m/4, \infty)_R \times \RR_w \times \RR_\ell^{>0}\to \RR^{\ge 0}\) and \(\a_\lin, A_\lin, \b_\lin, K_\lin\colon [0, \infty)_s\times [m/4, \infty)_R\to \RR\) be a solution to the linearized Einstein--Vlasov system on Schwarzschild background with initial data \(f_\init\) satisfying \(\supp(f_\init)\subset \Scal_{\cfr}\).
    Then, the corresponding \(\r, \tr T, j, S_R\) defined by \eqref{rhodef}--\eqref{SRdef} with metric coefficients from the Schwarzschild metric have decaying time derivative, of which decay rate is inverse polynomial.
    More precisely, say initial data has the size
    \begin{equation}\label{linthm-init}
    \sum_{i_1+i_2+i_3 \le N} \sup_{R, w, \ell} \abs{\rd_R^{i_1} \rd_w^{i_2} \rd_\ell^{i_3} f_\init(R, w, \ell)} \le B_\init,
    \end{equation}
    for some \(N\in \NN\) and \(B_\init > 0\), then the solution exists globally in time, with support always in \(\Scal_{\cfr}\), and satisfies the following decay estimates:
    \begin{enumerate}
        \item (Controlled growth of \(f\) derivatives)
        \begin{equation}\label{linthm-f1}
        \sup_{R, w, \ell} \abs{\rd_R^{i_1} \rd_w^{i_2} \rd_\ell^{i_3} f(s, R, w, \ell)} \le C B_\init \jap{s}^{i_1 + i_2 + i_3}, \qquad i_1+i_2+i_3 \le N.
        \end{equation}
        \item (Improved bound of \(f\) derivatives in action angle variables) There is a coordinate system \((t = s, Q, H_0, M = \ell)\) defined by the action angle variables together with a fixed vector field \(Z\) given as \eqref{Zdef} such that  
        \begin{equation}\label{linthm-f2}
        \sup_{Q, H_0, M} \abs{\rd_Q^{i_1} \rd_{H_0}^{i_2} Z^{i_3} f(t, Q, H_0, M)} \le C B_\init\jt^{i_2}, \qquad i_1+i_2+i_3 \le N.
        \end{equation}
        \item (Decay of \(\rd_s\r, \rd_s\tr T, \rd_s j, \rd_s S_R\)) The time derivatives of the energy-momentum tensor components decay in time polynomially, more precisely,
        \begin{equation}\label{linthm-density}
        \sup_{R} \abs{\rd_R^i \rd_s (\r, \tr T, j, S_R)(s, R)} \le C B_\init \jap{s}^{-N + i}, \qquad i \le N-2.
        \end{equation}
        \item (Decay of \(\rd_s\a_\lin, \rd_s A_\lin, \rd_s \b_\lin\)) The time derivatives of the metric components decay in time polynomially, more precisely,
        \begin{equation}\label{linthm-metric}
        \sup_{R} \abs{\rd_R^i \rd_s (\a_\lin, A_\lin, \b_\lin)(s, R)} \le C B_\init \jap{s}^{-N + \max{i-2, 0}}, \qquad i \le N.
        \end{equation}
        \item (Convergence of the linearized solution) The energy-momentum tensor and metric components converge to a limit as \(s\to \infty\) with inverse polynomial rate.
        More precisely, there exist \((\widetilde{\r}, \widetilde{\tr T}, \widetilde j, \widetilde{S_R})\) and \((\widetilde{\a}_\lin, \widetilde{A}_\lin, \widetilde{\b}_\lin)\) such that
        \begin{equation}\label{linthm-conv}
        \begin{aligned}
        \sup_{R} \abs{\rd_R^i [(\r, \tr T, j, S_R)(s, R) - (\widetilde{\r}, \widetilde{\tr T}, \widetilde j, \widetilde{S_R})(R)]} &\le C B_\init \jap{s}^{-N+i}, && i \le N-1, \\
        \sup_{R} \abs{\rd_R^i [(\a_\lin, A_\lin, \b_\lin)(s, R) - (\widetilde{\a}_\lin, \widetilde{A}_\lin, \widetilde{\b}_\lin)(R)]} &\le C B_\init \jap{s}^{-N+\max{i-2, 0}}, && i \le N+1.
        \end{aligned}
        \end{equation}
    \end{enumerate}
    Here, the constant \(C\) depends on \(m, \cfr, N\), but not on \(B_\init\) and \(t\) or \(s\).
\end{thm}
\begin{rmk}[The vector field \(Z\)]
    We define \(Z\) to be a linear combination of \(\rd_{H_0}\) and \(\rd_M\) as follows: 
    \[
    Z \coloneqq (\rd_M \O_{\Sch})(H_0, M)\rd_{H_0} - (\rd_{H_0} \O_{\Sch})(H_0, M)\rd_M,
    \]
    where \(\O_{\Sch}(H_0, M)\) is a function that will be introduced later. 
    In particular, \(Z\) commutes with \(\rd_{Q_T}\).
    Using this vector field is necessary to obtain the improved bound of \(f\) as this commutes with the linearized transport operator \(\Dfr^\lin_0\). 
    With the coordinate vectors, the best estimate on \(f\) we have is the following: 
    \[
    \sup_{Q, H_0, M} \abs{\rd_Q^{i_1} \rd_{H_0}^{i_2} \rd_M^{i_3} f(t, Q, H_0, M)} \le C B_\init\jt^{i_2+i_3}, \qquad i_1+i_2+i_3 \le N.
    \]
    Both \(\rd_{H_0}\) and \(\rd_M\) lose one \(\jap{t}\), but \(Z\) does not cause additional growth. 
\end{rmk}

The main point of the proof of this theorem is checking the condition for the non-stationary phase argument, which gives the decay of macroscopic quantities.
We first define the action angle variables and write the Vlasov equation in this coordinate system, where the solution can be written explicitly.
In particular, through the Fourier expansion we see that the frequency mode \(e^{ikQ}\) of \(f\) has a phase rotation \(e^{-ikt\O_{\Sch}(H_0, M)}\).
Therefore, for the non-stationary phase argument on integrals defining \(\r, \tr T, j, S_R\), it is crucial to show that the \(\rd_\ell\) derivative (which we integrate by parts) of \(\O_{\Sch}(H_0, M)\) has a sign. 
This condition, on the large class of initial data, will be verified in Proposition~\ref{prop:monotonic}, which is a key new ingredient of the linear problem.

\subsection{Nonlinear stability statement}
For the nonlinear problem, Einstein equation is coupled with the Vlasov equation; Vlasov equation has metric-dependent coefficient, and this metric is given as a solution of Einstein Equation. 
The energy-momentum tensor in Einstein equation is given by integrating the distribution function \(f\) along each fiber. 
Therefore, the metric and the distribution function are coupled together in a nonlinear way.

We first describe the class of the initial data that we are considering for the nonlinear problem. 
The data is given on the manifold \(\Sigma_\init = [m/4, \infty)_R\times \SS^2\), and the subset of mass shell, which is the union of fibers at the points in \(\Sigma_\init\).
We consider the metric \(\ghat = A_\init^2(R)[dR^2 + R^2d\O^2]\) on \(\Sigma_\init\) together with the second fundamental form \(\khat\), of which representation in the coordinate \((R, \th, \phi)\) is given by 
\[
\khat_{RR} = A_\init^2K_\init, \qquad
\khat_{\th\th} = -\frac{1}{2}A_\init^2R^2(K_\init-\Psi), \qquad
\khat_{\phi\phi} = -\frac{1}{2}A_\init^2R^2\sin^2\th(K_\init-\Psi). 
\]
Note that all of these quantities are functions of \(R\) only. 
We want our perturbation to be compactly supported, in particular \(A_\init\), \(K_\init\), and \(\Psi\) match with those for the hypersurface in Schwarzschild that we are considering on the complement of a compact set. 
To be explicit, we want 
\[
A_\init(R) = (1+\frac{m}{2R})^2, \qquad
K_\init(R) = 0, \qquad
\Psi(R) = 0
\]
for arbitrarily large \(R\).
For \(f_\init\), we want it to be supported on \(\Scal_{\frac{1}{2}\cfr}\) (allowing a room), just as in the linear case; this guarantees that (at least within the timescale that we consider) the characteristic curve stays in a compact region. 
In particular, \(f_\init\) is nonzero only for \(R\) in a bounded interval, with \(R > \big(\frac{3}{2}+\sqrt{2}\big)m\), hence the energy-momentum tensor is supported in this region as well.
Still, our initial data should satisfy the constraint equations, which is given as follows:
\begin{align}
(R^2\sqrt{A_\init}')' &= -\frac{A_\init^{5/2}R^2}{8}\big(16\pi\r + \frac{3}{2}K_\init^2 - K_\init\Psi - \frac{1}{2}\Psi^2\big), \label{eq:GaussInit}\\
(A_\init^3R^3K_\init)' &= 8\pi A_\init^4R^3 j + A_\init^2R^2(A_\init R\Psi)'. \label{eq:CodazziInit}
\end{align}
In particular, the initial data that we consider should satisfy certain smallness, 
i.e. \(A_\init(R)-A_0(R)\), \(K_\init(R)-K_0(R)\), and \(\Psi(R)-\Psi_0(R)\) are small, and \(f_\init\) in \(L^\infty\) norm is small. 

The local existence result for this system is given as follows: 
\begin{prop}[Local Existence]\label{prop:LWP}
    Let \(m, \lfr\) and \(m/4 < R_a < R_b < R_+\) be arbitrary positive numbers. 
    There are positive constants \(\e_0 > 0\) and \(C_0 > 0\) such that the following statement is true: Say the initial data \(A_\init, K_\init, \Psi\) and \(f_\init\) with the support condition 
    \[
    \supp(A_\init - A_0, K_\init - K_0, \Psi - \Psi_0) \subset [m/4, R_+], \qquad \supp(f_\init)\subset [R_a, R_b]\times [-\lfr, \lfr]\times (0, \lfr]
    \]
    and the smallness condition 
    \[
    \norm{A_\init-A_0}_{C^2} + \norm{K_\init-K_0}_{C^1} + \norm{\Psi-\Psi_0}_{C^1} + \norm{f_\init}_{C^0} \le \e < \e_0,
    \]
    satisfies the constraint equation \eqref{eq:GaussInit}--\eqref{eq:CodazziInit}.
    Then, there is a time \(T>0\) (that depends on the initial data only through \(m, \lfr, R_a, R_b, R_+\) and \(\e\)) such that the solution exists up to time \(T\), i.e. there is a function 
    \[
    \a, A, \b, K\colon [0, T)_s\times [m/4, \infty)_R\to \RR
    \]
    and 
    \[
    f \colon [0, T)_s\times [m/4, \infty)_R\times \RR_w\times \RR^{>0}_\ell\to \RR^{\ge 0} 
    \]
    which satisfy the equations \eqref{eq:Gauss}--\eqref{eq:EE} and \eqref{eq:vlasov-rwl}, and achieve the boundary condition \eqref{bdrycond}.
    This solution satisfies the following estimates:
    \begin{equation}\label{LWP-est}
        \sup_{s\in [0, T]}\big(\norm{\a-\a_0}_{C^2} + \norm{A-A_0}_{C^2} + \norm{\b-\b_0}_{C^2} + \norm{K-K_0}_{C^1} + \norm{f}_{C^0}\big) \le C_0 \e,
    \end{equation}
    and it also satisfies the following support condition: 
    \begin{equation}\label{LWP-support}
        \begin{cases}
        &\supp(\a(s)-\a_0, A(s)-A_0, \b(s)-\b_0, K(s)-K_0) \subset [m/4, R_+], \\
        &\supp(f(s)) \subset [(R_a + m/4)/2, (R_b + R_+)/2]\times [-\lfr-1, \lfr+1]\times (0, \lfr],
        \end{cases}
    \qquad s\in [0, T].
    \end{equation}
    Moreover, regularity is propagated on the same time interval: for every integer \(k\ge 1\) there is \(C_k > 0\) (depending in addition on \(k\)) such that if
    \[
    \norm{A_\init-A_0}_{C^{k+2}} + \norm{K_\init-K_0}_{C^{k+1}} + \norm{\Psi-\Psi_0}_{C^{k+1}} + \sum_{\abs{I}\le k}\norm{\rd^I f_\init}_{C^0} \le \e,
    \]
    then
    \begin{equation}\label{rLWP-est-high}
        \sup_{s\in[0,T]}\Big( \norm{\a-\a_0}_{C^{k+2}} + \norm{A-A_0}_{C^{k+2}} + \norm{\b-\b_0}_{C^{k+2}} + \norm{K-K_0}_{C^{k+1}} + \sum_{\abs{I}\le k}\norm{\rd^I f}_{C^0}\Big) \le C_k\,\e.
    \end{equation}
\end{prop}
The proof of this proposition is given in Subsection \ref{subsec:LWP}.

Now we can state the main result of this paper. 
\begin{thm}[]\label{thm:nonlin}
    Consider the Cauchy problem of Einstein--Vlasov system \eqref{eq:g}--\eqref{eq:f}. Then, for any given \(m, \cfr > 0\), and \(N\ge 6\) there are \(\e_0 > 0\) and \(\d > 0\), (which depend only on \(m\), \(\cfr\), and \(N\)) such that the following stability and phase mixing statement is true:
    Let a perturbation of Schwarzschild \((\Sigma_\init, \ghat, \khat, f_\init)\) be given which,
    \begin{enumerate}
        \item satisfies the constraint equation \eqref{eq:GaussInit}--\eqref{eq:CodazziInit},
        \item matches with Schwarzschild with mass \(m\) in standard foliation for large enough \(R\), which means 
        \[
        A_\init(R) = \big(1+\frac{m}{2R}\big)^2, \quad
        K_\init(R) = 0, \quad
        \Psi(R) = 0, \quad
        f_\init(R, w, \ell) = 0.
        \]
        \item satisfies the support condition \(\supp(f_\init)\subset \Scal_{\cfr}\),
        \item has a smallness condition 
    \begin{equation}\label{nonlinthm-initAK}
    \sum_{I \le N+3} \sup_{R} \abs{\rd_R^I (A_\init-A_0)(R)} 
    + \sum_{I \le N+2} \left[\sup_{R} \abs{\rd_R^I (K_\init-K_0)(R)} 
    + \sup_{R} \abs{\rd_R^I (\Psi-\Psi_0)(R)}\right]
    \le \d\e,
    \end{equation}
    \begin{equation}\label{nonlinthm-initf}
    \sum_{i_1+i_2+i_3 \le N+1} \sup_{R, w, \ell} \abs{\rd_R^{i_1} \rd_w^{i_2} \rd_\ell^{i_3} f_\init(R, w, \ell)} \le \d\e,
    \end{equation}
    for \(0 < \e < \e_0\).
    \end{enumerate}
    Then, there is a spacetime with distribution function on its mass shell, \((\Mcal, g, f)\), with \(\Mcal \cong \cup_{s\in [0, T_f]}\Sigma_s\) where \(T_f \coloneqq \e^{-1}(\log 1/\e)^{-2}\) and \(\Sigma_s\) is a diffeomorphic copy of \(\Sigma_\init\), which can be embedded in the MGHD of the initial data, and the following statements on this spacetime are true.
    \begin{enumerate}
        \item (Controlled growth of \(f\) derivatives) The solution \(f\) satisfies the following growth estimates for its derivatives:
        \begin{equation}\label{nonlinthm-f1}
        \sup_{R, w, \ell} \abs{\rd_R^{i_1} \rd_w^{i_2} \rd_\ell^{i_3} f(s, R, w, \ell)} \le C\d\e \jap{s}^{i_1 + i_2 + i_3}, \qquad i_1+i_2+i_3 \le N, \quad s\le T_f.
        \end{equation}
        \item (Improved bound of \(f\) derivatives in dynamical action angle variables) There is a coordinate system \((t = s, Q_{T_f}, H, M = \ell)\) defined \textbf{dynamically} with \(H\) given by \eqref{Hdef} together with a vector field \(Z\) given by Definition \ref{def:Zdef} such that
        \begin{equation}\label{nonlinthm-f2}
        \sup_{Q_{T_f}, H, M} \abs{\rd_{Q_{T_f}}^{i_1} \rd_H^{i_2} Z^{i_3} f(t, Q_{T_f}, H, M)} \le C\d\e\jt^{i_2}, \qquad i_1+i_2+i_3 \le N, \quad t\le T_f.
        \end{equation}
        \item (Decay of \(\rd_s\r, \rd_s\tr T, \rd_s j, \rd_s S_R\)) The time derivatives of the energy-momentum tensor components decay in time polynomially, more precisely,
        \begin{equation}\label{nonlinthm-density}
        \sup_{R} \abs{\rd_R^i \rd_s (\r, \tr T, j, S_R)(s, R)} \le C\d\e \jap{s}^{-N + i}, \qquad i \le N-2, \quad s\le T_f.
        \end{equation}
        \item (Decay of \(\rd_s \a, \rd_s A, \rd_s \b\)) The time derivatives of the metric components decay in time polynomially, more precisely,
        \begin{equation}\label{nonlinthm-metric}
        \sup_{R} \abs{\rd_R^i \rd_s (\a, A, \b)(s, R)} \le C\d\e \jap{s}^{-N + \max{i-2, 0}}, \qquad i \le N, \quad s\le T_f.
        \end{equation}
        \item (Metric stays near Schwarzschild) The metric components \(\a, A, \b\) are exactly equal to those of Schwarzschild for \(R \ge R_{\textrm{out}}\), where \(R_{\textrm{out}}\) is a constant that depends only on \(m, \cfr, N\), and remain close to their Schwarzschild values for all \(s \le T_f\), more precisely,
        \begin{equation}\label{nonlinthm-metric2}
        \sup_{R} \abs{\rd_R^i (\a-\a_0, A-A_0, \b-\b_0)(s, R)} \le C\d\e, \qquad i \le N+2, \quad s\le T_f.
        \end{equation}
    \end{enumerate}
    Here, the constant \(C\) depends on \(m, \cfr, N\), but not on \(\e\), \(\d\), \(t\) or \(s\).
\end{thm}

\begin{rmk}[The `dynamical' action angle variable]
    The coordinate \(Q_T(s, R, w, \ell)\) refers to the solution at time \(T\).
    This is the reason why we call this coordinate system `dynamical', because the definition of the coordinate depends on the solution itself.
    Here, time \(T\) refers to a later time that we use for the bootstrap argument; we bootstrap the fact that the difference of solution at time \(s\) and \(T\) with \(s < T\) decays in \(s\), and for that we use the coordinate system that depends on the solution at time \(T\). 
    Nevertheless, we use \(H(s, R, w, \ell)\) which depends only on the solution at time \(s\).
    This choice is to make the Vlasov equation \label{vlasov-DAAV} to have decaying coefficients. 
    This will be explained more in Remark \ref{rmk:coordchoice}.
\end{rmk}
\begin{rmk}[The vector field \(Z\)]
    For the nonlinear analysis we keep using the same formula for \(Z\) as in the linear case, but the definition of coordinates (and hence the coordinate vectors) were changed. 
    Still, \(Z\) commutes with \(\rd_{Q_T}\), and it depends on the solution only through the gauge choice (in the sense that \(\rd_H\) depends on the solution), meaning that the coefficients of \(\rd_H\) and \(\rd_M\) in the definition of \(Z\) are independent of the solution. 
    This will be explained in Definition \ref{def:Zdef} and the following discussion.
\end{rmk}
\begin{rmk}[The solution reaches the trapping region]
The solution described above touches the trapping region within our domain, i.e. in \(\S_s\) for each \(s\in [0, T_f]\). 
To be even more precise, for these \(s\) the point \((s, m/4)\) is trapping.
This is essentially because of (5) of the theorem, together with the fact that being trapping is a stable condition.
However, the theorem is insufficient to spot where the event horizon is, mainly because we do not have a control of the solution for \(s > T_f\).
\end{rmk}

\subsection{Ideas of proof}
\subsubsection{Action angle variables and linear phase mixing}
We restricted in the regime that the geodesics in the support of \(f\) are all bounded: in the linear system the conserved quantity \(H\) and \(M\) determine whether the geodesic is bounded or unbounded, and we restrict the support of the initial data to the bounded regime.
For fixed \(H\) and \(M\), we can analyze the geodesic as a solution for Hamiltonian ODE system. 
The quantity \(w\) is the radial velocity, and by fixing \(H\) and \(M\) the relation 
\[
H^2 = \a_0(R)^2\left(1 + w^2 + \frac{\ell}{A_0(R)^2R^2}\right)
\]
serves a conservation law in \((R, w)\) space.
In this context, the allowed range of \((H, M)\) forces this conservation law to have a bounded orbit in \((R, w)\) space.

Nevertheless, in contrast to the Kepler case \cite{CL24}, the trajectories are not closed; there is a precession of the orbit, which is a well-known fact in physics.
Therefore, it is impossible to define a global action angle variables on a subset of \(\RR_x^3\times \RR_v^3\). 
This is the reason why we need spherical symmetry even for the linear system; because the precession is only in the angular direction, in spherical symmetry the periodicity in \((R, w)\) space allows us to define global action angle variables on a subset of \(\RR_R\times \RR_w\times \RR_\ell\).
This can also be understood that we are considering thin mass shells at a fixed radius which shrink and dilate, of which motion is periodic in \((R, w)\) space.

In the action angle variable coordinates \((t, Q, H_0, M)\), the Vlasov equation is given as 
\[
\rd_t f + \O_\Sch(H_0, M)\rd_Q f = 0,
\]
where \(\O_\Sch(H_0, M)\) is given as a certain integral over the trajectory in \((R, w)\) space, with given parameter \(H_0, M\).
This is a transport equation with constant coefficient, and we can write the solution explicitly in terms of the initial data \(f_\init\) as 
\[
f(t, Q, H_0, M) = f_\init(Q - t\O_{\Sch}(H_0, M), H_0, M).
\]
This shows the controlled growth of the derivatives of \(f\) in \((Q, H_0, M)\) coordinates, meaning that \(\rd_Q\) derivative does not grow in time, and \(\rd_{H_0}\) and \(\rd_M\) derivatives grow at most linearly in time. 
Furthermore, the integrated quantity that we want to show the decay estimate can be written as 
\[
\rd_s I(s, R) = 
\frac{\pi}{R^2A_0(R)^2}\sum_{k\neq 0} \int_{-\infty}^{\infty}\int_0^\infty (-ik\O_{\Sch})e^{ikQ}e^{-iks\O_{\Sch}} \wh{(f_\init)}_k(H_0, M) \bar{\CC}_0(R, w, \ell) d\ell dw,
\]
where \(\wh{(f_\init)}_k\) is the Fourier coefficient of \(f_\init\) with respect to \(e^{ikQ}\), and \(\bar{\CC}_0(R, w, \ell)\) is a fixed function.
Therefore, the non-stationary phase argument on \(e^{-iks\O_{\Sch}}\) can be applied to obtain the desired decay estimate.

\subsubsection{Period function and its monotonicity}
To make the non-stationary phase argument work, we need to understand the quantity \(\O_\Sch(H_0, M)\), which is the frequency of the periodic motion in \((R, w)\) space.
This is given as a singular integral as follows: 
\begin{equation}\label{OSchdef}
\frac{\pi}{\O_{\Sch}(H_0, M)} = \int_{R_-^0(H_0, M)}^{R_+^0(H_0, M)} \frac{A_0(R)H_0}{\a_0(R)^2w(R, H_0, M)} dR.
\end{equation}
This is in stark contrast to the Kepler case \cite{CL24}, where the period function takes the simple explicit form that is proportional to \((-H_0)^{-3/2}\), depending only on the energy \(H_0\) and not on the angular momentum \(M\).
In particular, it is not at all trivial that the \(\ell\) derivative of \(\O_\Sch(H_0, M)\) is nonzero, which is crucial for the non-stationary phase argument.
Note that the derivative we take is in \(\rd_\ell\), not in \(\rd_M\), so we need to understand the derivative of \(\O_\Sch\) in both of its arguments \(H_0\) and \(M\).

We prove in Proposition~\ref{prop:monotonic} that \(\rd_\ell \O_\Sch(H_0, M)\) indeed has a definite sign throughout the admissible region. 
This contains the region that when the (normalized) angular momentum \(\ell/m^2\) is large enough, or eccentricity \(e\), which is the measure of how far the orbit is from being circular, is small enough.

\subsubsection{The gauge choice}
We need to choose a gauge to derive a PDE from the geometric formulation of Einstein's equation. 
In \cite{R96}, Rendall write the metric in isotropic coordinates and chose the maximal gauge, which makes each leaf of the foliation to have zero mean curvature. 
\[
g(s, R) = -\a^2(s, R)ds^2 + A^2(s, R)[(dR+\b(s, R)ds)^2 + R^2d\O^2]
\]
Then, the Einstein equation can be represented with a system of ODEs \eqref{eq:Gauss}--\eqref{eq:MeanCurv} and two PDEs \eqref{eq:Keq} and \eqref{eq:EE}.
We should understand \eqref{eq:EE} to have evolution information of the metric, \eqref{eq:MeanCurv} and \eqref{eq:Keq} to have the information of the foliation and the definition of \(K\), and \eqref{eq:Gauss}--\eqref{eq:Codazzi} to have the information of the constraint equations.

Because of the finite speed of propagation, as the matter is supported in a compact region for any finite time, the metric outside this region should be exactly Schwarzschild by Birkhoff's theorem, and the Hawking mass should be constant in this region.
Therefore, we choose our foliation in the way that the metric coefficient \(\a, A, \b\) are exactly the same as those of Schwarzschild in the standard foliation for \(R\) outside the matter support.
This allows us to understand the ODEs \eqref{eq:Gauss}--\eqref{eq:MeanCurv} as an ODE system with a source term given by the energy-momentum tensor.
In this way, instead considering the evolution of metric coefficient given by \eqref{eq:Keq} and \eqref{eq:EE}, we can understand all the evolution by Vlasov equation, and then calculate the metric coefficients as a solution with the ODE system.
Abstractly, we can write the ODE in the following form: 
\[
\vec{X}'(R) = \vec{F}(\vec{X}(R), \vec{S}(R), R)
\]
where \(\vec{X} = (A, A', \a, \a', K, \b) \in \RR^6\) and \(\vec{S}\) serves source term, contains the energy-momentum tensor coefficients.
With this ODE, we gain two derivatives on the metric coefficients \(A\) and \(\a\), one on \(\b\) and \(K\), relative to \(\vec{S}\), and one extra derivative of \(\b\) from the equation \eqref{eq:MeanCurv}.
This is the main mechanism of gaining two derivatives from the Einstein's equation. 

Since we are perturbing the Schwarzschild metric, we need to ensure that the matter support is contained in the future of the initial hypersurface.
The initial hypersurface should be chosen to ensure this.
We choose the initial hypersurface \(\Sigma_0\) to extend inside the black hole region, so that its causal future contains the entire matter support.
We then choose the foliation so that each leaf has constant mean curvature.
This ensures that we do not lose additional derivatives in the elliptic estimates, while leaving the Vlasov equation unchanged.

\subsubsection{Dynamical action angle variables}
We use the dynamical action angle variable coordinates \((t, Q_T, H, M)\) to write the Vlasov equation, following \cite{CL24}.
In this way, we can use the vector field method: from the equation 
\[
    \rd_t f + \O(H, M)\rd_{Q_T}f + \Pcal(t, T, Q_T, H, M)\rd_{Q_T} f + \Dfr H \rd_H f = 0
\]
(See \eqref{vlasov-DAAV}), we can predict that \(\rd_{Q_T}\) does not cause an additional \(\jap{t}\) loss, but \(\rd_H\) or \(\rd_M\) causes an additional \(\jap{t}\) loss.

The nonlinear terms in this equation should be controlled when we estimate the 
energy-momentum tensor.
There, even though there is a growth for the derivative of \(f\), the decay of multiplied factors \(\Pcal\) and \(\Dfr H\) allows us to control the nonlinear terms.
This allows us to prove the phase mixing phenomenon and the existence of the solution up to the desired timescale. 

The specific choice of \(Q_T\) (not the linear action angle \(Q\)) and \(H\) (not \(H_0\) or \(H_T\)) is important.
This is discussed in Remark \ref{rmk:coordchoice}.

\subsubsection{Vector field \(Z\)}
Comparing to \cite{CL24}, the nontrivial dependence of the function \(\O(H_T, M)\) in \(M\) makes \(\rd_M\) into a bad derivative, which is an undesirable situation. 
When we estimate the energy-momentum tensor, which is given by the integral of \(f\) over \(dwd\ell\), we need to use integration by parts to the bad derivatives on \(f\), to avoid picking up an additional \(\jap{t}\) growth.

To overcome this situation, we introduce a new vector field \(Z\) defined by \eqref{Zdef}, which is a linear combination of \(\rd_H\) and \(\rd_M\), commutes with the other good derivative \(\rd_{Q_T}\), and satisfying 
\[
Z(\O_{\Sch}(H, M)) = 0.
\]
Then, \(Z\) does not cause an additional \(\jap{t}\) growth when it hits \(f\), and this together with the vector identities \eqref{dHintoell} and \eqref{drintodell} we can efficiently estimate the integral of \(f\). 

\subsubsection{Main integral estimates}
The main part of the estimate is proving that the energy-momentum tensor components become stationary in time. 
Our strategy is given as follows: 
From Duhamel's principle, we can write the solution \(f\) of the Vlasov equation in the dynamical action angle variable coordinates as 
\[
f(t, Q_T, H, M) = f_\init(Q_T-t\O(H, M), H, M) + \int_0^t \Rcal (\t, T, Q_T-(t-\t)\O(H, M), H, M)d\t.
\]
Here, we write equation \eqref{vlasov-DAAV} in the following form:
\[
\rd_t f + \O(H, M)\rd_{Q_T} f = \Rcal(t, T, Q_T, H, M),
\]
where \(\Rcal\) represents nonlinear terms, which has a decaying factor in front of the derivative of \(f\); here it suffices to consider the \(\rd_s A \cdot \rd_H f\) term. 
Because we know that the non-stationary phase argument works for the linear system, it suffices to consider the nonlinear contribution.
This will be given as follows: 
\[
\qbf_N(s, R) = \frac{\pi}{R^2A^2}\int_{-\infty}^\infty\int_0^\infty\int_0^s \Rcal(\t, T, Q_T-(s-\t)\O(H, M), H, M) \bar{\CC} \, d\t d\ell dw,
\]
so the main term of \(\rd_s \qbf_N(s, R)\) is given as 
\[
-\int_{-\infty}^\infty\int_0^\infty\int_0^s 
\O(H, M)\rd_{Q_T}\Rcal(\t, T, Q_T-(s-\t)\O(H, M), H, M) \bar{\CC} \, d\t d\ell dw.
\]

At this point we can observe why we need to gain two derivatives from the ODE formulation of the Einstein equation: 
we need to estimate terms of the form 
\[
-\int_{-\infty}^\infty\int_0^\infty\int_0^s \rd_R\rd_s A \cdot \rd_H f d\ell dw,
\]
which is a component of \(\rd_s \rho\). 
Because \(\rd_H f\) has additional growth, we integrate by parts in \(\ell\) to avoid this growth, and we need to estimate the term
\[
-\int_{-\infty}^\infty\int_0^\infty\int_0^s \rd^2_R\rd_s A \cdot f d\ell dw.
\]
Therefore, to close the estimate we need to gain two derivatives from \(\rd_s A\) to \(\rd_s \rho\), which is exactly the point where we need to use the ODE. 

However, for higher order derivatives this naive approach does not help: when we differentiate in \(R\), we will pick up additional \(s\) from \((s-\t)\O\). 
To avoid this, we decompose the frequency components of \(\qbf_N\) in \(Q_T\), and this results in the following expansion of the main term: 
\[
\Tcal_{5, 3}(s, R) = -\sum_{k, l\in \ZZ} \int_{-\infty}^\infty\int_0^\infty\int_0^s 
ik\O e^{ikQ_T}\wh{(p_3)}_l \cdot e^{-ik(s-\t)\O}\wh{(\rd_H f)}_{k-l} \bar{\CC} \, d\t d\ell dw,
\]
where \(p_3\) is a time derivative of metric coefficients.
We use the fact that \(\rd_{Q_T}\) and \(Z\) commutes (or have a controllable commutator) with \(e^{-ik(s-\t)\O}\), and \(\rd_\ell\) can be integrated by parts. 
Because one more derivative is needed to sum up the frequency components, we are short in one derivative; for this reason we need a separate weak estimate for the top order.
This is given in \eqref{qtop} and \eqref{dsqtop}, which is tied to the bootstrap assumptions \eqref{BAH0} and \eqref{BAH1}. 

In the same reason, we can explain why we need the separate top-order estimates \eqref{ftop} on \(f\).
To get the control on \(N-2\)-th derivative of \(\qbf_N\) (which we still want decay), the \(N\)-th derivative of \(f\) appears when every derivative hits \(f\). 
We need one more derivative to sum up the frequency components, and this is the reason why we need to control the \(N+1\)-th derivative of \(f\) in \eqref{ftop}.
In the meantime, the \(N+1\)-th derivative of \(\rd_s A\) is needed from the Vlasov equation. 
To control this, we again need \(\rd_R^{I-1}\rd_s\) of the energy-momentum tensor, which explains why we need the separate top order estimates. 

\subsubsection{Vector field \(Y_H\)}
In this framework, we can also explain why we need the vector field \(Y_H\) defined in \eqref{YHdef}.
Controlling the derivative of \(\Tcal_{5, 3}\) above, we will see that the contribution of each frequency component is given by 
\[
\Ical_{k, l}^{i_1, i_2, i_3}(\t, s, R)
= \int_{-\infty}^\infty\int_0^\infty e^{ikQ_T}\wh{(\rd_R^{i_1}p_b)}_l \cdot \big(e^{-ik(s-\t)\O}\wh{(\rd_b \rd_{Q_T}^{i_2}Z^{i_3}f)}_{k-l}\big)\CC(s) \, d\ell dw.
\]
Clearly, the worst situation is \(i_2 = i_3 = 0\), \(i_1 = I+1\), where we need to show that the \(\t\)-integral of \(\Ical_{k, l}^{i_1, i_2, i_3}\) summed up in \(k, l\) has decay \(\jap{t}^{-N+I}\). 
We originally pay one more derivative of \(p\) to sum the Fourier modes, where we lose one \(\jap{t}\) and get one \(\jap{l}\).
Instead, we can use \(Y_{k, -l\t} = i(ks-l\t)(\rd_X \O)(H, M) + \rd_H\), which satisfies 
\[
\frac{Y_{k, -l\t} - \rd_H}{i(ks-l\t)(\rd_X \O)(H, M)} = 1, \qquad
Y_{k, -l\t}(e^{-ik(s-\t)\O}\ghat_{k-l}) = e^{-ik(s-\t)\O}\wh{(Y_H g)}_{k-l}.
\]
Hence, instead of obtaining summability in this way, we can use \(Y_{k, -l\t}\), which gives one \(\jap{ks-l\t}\) with one derivative on \(p\), or \(Y_H\) on \(f\).
Taking integral in \(\t\), \(\jap{ks-l\t}\) gives \(\log \jap{s}\) and in the meantime we get one \(\jap{l}\),  which helps us to sum up the Fourier modes.
The details will be explained in Lemma \ref{lem:T5comp} and \ref{lem:T5b}.

\subsubsection{Timescale \(\e^{-1}(\log 1/\e)^{-2}\)}
It is known that the solution evolving from the initial data with finite regularity has uncontrollable nonlinear terms \cite{Bed21}, which is called the \textit{nonlinear echo}.
There, it was proved that the nonlinear echo term cannot be controlled at the timescale \(\e^{-1}\), where \(\e\) is the size of the initial data.
Hence, it is natural to expect that the solution can be controlled only up to a finite timescale.

From the above argument, we can predict the following bound: 
\[
\abs{\Ical_{k, l}^{I+1, 0, 0}(\t)} \lesssim \d^{\frac{7}{4}}\e^2\jap{s}^{-N+I+1} \frac{a_l(\t)}{\jap{ks-l\t}}\frac{1}{\jap{k-l}^2}.
\]
Here, \(a_l(\t)\) is a sequence with \(\sum_l a_l^2(\t) \le 1\), but because we will integrate in \(\t\) first, we cannot utilize the summability, and are forced to use the crude estimate \(a_l\le 1\).
Therefore, we are still short in \(\jap{l}\) (we only have one from the time integral in \(\t\)) to obtain the summability in \(k\) and \(l\). 

To overcome this situation, we need to allow another log loss in time. 
By paying one more derivative in \(p\), we also have the following bound: 
\[
\abs{\Ical_{k, l}^{I+1, 0, 0}(\t)} \lesssim \d^{\frac{7}{4}}\e^2\jap{s}^{-N+I+2} \frac{1}{\jap{ks-l\t}}\frac{1}{\jap{l}}\frac{1}{\jap{k-l}^2}.
\]
Hence, we can sum them in the following way:
\[
\abs{\Ical_{k, l}^{I+1, 0, 0}(\t)} \lesssim \d^{7/4}\e^2\jap{s}^{-N+I + 1} \frac{\jap{l}}{\jap{ks-l\t}} \min\left(\frac{1}{\jap{l}}, \frac{\jap{s}}{\jap{l}^2}\right)\frac{1}{\jap{k-l}^2}.
\]
Integral of the factor \(\frac{\jap{l}}{\jap{ks-l\t}}\) in \(\t\) gives one \(\log \jap{s}\), and the summation of \(\min(1/\jap{l}, \jap{s}/\jap{l}^2)\) gives another \(\log \jap{s}\).
This is the reason why we need to restrict the timescale to \(s\le T_f = \e^{-1}(\log 1/\e)^{-2}\).

\subsection{Related works}
\subsubsection{Global results for the Einstein--Vlasov system}
The global property of the solution to the Einstein--Vlasov system was first extensively studied by Rein--Rendall \cite{RR92}, where they proved the global existence of the solution in spherical symmetry for small data in the \textit{massive case}.
Using isotropic coordinate system, Rendall \cite{R96} proved in spherical symmetry that the `first' singularity occurs at the center (if one exists), and the small data gives a globally well-posed solution. 
The stability result for small data in the \textit{massless case} was proved by Dafermos \cite{Daf06}. 
We refer the reader to the survey of Andréasson \cite{And11} for a comprehensive account of the system and of its known global results.

Beyond the small data regime, the global structure of the system is also understood in spherical symmetry: 
Dafermos--Rendall \cite{DaRe05} proved an extension principle asserting that the first singularity must emanate from the center, which places the system in the class of matter models covered by \cite{Daf05},
while in the cosmological setting strong cosmic censorship for surface-symmetric spacetimes with collisionless matter was proved in \cite{DR16}.
The black hole formation in the Einstein--Vlasov system was studied in \cite{AKR11, And12, AR25}.

Nonlinear stability of Minkowski spacetime for the Einstein--Vlasov system without symmetry was proved in various results: Lindblad--Taylor \cite{LT20} and Fajman--Joudioux--Smulevici \cite{FJS21} proved the stability result for the massive case, and the restriction on the initial data was removed in the recent result of Wang \cite{Wan22}.
Taylor \cite{Tay17} proved the stability result for small data in the massless case outside symmetry. 
This result was further generalized in \cite{BFJST21} to the class containing non-compactly supported initial data.

\subsubsection{Stability of Schwarzschild for vacuum Einstein equation}
The stability of black hole spacetimes in various settings has been extensively studied.
The monumental work by Christodoulou--Klainerman \cite{CK93} and Lindblad--Rodnianski \cite{LR10} proved the nonlinear stability of Minkowski spacetime, using the vector field method.

The stability of Schwarzschild spacetime in various settings was studied in recent years: For the vacuum case, there is a rigidity result called Birkhoff's theorem, which states that any spherically symmetric solution to the vacuum Einstein equation is isometric to a Schwarzschild spacetime.
Therefore, any result on Schwarzschild must be either coupled to matter, or without spherical symmetry.
For vacuum, hence without spherical symmetry case, Klainerman--Szeftel \cite{KS20} proved the stability under axially symmetric polarized perturbations.
Dafermos--Holzegel--Rodnianski--Taylor \cite{DHRT21} proved the stability of Schwarzschild family without any symmetry assumption. 
(See also \cite{DHR24, GKS24, She23, Hin26} for Kerr.)

\subsubsection{Matter fields on black hole spacetime}
Matter fields with their own governing equations on a black hole spacetime were studied in various works, mostly on a fixed background.
For the linear scalar wave equation, see \cite{DR09, DR13, Mos16, AAG18, OS20, Hin22, LO24, DSS11}, for the charged scalar field equation, see \cite{VdM22}, and for the Maxwell field, see \cite{Blu08, ST15, Pas19b, MTT17, ABB16}.
For \emph{massive} fields the mechanism is different, since timelike geodesics can be trapped outside the event horizon and the field does not disperse; see \cite{Kei16, PSRVdM23, SRVdM26}.
Nonlinear matter models on a fixed background were also studied: semilinear wave equations in \cite{BS06}, quasilinear wave equations in \cite{LTo18}, and the Maxwell--Born--Infeld system in \cite{Pas19a}.

For the Vlasov field, the decay estimate of the massless Vlasov equation was proved in \cite{ABJ18} on the Kerr background, and in \cite{Big23} on the Schwarzschild background.
In particular, Velozo Ruiz \cite{Vel24} proves the decay estimate of the Vlasov field, in the regime that the matter is massless, or massive and dispersive, which leads to the proof of asymptotic stability of Schwarzschild for the coupled system. 

\subsubsection{Spherically symmetric Einstein--matter models near exterior of Schwarzschild spacetime}
In this paper we consider the fully coupled Einstein--Vlasov system.
There have been several results on the Einstein equation coupled with various matter models in spherical symmetry, concerning data near the exterior of a black hole.
For the decay of the matter field and the stability of the exterior, see \cite{DR05, LO19b, Gau24} for the Einstein--Maxwell--scalar field system, and \cite{HS13} for the Einstein--Klein--Gordon system.
The formation of a black hole in this class was also studied: for the Einstein--scalar field system in \cite{Daf09}, and for the Einstein--Maxwell--charged scalar field system in \cite{KU22}.
The global structure of the maximal development was classified in \cite{Kom13}, and the extremal case, in which the event horizon is degenerate, was studied in \cite{AKU26a, AKU26b}.
We refer the reader to the closely related papers \cite{DH06, Hol10a}.

In the current paper, we prove the stability of Schwarzschild spacetime coupled with Vlasov equation, in the regime that the matter is massive and stays in a compact region, where we cannot expect dispersion.
In particular, the existence of static spacetime in this regime which can be an arbitrarily small perturbation of Schwarzschild is proved in \cite{Rei94, Jab21}, and such a static shell surrounding a Schwarzschild black hole was shown to be linearly stable in \cite{GRS25} when the mass of the shell is small compared to the mass of the black hole.
Therefore, the best that we can expect is that the perturbation converges toward a steady state, or the time derivative decays. 
This is the point where the nonlinear phase mixing comes in. 

\subsubsection{Nonlinear phase mixing}
The nonlinear phase mixing phenomenon was first studied in the context of plasma physics, where the Vlasov--Poisson system was considered.
In this case the phenomenon is called Landau damping, which was first observed in \cite{Lan46}.
The nonlinear stability of Vlasov--Poisson system on \(\TT_x^n\times \RR_v^n\), which is tightly connected to Landau damping, was first proved in the celebrated work by Mouhot--Villani \cite{MV11}.

Later, more general results were proved in this setting, namely on the space domain \(\TT_x^n\).
The result \cite{MV11} was proved for analytic initial data, and the result was extended to Gevrey data with \(\g < 3\) in \cite{BMM16}, of which proof was greatly simplified in \cite{GNR21}, by using the generator function method to control the weighted norm with infinite regularity. 
Recently, the result was extended to the case of Gevrey data with \(\g = 3\) in \cite{IPWW24}, where the existence of nonlinear scattering operator was also proved.
For the case with relativistic Vlasov--Poisson system, a similar Landau damping result was proved in \cite{You16}.

There were several other approaches to establish nonlinear phase mixing in the Vlasov-Poisson system. 
It is known that the nonlinear echoes cannot be controlled in finite regularity, and the nonlinear phase mixing phenomenon is known to breakdown in finite timescales \cite{Bed21}.
In the collisional setting, Bedrossian \cite{Bed17} studied how the nonlinear echoes can be suppressed by collisional relaxation in the Vlasov-Fokker-Planck equation, and Chaturvedi-Luk-Nguyen~\cite{CLN23} proved nonlinear phase mixing for the Vlasov-Poisson-Landau system in the weakly collisional regime.
Later, Bedrossian-Zhao-Zi \cite{BZZ25} proved the Landau damping result uniformly in the small collision frequency \(\n > 0\) for the Vlasov-Poisson-Landau equation, which is sufficient to establish the collisionless limit \(\n \to 0\) to the Vlasov-Poisson equation. 

Without collision, we cannot expect the nonlinear phase mixing phenomenon to hold in finite regularity for all time. 
Chaturvedi-Luk \cite{CL24} studied the nonlinear phase mixing phenomenon for the Vlasov-Poisson system in the Kepler potential, where particles are confined to bounded orbits.
This result proves the same type of estimates as in the current paper, namely that the nonlinear phase mixing phenomenon holds only on a finite timescale.

This paper and \cite{CL24} study the case where the particles are confined to bounded orbits, rather than the case where the given space domain is \(\TT_x^n\).
In contrast to the above setting on the torus, where the mixing occurs in the spatial variable, the mixing here occurs in the angle variable of the bounded orbits, which is not equal to the spatial variable; this adds a novel difficulty.

\subsubsection{Gravitational phase mixing}
There are several works on the phase mixing phenomenon in the system with external potential.
For linear systems, Chaturvedi-Luk~\cite{CL22} proved the decay estimate on 1D system with confining potential, which was generalized in \cite{MRV22} to the case with general confining potential.
Recently, Had\v{z}i\'c-Rein-Schrecker-Straub~\cite{HRSS24} and Had\v{z}i\'c-Schrecker~\cite{HS25} proved the linear stability of steady states for a general class of trapping Hamiltonians, overcoming the degeneracy near the elliptic stagnation point.

For the nonlinear system, Chaturvedi-Luk~\cite{CL24} proved the nonlinear phase mixing phenomenon for the Vlasov--Poisson system in the Kepler potential, which acts as a central force confining the particles to bounded orbits.
This can be understood as a direct Newtonian version of the current paper. 
The case when the potential is repulsive, hence the particles disperse, was studied in \cite{VV24} and \cite{PWY22}.

The stability of nonzero equilibrium was studied in this context.
Lynden-Bell \cite{Lyn62, Lyn67} suggested that phase mixing might be a mechanism for stability. 
Recently, this was studied rigorously in \cite{HRSS24} and \cite{HS25}.

Rioseco--Sarbach employed action angle variables to study the dynamics in general relativity: the relaxation of a collisionless gas accreting into a Schwarzschild black hole was proved in \cite{RS17}, where particles move on unbounded orbits and hence disperse. 
Linear phase mixing for bounded orbits was subsequently proved in \cite{RS18, RS20, RS23}, on a fixed black hole spacetime background or similar settings.

\subsection{Outline of paper}

In \textbf{Section~\ref{sec:Prelim}}, we establish the local well-posedness of the system and study the period function of orbits in the action angle coordinate and its monotonicity. 
In \textbf{Section~\ref{sec:LinEst}}, we prove the linear estimates stated in Theorem~\ref{thm:lin}. 
In Sections~\ref{sec:BA}--\ref{sec:PutEverything}, we prove our main nonlinear theorem (Theorem~\ref{thm:nonlin}). 
We begin by introducing the main bootstrap assumptions in \textbf{Section~\ref{sec:BA}}. 
Then in \textbf{Section~\ref{sec:DAAV}}, we introduce the dynamical action angle variables and control the change of variables map. 
The next three sections are devoted to the main estimates: 
in \textbf{Section~\ref{sec:vlasov}}, we prove the estimates for the distribution function \(f\) solving the Vlasov equation; 
in \textbf{Section~\ref{sec:Integral}}, we prove the estimates for the energy-momentum tensor with integral estimates; 
in \textbf{Section~\ref{sec:Elliptic}}, we prove the estimates for the metric components using the ODE formulation of the Einstein equation. 
Finally, we put together the estimates and conclude the proof of Theorem~\ref{thm:nonlin} in \textbf{Section~\ref{sec:PutEverything}}. 

\bigskip
\noindent\textbf{Acknowledgements.}
I would like to thank my advisor Jonathan Luk for introducing me to this problem and for his invaluable guidance and support.
I gratefully acknowledge the support of the National Science Foundation under the grant DMS-2304445.

\section{Setup and Preliminaries}\label{sec:Prelim}
\subsection{Local existence}\label{subsec:LWP}
We prove that any small perturbation \((A_\init, K_\init, \Psi, f_\init)\) of the Schwarzschild initial data gives rise to a local-in-time solution to the Einstein--Vlasov system.
In particular, we prove that if the initial data is small in a weak norm, then the solution exists for a small time, and in addition if the initial data has some regularity then it propagates in this small time. 

To prove this, we first present the following reduced system: 
\begin{align}
(R^2\sqrt{A}')' &= -\frac{A^{5/2}R^2}{8}\big(16\pi\r + \frac{3}{2}K^2 - K\Psi - \frac{1}{2}\Psi^2\big), \label{eq:rGauss}\\
(R^2A\a')' &= \a R^2A^3\big(4\pi\tr T + 8\pi\r + \frac{3}{2}K^2-K\Psi+\frac{1}{2}\Psi^2\big) + R^2A^3\b\Psi', \label{eq:rSumEE}\\
(A^3R^3K)' &= 8\pi A^4R^3 j + A^2R^2(AR\Psi)', \label{eq:rCodazzi}\\
\left(\frac{\b}{R}\right)' &= \frac{\a}{2R}(3K-\Psi), \label{eq:rMeanCurv}\\
\Dfr f &= 0, \label{eq:rVlasov}
\end{align}
where the energy-momentum tensor is defined in the same way as in \eqref{rhodef}--\eqref{SRdef}, and \(\Dfr\) is the Vlasov operator defined in \eqref{eq:vlasov-rwl}.
We first prove the local existence of this system. 
Then, to complete the proof of Proposition \ref{prop:LWP}, we show that any solution to the reduced system gives rise to a solution of the full system.

\begin{prop}[Local Existence for Reduced System]\label{prop:rLWP}
    Let \(m, \lfr\) and \(m/4 < R_a < R_b < R_+\) be arbitrary positive numbers. 
    There are positive constants \(\e_0 > 0\) and \(C_0, C_k > 0\) such that the following statement is true: Say the initial data \(A_\init, K_\init, \Psi\) and \(f_\init\) with the support condition 
    \[
    \supp(A_\init - A_0, K_\init - K_0, \Psi - \Psi_0) \subset [m/4, R_+], \qquad \supp(f_\init)\subset [R_a, R_b]\times [-\lfr, \lfr]\times (0, \lfr],
    \]
    and the smallness condition 
    \[
    \norm{A_\init-A_0}_{C^2} + \norm{K_\init-K_0}_{C^1} + \norm{\Psi-\Psi_0}_{C^1} + \norm{f_\init}_{C^0} \le \e < \e_0.
    \]
    Furthermore let \(f_{\init}\) be Lipschitz. 
    Then, there is a time \(T>0\) (that depends on the initial data only through \(m\) and \(\norm{f_\init}_{C^{0, 1}}\)) such that the solution exists up to time \(T\), i.e. there is a function 
    \[
    \a, A, \b, K\colon [0, T)_s\times [m/4, \infty)_R\to \RR
    \]
    and 
    \[
    f \colon [0, T)_s\times [m/4, \infty)_R\times \RR_w\times \RR^{>0}_\ell\to \RR^{\ge 0} 
    \]
    which satisfy the equations \eqref{eq:rGauss}--\eqref{eq:rMeanCurv} and \eqref{eq:rVlasov}, and achieve the boundary condition \eqref{bdrycond}.
    Moreover, the analogous statements as in Proposition \ref{prop:LWP} hold for this solution, i.e. the estimates \eqref{LWP-est} with constant \(C_0\), \eqref{LWP-support} and \eqref{rLWP-est-high} with constant \(C_k\) hold for this solution.
\end{prop}
\begin{proof}
The proof is a standard iteration argument. 
Throughout, \(C\) is a constant which depends only on \(m, \lfr, R_a, R_b, R_+\).
Set 
\[
\check\a \coloneqq \a-\a_0, \quad \check A \coloneqq A-A_0, \quad \check\b \coloneqq \b-\b_0, \quad \check K \coloneqq K-K_0,
\]
which are all functions of \(R\) only. 

\step{1} (Description of the norm) Let \(X_T\) be the Banach space of tuples \(u = (\check\a, \check A, \check\b, \check K, f)\) of deviations with the regularity and support stated in the proposition, equipped with the norm
\[
\norm{u}_{X_T} \coloneqq \sup_{s\in[0,T]}\Big(\norm{\check\a}_{C^2} + \norm{\check A}_{C^2} + \norm{\check\b}_{C^1} + \norm{\check K}_{C^1} + \norm{f}_{C^0}\Big),
\]
and let \(\Bcal_M = \{\norm{u}_{X_T}\le M\e\}\), where \(M\) is a large constant to be determined later, which should be independent of \(\e\).
Note that we will choose \(\e_0\) later so that \(M\e \le 1/2\). 
Furthermore, we use a weaker norm \(Y_T\) in which we prove the convergence of iteration sequence:
\[
\norm{u}_{Y_T} \coloneqq \sup_{[0,T]}\big(\norm{\check\a}_{C^1} + \norm{\check A}_{C^1} + \norm{\check\b}_{C^0} + \norm{\check K}_{C^0} + \norm{f}_{C^0}\big).
\]
Assuming that the solution satisfies the size condition \(u\in \Bcal_M\) with sufficiently small \(\e\), we know that \(A\) is away from \(0\), and each metric coefficient has bounded range. 

\step{2} (Iteration scheme) We define the iteration scheme as follows: Let 
\begin{align*}
u_0 &= (\check\a_\init, \check A_\init, \check \b_\init, \check K_\init, f_\init) \in \Bcal_M \\
\intertext{and for given \(u_n\in \Bcal_M\),}
u_{n+1} &= (\check \a_{n+1}, \check A_{n+1}, \check \b_{n+1}, \check K_{n+1}, f_{n+1}).
\end{align*}
Here, we define \(f_{n+1}\) as a solution of the transport equation \(\Dfr_n f_{n+1} = 0\) with initial data \(f_\init\), where \(\Dfr_n\) is the transport operator given in \eqref{eq:vlasov-rwl} with \(\a, A, \b, K\) replaced by \(\a_n, A_n, \b_n, K_n\).
This updates the energy-momentum tensor, which we denote by \(\r_{n+1}, \tr T_{n+1}, j_{n+1}, S_{R, n+1}\), which is calculated with \(f_{n+1}\) and \(A_n\). 
Then, we define the updated metric coefficients \((\a_{n+1}, A_{n+1}, \b_{n+1}, K_{n+1})\) as a solution of the Einstein equation with new energy-momentum tensor, i.e. by solving the ODE system \eqref{absODE}, with 
\[
\vec{S} = (\r_{n+1}, \tr T_{n+1}, j_{n+1}, S_{R, n+1}, \Psi-\Psi_0).
\] 
Note that \(\Psi\) is fixed in each iteration here. 

\step{3}(\(u_{n+1}\in \Bcal_M\)) 
We need to show three things: i) The solutions \(f_{n+1}\) and \((\check\a_{n+1}, \check A_{n+1},\allowbreak \check\b_{n+1}, \check K_{n+1})\) are well-defined, ii) They satisfies the support condition, iii) They satisfy the size condition \(\norm{u_{n+1}}_{X_T} \le M\e\). These tasks should be done in certain order. 
Note that \(C\) here is independent of \(M\) in particular.  

First, we prove that \(f_{n+1}\) is well-defined and has a desired support and size. 
The existence of solution for the transport equation \(\Dfr_n f_{n+1} = 0\) is guaranteed from the bound on the derivative of the coefficients \(\a_n, A_n, \b_n, K_n\): 
Because we have 
\[
\norm{\a_n}_{C^2} + \norm{A_n}_{C^2} + \norm{\b_n}_{C^1} + \norm{K_n}_{C^1} \le C,
\]
we know that the \(C^0\) solution \(f_{n+1}\) exists in \(\Ocal(1)\) time interval, with size \(\norm{f_{n+1}}_{C^0}\le \e\).
The support condition can be easily verified, as the coefficients themselves are bounded: 
\[
\norm{\a_n}_{C^1} + \norm{A_n}_{C^1} + \norm{\b_n}_{C^0} + \norm{K_n}_{C^0} \le C,
\]
which shows that up to time \(T\) which depends only on the constants \(m, \lfr, R_a, R_b, R_+\), the characteristic curve stays in a desired region. 
Note also that \(\ell\) is a conserved quantity, hence the \(\ell\)-support stays the same. 

Then, we can estimate the energy-momentum tensor components, which are given by integrals of \(f_{n+1}\) and \(A_n\).
Because \(A_n\) and \(R\) are globally away from \(0\) and the \(\ell\)- and \(w\)-support of \(f_{n+1}\) are bounded, the energy-momentum tensor components are also bounded by \(C\e\).
The assumed smallness on \(\Psi-\Psi_0\) is sufficient.

Now we estimate \(\check\a_{n+1}, \check A_{n+1}, \check\b_{n+1}, \check K_{n+1}\) from the ODE stability, and get that they are bounded by \(C\e\) as well.
We state the ODE stability result that we need here; this is a standard result, and will be used in Section \ref{sec:Elliptic} for nonlinear bootstrap argument as well.

\begin{lem}[\cite{T12}, Theorem 2.11, 2.12]\label{lem:ODEstab}
    Let \(\vec{F}\colon \RR^{N+M+1}\to \RR^N\) be a smooth function, and consider the ODE
\[
\vec{X}'(R) = \vec{F}(\vec{X}(R),\vec{S}(R), R), \quad \vec{X}(R_0) = \vec{X}_0.
\]
Let \(\vec{X}_1\) be a solution on \([R_1, R_0]\) for \(\vec{S} = \vec{S}_1\). 
Then, there is \(\d_1 > 0\) and \(C > 0\) such that if \(\norm{\vec{S}-\vec{S}_1}_{C^0} \le \d_1\), then there is a solution \(\vec{X}\) on \([R_1, R_0]\) for \(\vec{S}\) with the same initial data, and it satisfies
\begin{equation}\label{ODEest}
\norm{\vec{X}-\vec{X}_1}_{C^1} \le C\norm{\vec{S}-\vec{S}_1}_{C^0}.
\end{equation}
Moreover, if \(\vec{S}_1\in C^k\) for \(k\ge 0\), then we have \(C^{k+1}\) stability: there are \(C_k\) which in addition depend on \(k\) such that 
\begin{equation}\label{ODEestH}
\norm{\vec{X}-\vec{X}_1}_{C^{k+1}} \le C\norm{\vec{S}-\vec{S}_1}_{C^k}.
\end{equation}
\end{lem}

Because the energy-momentum tensor and \(\Psi\) is zero outside, the support condition on \((\check\a_{n+1}, \check A_{n+1}, \check\b_{n+1}, \check K_{n+1})\) is clearly true.
The existence and the size condition is a direct consequence of Lemma \ref{lem:ODEstab}, near Schwarzschild solution \(\vec{S}_1 = 0\), \(\vec{X}_1 = (A_0, A_0', \a_0, \a'_0, K_0, \b_0)\).
Note that in the meantime we need to choose \(\e_0\) small enough so that \(\norm{\vec{S}}_{C^0} \le \d_1\), where \(\d_1\) is the constant in Lemma \ref{lem:ODEstab} at this \(\vec{X}_1\). 
At this point we can choose \(M_0\) such that our iteration scheme is well-defined for \(M \ge M_0\). 

\step{4}(Contraction) We now show that the iteration scheme is a contraction in the \(Y_T\) norm.
Throughout this step we work with two consecutive differences and abbreviate
\[
\Delta_n \coloneqq \norm{u_n - u_{n-1}}_{Y_T}.
\]
Our goal is to prove that, after shrinking \(\e_0\) and \(T\) if necessary,
\begin{equation}\label{rLWP-contraction}
\norm{u_{n+1}-u_n}_{Y_T} \le \tfrac12\,\Delta_n,
\end{equation}
so that \(\{u_n\}\) is Cauchy in \(Y_T\). We estimate the Vlasov part, the density variables, and the metric coefficients in this order.

\emph{(i) Transport difference.} Let \(Z_n(\t; s, R, w, \ell)\) be the characteristic flow of \(\Dfr_n\), i.e. the solution of
\[
\frac{d}{d\t}Z_n = \Vcal_n(Z_n), \qquad Z_n(s; s, R, w, \ell) = (R, w, \ell),
\]
where \(\Vcal_n = \big(\frac{\a_n w}{A_n E_n}-\b_n,\ \frac{\a_n}{A_n^3}\frac{\ell}{R^3E_n}+\a_n\frac{A_n'}{A_n^4}\frac{\ell}{R^2E_n}-\frac{\a_n'}{A_n}E_n+\a_n K_n w, 0\big)\) is the coefficient vector field of \eqref{eq:vlasov-rwl} with \(u_n\) inserted, and \(\ell\) is conserved.
Since \(f_{n+1}\) and \(f_n\) solve \(\Dfr_n f_{n+1} = 0\) and \(\Dfr_{n-1} f_n = 0\) with the same initial data \(f_\init\),
\[
f_{n+1}(s, R, w, \ell) = f_\init\big(Z_n(0; s, R, w, \ell)\big), \qquad f_n(s, R, w, \ell) = f_\init\big(Z_{n-1}(0; s, R, w, \ell)\big),
\]
so that
\begin{equation}\label{rLWP-fdiff}
\abs{f_{n+1}-f_n}(s, R, w, \ell) \le \norm{f_\init}_{C^{0,1}}\,\abs{Z_n(0; s, R, w, \ell)-Z_{n-1}(0; s, R, w, \ell)}.
\end{equation}
The two flows obey the same ODE with vector fields \(\Vcal_n, \Vcal_{n-1}\). On the compact region where the characteristics live (Step 3), \(A_n\) is away from zero, and \(\Vcal_n\) depends on \((\a_n, A_n, \b_n, K_n)\) together with the first derivatives \(A'_n, \a'_n\). Hence
\begin{align*}
\norm{\Vcal_n - \Vcal_{n-1}}_{C^0}
&\le C\Big(\norm{\check\a_n-\check\a_{n-1}}_{C^1}+\norm{\check A_n-\check A_{n-1}}_{C^1}+\norm{\check\b_n-\check\b_{n-1}}_{C^0}+\norm{\check K_n-\check K_{n-1}}_{C^0}\Big) \\
&\le C\,\Delta_n.
\end{align*}
It is precisely the first derivatives \(A', \a'\) that force the \(C^1\) norm of \(\check\a, \check A\) here, which is why the contraction is measured in \(Y_T\) rather than \(X_T\). 
Moreover \(\Vcal_n\) is Lipschitz in \((R, w)\) on this region, uniformly in \(n\), with constant \(\le C\) by the \(\Bcal_M\) bounds of Step 3. 
Gronwall's inequality applied to \(Z_n - Z_{n-1}\) on \([0, s]\subset [0, T]\) therefore gives, for \(T\le 1\), 
\[
\abs{\frac{d}{d\t}(Z_n(\t; \cdot)-Z_{n-1}(\t; \cdot))} \le C\abs{Z_n(\t; \cdot)-Z_{n-1}(\t; \cdot)} + C\norm{\Vcal_n-\Vcal_{n-1}}_{C^0},
\]
\[
\norm{Z_n(0; \cdot)-Z_{n-1}(0; \cdot)}_{C^0} \le CTe^{CT}\norm{\Vcal_n-\Vcal_{n-1}}_{C^0} \le CT\,\Delta_n.
\]
Combined with \eqref{rLWP-fdiff} and the Lipschitz regularity of \(f_\init\), this yields
\begin{equation}\label{rLWP-fest}
\sup_{[0,T]}\norm{f_{n+1}-f_n}_{C^0} \le CT\,\Delta_n.
\end{equation}

\emph{(ii) Density difference.} The density variables of \(\vec{S}_{n+1}\) are the velocity moments \eqref{rhodef}--\eqref{SRdef} of \(f_{n+1}\) with weights depending on \(A_n\) and \(R\). 
The difference \(\vec{S}_{n+1}-\vec{S}_n\) follows from two parts: the difference in \(A\) and the integral of \(f\).
Since the \(w\)- and \(\ell\)-supports of \(f_{n+1}, f_n\) lie in a fixed compact set and \(A_n, A_{n-1}, R\) are bounded away from \(0\), we can control the difference of integral using only the \(C^0\) norm of the distribution function \eqref{rLWP-fest}.
The difference caused by \(A\) is multiplied by the integral, which is of size \(\Ocal(\e)\). 
This explains that
\begin{equation}\label{rLWP-Sest}
\norm{\vec{S}_{n+1}-\vec{S}_n}_{C^0} \le C(T+\e)\Delta_n.
\end{equation}
Note that only the \(C^0\) norm of \(\vec{S}_{n+1}-\vec{S}_n\) is needed, so no derivative falls on \(f\); this is the step where the weaker norm avoids the loss of derivatives.

\emph{(iii) Metric difference.} The tuples \(\vec{X}_{n+1}, \vec{X}_n\) solve the ODE \eqref{absODE} with sources \(\vec{S}_{n+1}, \vec{S}_n\) and the same boundary data at \(R = R_+\). 
Both stay within \(\Ocal(\e)\) of \(\vec{X}_1\), hence we know the existence from Lemma \ref{lem:ODEstab}.
To estimate the difference, we have
\[
\abs{\frac{d}{dR}(\vec{X}_{n+1}-\vec{X}_n)} \le C\abs{\vec{X}_{n+1}-\vec{X}_n} + C\abs{\vec{S}_{n+1}-\vec{S}_n}
\]
pointwise, hence from Gronwall's inequality we have 
\[
\norm{\vec{X}_{n+1}-\vec{X}_n}_{C^0} \le C\norm{\vec{S}_{n+1}-\vec{S}_n}_{C^0} \le CT\,\Delta_n.
\]
These are exactly the metric components of the \(Y_T\) norm.

\emph{Conclusion.} Adding \eqref{rLWP-fest} and the bound of (iii),
\[
\norm{u_{n+1}-u_n}_{Y_T} \le C(T+\e)\,\Delta_n.
\]
Shrinking \(\e_0\) and \(T\) so that \(C(T+\e_0)\le 1/2\) yields \eqref{rLWP-contraction}. Thus \(\{u_n\}\) is Cauchy in \(Y_T\) and converges to a limit \(u\) in \(Y_T\). 
Because we know that those are solutions of ODEs, this automatically upgrades to the convergence in \(X_T\) as well, and the limit \(u\) is a solution of the reduced system.
The estimates \eqref{LWP-est} and \eqref{LWP-support} follow from the construction of the iteration scheme, except the control on the second derivative of \(\b\): this is a direct consequence of \eqref{eq:rMeanCurv}.

\step{5}(Higher regularity) We now show that if the initial data has higher regularity, then the solution also has higher regularity. 
Let \(u = (\check \a, \check A, \check\b, \check K, f)\) be the solution of the reduced system that we have, with the initial size condition 
\[
\norm{A_\init-A_0}_{C^{k+2}} + \norm{K_\init-K_0}_{C^{k+1}} + \norm{\Psi-\Psi_0}_{C^{k+1}} 
+\sum_{\abs{I}\le k} \norm{\rd^I f_\init}_{C^0} \le \e < \e_0.
\]
Then, we induct to obtain the desired estimate. 
From the previous step, we know that the solution has the following estimate:
\[
\sup_{s\in[0,T]}\Big( \norm{\a-\a_0}_{C^{2}} + \norm{A-A_0}_{C^{2}} + \norm{\b-\b_0}_{C^{2}} + \norm{K-K_0}_{C^{1}} + \norm{f}_{C^0}\Big) \le C_0\,\e,
\]
which is exactly \eqref{LWP-est}, with the second derivative of \(\b\) coming from \eqref{eq:rMeanCurv} as noted above.
This is the base case \(l = 0\) of the induction: we will show that for \(0\le l \le k\) we have a constant \(C_l\) such that
\begin{equation}\label{rLWP-high-induction}
\sup_{s\in[0,T]}\Big( \norm{\a-\a_0}_{C^{l+2}} + \norm{A-A_0}_{C^{l+2}} + \norm{\b-\b_0}_{C^{l+2}} + \norm{K-K_0}_{C^{l+1}} + \sum_{\abs{I}\le l}\norm{\rd^I f}_{C^0}\Big) \le C_l\,\e
\end{equation}
is true.
Assume that \eqref{rLWP-high-induction} holds at level \(l-1\) for some \(1\le l\le k\); in particular
\begin{equation}\label{rLWP-high-IH}
\sup_{s\in[0,T]}\Big( \norm{\a-\a_0}_{C^{l+1}} + \norm{A-A_0}_{C^{l+1}} + \norm{\b-\b_0}_{C^{l+1}} + \norm{K-K_0}_{C^{l}} + \sum_{\abs{I}\le l-1}\norm{\rd^I f}_{C^0}\Big) \le C_{l-1}\,\e.
\end{equation}
For the induction step, we improve the estimates on \(f\), on the source term, and on the metric coefficients in this order.
Throughout, \(\rd^I = \rd_R^{i_1}\rd_w^{i_2}\rd_\ell^{i_3}\) denotes a derivative of order \(\abs{I} = i_1+i_2+i_3\), and \(C\) is a constant that may depend on \(l\) and \(C_{l-1}\) but not on \(\e\).

\emph{(i) Distribution function.}
On the compact region in which the characteristics live (Step 3), the quantities \(E, R, A\) are bounded away from their degenerate values, so the coefficients are smooth functions of \((R, w, \ell)\) and of the metric coefficients \(\a, A, \b, K\) together with the first derivatives \(\a', A'\).
Since \(\rd\) commutes with \(\rd_s\), this proves that all the coefficients of \([\Dfr, \rd^l]\) is of size \(\Ocal(1)\), from the induction hypothesis.
Hence, for a multi-index \(I\) with \(\abs{I} = l\), commuting \(\rd^I\) through \(\Dfr f = 0\) we get the estimate 
\[
\abs{\Dfr(\rd^I f)} \le \sum_{\substack{i_1 + i_2\le \abs{I}+1 \\ i_2 \le \abs{I}}} \abs{\rd^{i_1} (\a, \a', A, A', \b, K)}\abs{\rd^{i_2} f} 
\le C\sum_{\abs{I'}=l}\abs{\rd^{I'} f} + C\,\e.
\]
Therefore, using Gronwall on each characteristic line, we obtain that 
\[
\sum_{\abs{I}\le l}\norm{\rd^I f}_{C^0}
\le C\e (T+1)e^{CT}
\] 
on \([0,T]\).
Here we used the fact that the initial data has smallness \(\sum_{\abs{I}\le l}\norm{\rd^I f_\init} \le \e\).

\emph{(ii) Source term.} The density variables of \(\vec{S}=(\r,\tr T, j, S_R, \Psi-\Psi_0)\) are the velocity moments \eqref{rhodef}--\eqref{SRdef} of \(f\), whose weights are smooth functions of \(A, R, w, \ell\) on the support of \(f\), integrated over the fixed compact \((w,\ell)\)-support.
The derivative falling on \(A\) is clearly controllable, and for the derivative falling on the integral we need the above estimates on the \(l\)-th derivative of \(f\). 
The assumed smallness on \(\Psi-\Psi_0\) is sufficient.
\begin{equation}\label{rLWP-high-Sest}
\sup_{s\in[0,T]}\norm{\vec{S}}_{C^l}\le C\,\e.
\end{equation}

\emph{(iii) Metric coefficients.} The tuple \(\vec{X} = (A, A', \a, \a', K, \b)\) solves the ODE \eqref{absODE} with source \(\vec{S}\) and the boundary data of \(\vec{X}_0\) at \(R=R_+\).
The Schwarzschild profile \(\vec{X}_0\) solves the same equation with source \(\vec{S}_1 = 0\in C^l\), and \(\norm{\vec{S}}_{C^l}\le C\e \le \d_1\) by \eqref{rLWP-high-Sest}, so the \(C^{l+1}\)-stability \eqref{ODEestH} of Lemma \ref{lem:ODEstab} gives
\[
\sup_{s\in[0,T]}\norm{\vec{X}-\vec{X}_0}_{C^{l+1}} \le C\,\norm{\vec{S}}_{C^l}\le C\,\e.
\]
Because \(\vec{X}\) contains both \(A\) and \(A'\) (resp. \(\a\) and \(\a'\)), this controls \(\norm{A-A_0}_{C^{l+2}}\) and \(\norm{\a-\a_0}_{C^{l+2}}\), together with \(\norm{K-K_0}_{C^{l+1}}\) and \(\norm{\b-\b_0}_{C^{l+1}}\), each by \(C\e\).
This is the gain of one derivative that we need.
It remains to upgrade \(\b\) by one more derivative, exactly as in the base case: we use the equation \eqref{eq:rMeanCurv}.

Collecting (i) and (iii) establishes \eqref{rLWP-high-induction} at level \(l\). This closes the induction, and the case \(l = k\) is precisely the higher-regularity estimate \eqref{rLWP-est-high}, completing the proof of Proposition \ref{prop:rLWP}.

\end{proof}

Now we prove Proposition \ref{prop:LWP}.
Because we claimed the same estimates and support properties for solutions, all that we need to show is that the solution solves the full system, i.e. it satisfies the equations \eqref{eq:Keq} and \eqref{eq:EE}.

\begin{proof}[Proof of Proposition \ref{prop:LWP}]
Given an initial data \((A_\init, K_\init, \Psi, f_\init)\) which satisfies the condition of Proposition \ref{prop:LWP}, we know the existence of the solution of the reduced system for some time \(T > 0\) by Proposition \ref{prop:rLWP}.
Now, we prove that this is the solution of the full system, i.e. it satisfies the equations \eqref{eq:Keq} and \eqref{eq:EE}.
For this, set \(F(s, R), G(s, R)\) as follows:
\begin{align*}
F(s, R) &= \rd_s A - \b A' - \frac{\b}{R}A - \frac{1}{2}\a AK + \frac{1}{2}\a A\Psi, \\
G(s, R) &= \rd_s K - \b K' - \frac{2\a'}{RA^2} - \frac{2\a'A'}{A^3} - \frac{\a(A')^2}{A^4} - \frac{2\a A'}{RA^3} \\
&\qquad + \a\left(\frac{3}{4}K^2 - \frac{3}{2}K\Psi + \frac{3}{4}\Psi^2 + 8\pi S_R\right) + \b\Psi'.
\end{align*}
so that our goal becomes showing \(F \equiv G \equiv 0\). 

By differentiating \eqref{eq:rGauss} and \eqref{eq:rCodazzi} in \(s\), we get a relation between \(\rd_s\) of \(K\) and \(A\). Then, by substituting the expressions of \(\rd_s A\) and \(\rd_s K\) in terms of \(F\) and \(G\), we get a system of ODEs for \(F\) and \(G\).
In the meantime, we need an information on the \(s\)-derivative of density variables. 
From the definition of the energy-momentum tensor, it should satisfy the Bianchi identity, which is written as follows in our setup: 
\begin{align*}
\frac{1}{\a}(\rd_s - \b \rd_R)\r &= -\frac{j'}{A} - j\left(\frac{2\a'}{A\a} + \frac{2}{AR} + \frac{2A'}{A^2}\right) + \frac{K}{2}(3S_R - \r - \tr T) + \frac{\Psi}{2}(3\r - S_R + \tr T),\\
\frac{1}{\a}(\rd_s - \b \rd_R)j &= -\frac{S_R'}{A} - \frac{\a'\r}{A\a} - S_R \left(\frac{\a'}{A\a} + \frac{3}{AR} + \frac{3A'}{A^2}\right) + (\r + \tr T) \left(\frac{1}{AR} + \frac{A'}{A^2}\right) + (K+\Psi) j.
\end{align*}
This gives the following differential identity in \(F\) and \(G\). 
\begin{align*}
0 &= F'' + \Big(\frac2R - \frac{A'}{A}\Big)F' + \Big[\frac{(A')^2}{2A^2} + \frac{3A^2}{4}\Big(16\pi\r + \frac{3}{2}K^2-K\Psi -\frac{1}{2}\Psi^2\Big)\Big]F + \frac{A^3(3K-\Psi)}{4}G, \\
0 &= \frac{3K-\Psi}{A}F' - \Big(\frac{(3K-\Psi)A'}{A^2} + 8\pi j\Big)F + G' + \Big(\frac{3A'}{A} + \frac3R\Big)G.
\end{align*}
From the boundary condition, we know that \(F = G = 0\) at \(R = R_+\). 
Therefore, from the uniqueness of the ODE solution, we have \(F = G = 0\) for all \(R\in [m/4, \infty]\).
This completes the proof of Proposition \ref{prop:LWP}. 
\end{proof}

\subsection{Period function and its monotonicity}
To define the action angle \(Q\) (or \(Q_T\) for the dynamical case), we need to understand the following quantity.
From the following discussion, we consider fixed \((H_0, M)\) which is in the support of our data, i.e. \((H_0, M)\) such that the potential curve 
\begin{equation}\label{Vdef}
V_0(R, M) = \a_0(R)^2\left(1+\frac{M}{A_0(R)^2R^2}\right)
\end{equation}
as a function of \(R\) has a local maximum and local minimum, and \(H_0^2\) is greater than this minimum and less than \(1\).
In particular, under the condition \(M \ge 16m^2\) the local maximum is at least \(1\), and \(V_0(R, M) = 1\) for some \(R \ge (\frac{3}{2}+\sqrt{2})m\).
Hence, the potential curve intersects \(V_0(R, M) = H_0^2\) twice between the local maximum and infinity. 
\begin{defn}[Normalization constant \(\O_{\Sch}(H_0, M)\) of the action angle]\label{Omegadef1}
Given \(H_0\) and \(M\) with above condition, we first define \((\frac{3}{2}+\sqrt{2}) m < R_-^0(H_0, M) < R_+^0(H_0, M)\) as the two roots of the equation \(H_0^2 = V_0(R, M)\) so that the local minimum of the potential curve lies in \((R_-^0(H_0, M), R_+^0(H_0, M))\), which are well-defined under the above condition.
We then define \(\O_{\Sch}(H_0, M)\) to satisfy
\begin{equation}\label{Omegadef}
\frac{\pi}{\O_{\Sch}(H_0, M)} = \int_{R_-^0(H_0, M)}^{R_+^0(H_0, M)} \frac{A_0(R)H_0}{\a_0(R)^2w(R, H_0, M)} dR,
\end{equation}
where \(w(R, H_0, M)\) is given by
\[
w(R, H_0, M) \coloneqq \sqrt{\frac{H_0^2}{\a_0(R)^2}-1-\frac{M}{A_0(R)^2R^2}}.
\]
\end{defn}
Note that \(w\) is zero at \(R = R_-^0(H_0, M)\) and \(R = R_+^0(H_0, M)\), hence the integral is singular.
Since the potential curve has nonzero derivative at those points, \(w^2\) converges to \(0\) linearly, hence the integral is convergent. 
Also, because the second \(R\)-derivative of \(V_0(R, M)\) is bounded, the above integral is away from zero, and hence \(\O_{\Sch}(H_0, M)\) is well-defined and finite.

For the linear estimate, we need to use the non-stationary phase argument with 
\[
\left(\frac{1}{ikt\rd_\ell (\O_{\Sch}(H_0, M))}\rd_\ell \right)e^{ikt\O_{\Sch}(H_0, M)}
= e^{ikt\O_{\Sch}(H_0, M)},
\]
hence we need to show that \(\rd_\ell \O_{\Sch}(H_0, M)\) is away from zero.
Note that both \(H_0\) and \(M\) depend on \(\ell\), so this is actually equal to
\[
\rd_\ell (\O_{\Sch}(H_0, M))
= (\rd_\ell H_0)(\rd_{H_0}\O_{\Sch})(H_0, M) + (\rd_M\O_{\Sch})(H_0, M).
\]
We want to work on the set of \((H_0, M)\), at which this value has sign, for any \((r, w, \ell)\) which gives this \((H_0, M)\). 
In addition we require the derivative \(\rd_{H_0}\O_\Sch(H_0, M)\) to be negative as well. 
This is the definition of the set \(\Sfr_{\cfr}\).
\begin{defn}[The admissible set \(\Sfr_{\cfr}\)]\label{Sfrdef}
    We define \(\Sfr_{\cfr}\) to be the set of \((H_0, M)\) such that the following is true: For \((H_0, M)\) which satisfies the condition above, we set \(R_-^0(H_0, M)\) and \(R_+^0(H_0, M)\). 
    If for all \(R\in [R_-^0(H_0, M), R_+^0(H_0, M)]\), at \((R, w(R, H_0, M), \ell = M)\) 
    \[\rd_\ell (\O_\Sch(H_0, M)) < -\cfr,\]
    and in addition the derivative satisfies
    \[
    \rd_{H_0}\O_\Sch(H_0, M) < -\cfr,
    \] 
    (which is not a condition depending on \(R\)) then we set \((H_0, M)\in \Sfr_{\cfr}\).
\end{defn}

As soon as we choose \(\cfr > 0\), the support \(\Scal_\cfr\) of our initial data is compact in \((H_0, M)\) coordinate. 
Therefore, it suffices to understand the set \(\Sfr_0\), which is given as an increasing limit of the open sets \(\Sfr_{\cfr}\) for \(\cfr > 0\). 
The next proposition describes the condition that we know \((H_0, M)\in \Sfr_0\).
First, we know that if \(M = \ell\) is sufficiently large then \((H_0, M)\in \Sfr_0\), and we also know that if \(H_0^2\) is sufficiently close to \(V_{\min}(M)\) then \((H_0, M)\in \Sfr_0\) as well, where \(V_{\min}(M)\) is the minimum value of the potential curve \(V_0(R, M)\).
This is the content of the following proposition. 

\begin{prop}\label{prop:monotonic}
    Given \(m\), the following is true.
    \begin{enumerate}
        \item There is \(M_0(m) > 0\) such that for any \(M > M_0(m)\) and \(V_{\min}(M) < H_0^2 < 1\), we have \((H_0, M) \in \Sfr_0\).
        \item For any \(\Mbar > 16m^2\), there is \(\d(m, \Mbar) > 0\) such that for any \(16m^2 < M < \Mbar\) and \(V_{\min}(M) < H_0^2 < V_{\min}(M) + \d(m, \Mbar)\), we have \((H_0, M) \in \Sfr_0\).
    \end{enumerate}
\end{prop}
\begin{proof}
(1) \step{0} By the chain rule we have the following identity: 
\[
\rd_\ell (\O_\Sch^{-1}(H_0, M)) = \rd_\ell H_0 \cdot (\rd_{H_0}\O_\Sch^{-1})(H_0, M) + (\rd_M\O_\Sch^{-1})(H_0, M).
\]
In the proof we first show that \((\rd_M \O_\Sch^{-1})(H_0, M) < 0\).
Therefore, because \(\rd_\ell H_0 = \frac{\a_0^2}{2H_0}\frac{1}{R^2A_0^2}\) decreases in \(R\), it suffices to prove when \(R\) is maximal for given \((H_0, M)\), i.e. when \(R = R_+^0(H_0, M)\).
There, we write \(\O\) in \((R_+^0, \ell = M)\) coordinates, and show that the \(\ell\)-derivative is negative. 
Since we already have an opposite sign for \(\rd_M\O_\Sch\), this automatically proves the inequality for \(\rd_{H_0} \O_\Sch\), hence completing the proof. 

In both steps, we simplify \eqref{Omegadef} by using the areal radius \(r\): especially because
\[
\frac{dr}{dR} = \a_0 A_0, 
\]
we have 
\[
c(H_0, M) \coloneqq 
\frac{\pi}{\O_\Sch(H_0, M)} = \int_{r_-^0(H_0, M)}^{r_+^0(H_0, M)} \frac{H_0 dr}{\big(1-\frac{2m}{r}\big)\sqrt{H_0^2 - (1-\frac{2m}{r})(1+\frac{M}{r^2})}},
\]
where \(r_-^0(H_0, M)\) and \(r_+^0(H_0, M)\) are \(R_-^0(H_0, M)\) and \(R_+^0(H_0, M)\) in areal radius coordinate. 
We abuse \(V_0(r, M)\) to denote the potential function 
\[
V_0(r, M) \coloneqq (1-\frac{2m}{r})(1+\frac{M}{r^2}),
\]
and use \(V_{\min}(M)\) to denote the local minimum of \(V_0(r, M)\) as a function of \(r\). 

\step{1} (\((\rd_M c)(H_0, M) < 0\)) For this, we define the area function 
\begin{equation}\label{Areadef}
I(H_0, M) = \int_{r_-^0(H_0, M)}^{r_+^0(H_0, M)} \frac{\sqrt{H_0^2 - V_0(r, M)}}{1-\frac{2m}{r}}dr,
\end{equation}
so that \(\rd_{H_0} I = c\). This gives 
\[
(H_0^2-V_{\min}(M))\rd_M c(H_0, M) = \rd_{H_0} ((H_0^2-V_{\min}(M))\rd_M I) - 2H_0\rd_M I.
\]
We can compute RHS as follows:
\begin{align*}
(H_0^2-V_{\min}(M))\rd_M I
&= -\int_{r_-}^{r_+} \frac{\sqrt{H_0^2 - V_0(r, M)}}{2r^2} dr 
-\int_{r_-}^{r_+} \frac{V_0(r, M)-V_{\min}(M)}{2r^2\sqrt{H_0^2 - V_0(r, M)}}dr, \\
\rd_{H_0} ((H_0^2-V_{\min}(M))\rd_M I)
&= -\int_{r_-}^{r_+} \left[\frac{1}{2r^2}+\frac{d}{dr}\left(\frac{V_0(r, M)-V_{\min}(M)}{r^2V'_0(r, M)}\right)\right]\frac{H_0}{\sqrt{H_0^2 - V_0(r, M)}}dr.
\end{align*}
Here \('\) denotes the \(r\) derivative, and we use \(r_\pm\) to denote \(r_\pm^0(H_0, M)\). In conclusion,
\[
(H_0^2-V_{\min}(M))\rd_M c(H_0, M)
= \int_{r_-}^{r_+}\left[\frac{1}{2r^2}- \frac{d}{dr}\left(\frac{V_0(r, M)-V_{\min}(M)}{r^2V'_0(r, M)}\right)\right]\frac{H_0}{\sqrt{H_0^2 - V_0(r, M)}}dr.
\]

To show that this is less than \(0\), we use the change of variable \(\r(r) = -1/2r\) so that \(\frac{d}{dr} = \frac{1}{2r^2}\frac{d}{d\r}\). 
This gives 
\[
(H_0^2-V_{\min}(M))\rd_M c(H_0, M)
= \int_{\r(r_-)}^{\r(r_+)}\left[1-2\frac{d}{d\r}\left(\frac{V_0(\r, M)-V_{\min}(M)}{V'_0(\r, M)}\right)\right]\frac{H_0}{\sqrt{H_0^2 - V_0(\r, M)}}d\r.
\]
where we abused \(V_0(\r, M) = V_0(r, M)\), hence
\[
V_0(\r, M) = (1+4m\r)(1+4M\r^2).
\]
Now the remaining part is to show that the contribution at two points with the same \(V_0\) value is negative. 
To be precise, we divide the interval \([r_-, r_+]\) into two parts \([r_-, r_*]\) and \([r_*, r_+]\), where \(r_*\) is the point where \(V_0(r, M)\) attains its minimum.
Then we define the change of variable \(\et(\r) = V_0(\r, M) - V_{\min}(M)\) on each interval, so that the above quantity has the following expression:
\begin{align*}
(H_0^2-V_{\min}(M))\rd_M c(H_0, M)
&= -\int_0^{H_0^2-V_{\min}(M)} \left[\frac{1}{V'_0}-\frac{2}{V'_0}\frac{d}{d\r}\left(\frac{V_0-V_{\min}(M)}{V'_0}\right)\right](\r(\et))\frac{H_0}{\g^{1/2}}d\et \\ 
&\qquad
+\int_0^{H_0^2-V_{\min}(M)} \left[\frac{1}{V'_0}-\frac{2}{V'_0}\frac{d}{d\r}\left(\frac{V_0-V_{\min}(M)}{V'_0}\right)\right](\rhotil(\et))\frac{H_0}{\g^{1/2}}d\et \\
&= \int_0^{H_0^2-V_{\min}(M)} \frac{d}{d\r}\left(\frac{V_0-V_{\min}(M)}{V_0'^2}\right)(\r(\et))\frac{H_0}{\g^{1/2}}d\et \\
&\qquad
-\int_0^{H_0^2-V_{\min}(M)} \frac{d}{d\r}\left(\frac{V_0-V_{\min}(M)}{V_0'^2}\right)(\rhotil(\et))\frac{H_0}{\g^{1/2}}d\et.
\end{align*}
Here, \(\r\) and \(\rhotil\) are used to denote the inverse of the change of variable \([\r(r_-), \r(r_*)]\) and \([\r(r_*), \r(r_+)]\) to \([0, H_0^2-V_{\min}(M)]\), respectively. 
Because \(\r \le \rhotil\), it suffice to show that the fraction \(\frac{V_0-V_{\min}(M)}{V_0'^2}\) is convex; this can be checked by direct computation, and this is actually true for any cubic polynomial including our \(V\).

\step{2} (Completing the proof for \(r_0 = r_+^0(H_0, M)\) case)
Because we will take \(\ell\)-derivative, we write this quantity as a function of \(r_0\) and \(\ell\):
\begin{equation}\label{Adef}
\cbar(r_0, \ell) \coloneqq \int_{r_1}^{r_0} \frac{dr}{\big(1-\frac{2m}{r}\big)\sqrt{H_0^2 - (1-\frac{2m}{r})(1+\frac{\ell}{r^2})}},
\end{equation} 
where \(r_1\) is the other root of the equation \(H_0^2 = (1-\frac{2m}{r})(1+\frac{\ell}{r^2})\) so that \(r_1 < r_0\).
Note that we deleted \(H_0\) here; since \(\rd_\ell H_0 > 0\), \(\rd_\ell \cbar > 0\) is a stronger statement than we want. 
To simplify further for this case, we use the change of variable \(\r(r) = r + 2m\log(r-2m)\), then we can write 
\[
\cbar(r_0, \ell) = \int_{\r_1}^{\r_0} \frac{d\r}{\sqrt{V_0(\r_0, \ell)-V_0(\r, \ell)}},
\]
where \(\r_0 = \r(r_0)\), \(\r_1 = \r(r_1)\), with \(V_0\) abused once more.
To avoid picking up the boundary term, we normalize the interval of integral: 
\[
\cbar(r_0, \ell) = \int_0^1 \frac{\r_0-\r_1}{\sqrt{V_0(\r_0, \ell)-V_0(\r(y), \ell)}}dy,
\]
where \(\r(y) = \r_0 + y(\r_1-\r_0)\).
Then, we can compute \(\rd_\ell \cbar\) by differentiating under the integral sign, and show that the integrand increases in \(\ell\), which completes the proof.
Therefore, the inequality that we want to show is, 
\begin{equation}\label{monotoneTarget}
\frac{2\r_1'}{\r_0-\r_1} + \frac{\rd_\ell V_0(\r_0, \ell)-\rd_\ell V_0(\r(y), \ell) - \rd_\r V_0(\r(y), \ell)\r_1'y}{V_0(\r_0, \ell)-V_0(\r(y), \ell)} < 0,
\end{equation}
where \(\r_1'\) denotes \(\rd_\ell \r_1\).

From \(V_0(\r_0, \ell) = V_0(\r_1, \ell)\), we know that
\[
\r_1' = \frac{\rd_\ell V_0 (\r_0, \ell) - \rd_\ell V_0 (\r_1, \ell)}{\rd_\r V_0 (\r_1, \ell)} > 0.
\]
To use the known formula of \(V_0\), we switch back to \(r\); then, the RHS of \eqref{monotoneTarget} can be written as 
\begin{multline*}
\frac{2}{\r_0-\r_1}[\rd_\ell V_0 (r_0, \ell) - \rd_\ell V_0 (r_1, \ell)] \\
+ \frac{1}{V_0(r_0, \ell)-V_0(r(y), \ell)}
[(\rd_\ell V_0(r_0, \ell)-\rd_\ell V_0(r(y), \ell))\rd_r V_0(r_1, \ell)\big(1-\frac{2m}{r_1}\big) \\
- (\rd_\ell V_0(r_0, \ell)-\rd_\ell V_0(r_1, \ell))\rd_r V_0(r(y), \ell)\big(1-\frac{2m}{r(y)}\big) y].
\end{multline*}
Now we use some algebraic identity. 
We set \(x_0 = 1/r_0\), \(x_1 = 1/r_1\), \(x = 1/r(y)\), and \(f(x) = x^2(1-2mx)\). 
Then, \(f\) is a increasing function in our domain of interest \([0, 1/4m]\), and 
\[
\ell = \frac{2m(x_1 - x_0)}{f(x_1)-f(x_0)}, \qquad
V_0(r, \ell) = 1-2mx + \ell f(x), \qquad
\rd_r V_0(r, \ell) = \frac{1}{r^2}(2m-\ell f'(x)),
\]
and most importantly \(\rd_\ell V_0(r, \ell) = f(x)\).
Thus, our target to show is now 
\begin{multline*}
2(f(x_0)-f(x_1))(x-x_0 + \frac{x_1 - x_0}{f(x_1)-f(x_0)}(f(x_0)-f(x))) \\
+ (\r_0-\r_1)(f(x_0)-f(x))f(x_1)(1-\frac{x_1 - x_0}{f(x_1)-f(x_0)}f'(x_1)) \\
- (\r_0-\r(y))(f(x_0)-f(x_1))f(x)(1-\frac{x_1 - x_0}{f(x_1)-f(x_0)}f'(x)) > 0.
\end{multline*}
Dividing this with \((x_0-x)(x-x_1)(x_0-x_1) < 0\), it suffices to show that 
\begin{align*}
&L(x_1, x, x_0) \\
&\coloneqq \frac{1}{x_1-x_0}\left[2\frac{f(x_1)-f(x)}{x_1-x} - 2\frac{f(x)-f(x_0)}{x-x_0}\right] \\
&\qquad + \frac{1}{x_1-x_0}\left[\frac{f(x_1)-f(x)}{x_1-x}f(x_0)\frac{\r-\r_0}{x-x_0}-\frac{f(x)-f(x_0)}{x-x_0}f(x_1)\frac{\r_1-\r}{x_1-x}\right] \\
&\qquad + \frac{x_1-x_0}{f(x_1)-f(x_0)}\frac{1}{x_1-x_0}\left[\frac{f(x)-f(x_0)}{x-x_0}\frac{\r_1-\r}{x_1-x}-\frac{f(x_1)-f(x)}{x_1-x}\frac{\r-\r_0}{x-x_0}\right]f(x)f'(x) \\
&\qquad + \frac{x_1-x_0}{f(x_1)-f(x_0)}\frac{\r_1-\r_0}{x_1-x_0}\frac{f(x)-f(x_0)}{x-x_0}\frac{f(x_1)f'(x_1)-f(x)f'(x)}{x_1-x} < 0.
\end{align*}
We label each line as \(L_1\), \(L_2\), \(L_3\) and \(L_4\) respectively.

Now we analyze the condition we have: when \(\ell\) is large, the differential quotient \((f(x_1)-f(x_0))/(x_1-x_0)\) should be small. 
Since \(f\) has first derivative away from zero for \(x\) away from zero, this means that \(x_0 < x < x_1\) are close to zero.
This means, for any \(\e_0 > 0\), by taking \(\ell\) sufficiently large, we have \(mx_0, mx, mx_1 < \e_0\).
We will regard the term that can be controlled by 
\[
\e_0 \left(\frac{\abs{x_1-x}}{x_1} + \frac{\abs{x_0-x}}{x_0}\right) \eqqcolon \Ecal_0
\]
to be an error, for the reason that will be clear later. 

\header{Estimate of \(L_1\):} This is easy: 
\begin{equation}\begin{aligned}\label{L1est}
L_1(x_1, x, x_0) &= 2 - 4m(x_1+x+x_0) = 2-12mx -4m(x_1-x) + 4m(x-x_0) \\
&= 2-12mx +\Ocal(\Ecal_0).
\end{aligned}\end{equation}

\header{Estimate of \(L_2\):} We need to analyze the term in the bracket. 
We first express each fraction up to first order in \(x_1-x\) and \(x-x_0\).
Here, \(y\) is a placeholder for \(x_0\) or \(x_1\).
\begin{align*}
\frac{f(y)-f(x)}{y-x} &= 2x - 6mx^2 + (y-x)(1-6mx) - (y-x)^2\cdot 2m \\
f(y)\frac{\r(y)-\r(x)}{y-x} &= -1 -(y-x)\frac{1-3mx}{x(1-2mx)} + (y-x)^2e_2(x, y).
\end{align*}
We estimate the second-order error \(e_2\): From the fundamental theorem of calculus, we have 
\[
e_2(x, y) = \frac{1}{(y-x)^3}\int_x^y \frac{1}{2}(y-t)^2 \frac{4m}{t}\frac{2-3mt}{(1-2mt)^2} dt.
\]
For \(y = x_0\), we have 
\[
\abs{e_2(x, x_0)} \le \frac{Cm(\ln x-\ln x_0)}{x-x_0} \le \frac{C\e_0}{x_1x_0},
\]
and for \(y = x_1\), we have
\[
\abs{e_2(x, x_1)} \le \frac{Cm(\ln x_1-\ln x)}{x_1-x} \le \frac{C\e_0}{x_1x}.
\]
Whenever second-order terms are involved we can control them with \(\Ecal_0\), so 
\begin{equation}\label{L2est}
L_2(x_1, x, x_0) = -(1 - 6mx) + \frac{2(1-3mx)^2}{1-2mx} + \Ocal(\Ecal_0)
\end{equation}
Note that the main terms can be written as 
\[
1 - \frac{2mx(1 - 3mx)}{1-2mx}.
\]

\header{Estimate of \(L_3\):} We have 
\[
L_3(x_1, x, x_0) = \frac{x_1-x_0}{f(x_1)-f(x_0)}
\frac{1}{x_1-x_0}\left[\frac{f(x)-f(x_0)}{x-x_0}\frac{\r_1-\r}{x_1-x}-\frac{f(x_1)-f(x)}{x_1-x}\frac{\r-\r_0}{x-x_0}\right]f(x)f'(x).
\]
We need to analyze two parts separately. 
For the first fraction, we want to write it as a perturbation from 
\[
\frac{1}{x_1+x_0} + \frac{3m}{2(1-3mx)}.
\]
The error can be controlled as follows: 
\begin{multline*}
\abs{\frac{1}{(x_1+x_0-2m(x_1^2+x_1x_0+x_0^2))} - \frac{1}{x_1+x_0}-\frac{3m}{2(1-3mx)}} \\
\le \frac{Cm}{x_1^2}\big((x_1-x_0)^2 + 6mx_1^2(\abs{x_1-x}+\abs{x-x_0})\big)
\le \frac{Cm}{x_1}(\abs{x_1-x}+\abs{x-x_0}).
\end{multline*}

For the term in the bracket, we proceed similarly with \(L_2\). 
\begin{align*}
\frac{f(y)-f(x)}{y-x} &= 2x - 6mx^2 + (y-x)(1-6mx) - (y-x)^2\cdot 2m, \\
\frac{\r(y)-\r(x)}{y-x} &= -\frac{1}{x^2} + \frac{y-x}{x^2y} -\frac{2mx}{x^2(1-2mx)} + \frac{m(1-4mx)}{x^2(1-2mx)^2}(y-x) + (y-x)^2e_2(x, y).
\end{align*}
Again we estimate \(e_2\) first: 
\begin{align*}
e_2(x, y) &= \frac{1}{(y-x)^3}\int_x^y \frac{1}{2}(y-t)^2 \left[-\frac{4m}{t^3}-\frac{32m^4}{(1-2mt)^3}\right] dt.
\end{align*}
For \(y = x_0\), we have 
\[
\abs{e_2(x, x_0)} \le \frac{Cm}{x^3(1-x_0/x)^3}\int_{x_0}^x \frac{(x_0-t)^2}{t^3} dt 
\le \frac{Cm}{x^3(1-x_0/x)^3} \left(\ln x_0/x - \frac{2x_0}{x} + \frac{x_0^2}{x^2} +\frac{3}{2}\right)
\le \frac{Cm}{x^2x_0}.
\]
For \(y = x_1\), we have 
\[
\abs{e_2(x, x_1)} \le \frac{Cm}{x_1-x}\int_x^{x_1} \frac{1}{t^3} dt \le \frac{Cm}{x^2x_1}.
\]
The second-order terms can be controlled exactly in the same way as \(L_2\), hence the part that we are considering is 
\[
\frac{1-6mx}{x^2(1-2mx)} + \frac{2m(1-4mx)(1-3mx)}{x(1-2mx)^2} + \frac{2(1-3mx)}{x_1x_0} + \frac{(x_1-x)(x-x_0)(1-6mx)}{x_1x^2x_0} + l.o.t.
\]
The lower order terms can be controlled without cancellations: This is controlled by 
\[
\frac{\e_0}{x^2}\big(\frac{\abs{x-x_0}}{x_0} + \frac{\abs{x_1-x}}{x_1}\big) = \frac{1}{x^2}\Ecal_0,
\]
which will be explained in a moment. 
Actually, this can control some part of the main terms as well: we can write 
\begin{align*}
&\frac{1-6mx}{x^2(1-2mx)} + \frac{2m(1-4mx)(1-3mx)}{x(1-2mx)^2} + \frac{2(1-3mx)}{x_1x_0} + \frac{(x_1-x)(x-x_0)(1-6mx)}{x_1x^2x_0} + \frac{1}{x^2}\Ocal(\Ecal_0) \\
&=\frac{x_1+x+x_0}{x_1xx_0}-\frac{4m}{x(1-2mx)} + \frac{2m(1-4mx)(1-3mx)}{x(1-2mx)^2} - \frac{6m}{x} + \frac{1}{x^2}\Ocal(\Ecal_0)
\end{align*}
Note that we used the identity 
\[
\frac{1}{x_1x_0}-\frac{1}{x^2} = \frac{-x_0(x_1-x)+x(x-x_0)}{x_1x^2x_0}.
\]

We can multiply two easy fractions first:
\begin{align*}
&\frac{x_1-x_0}{f(x_1)-f(x_0)}f(x)f'(x) \\
&= \left(\frac{1}{x_1+x_0}+\frac{3m}{2(1-3mx)} + \frac{\e_0}{x_1x}\Ocal(\abs{x_1-x}+\abs{x-x_0})\right)2x^3(1-2mx)(1-3mx) \\
&= \frac{2x^3}{x_1+x_0} + x^2(-5mx + 6m^2x^2) + 3mx^3(1-2mx) + x^2\Ocal(\Ecal_0) \\
&= \frac{2x^3}{x_1+x_0} -2mx^3 + x^2\Ocal(\Ecal_0).
\end{align*}
Now, we first figure out that the terms involving error terms are controlled; the only problematic term is \(\frac{x_1+x}{x_1xx_0}\).
This can be written as \(\frac{x_1+x}{x_1x^2} + \frac{x_1+x}{x_1x^2x_0}(x-x_0)\), then this is good.
For the main term, the product of two first terms gives 
\[
\frac{2x^2(x_1+x+x_0)}{x_1x_0(x_1+x_0)}.
\] 
Then, we can check that 
\[
mx^3\left(\frac{x_1+x+x_0}{x_1xx_0}-\frac{3}{x^2}\right), \quad
\frac{m}{x}\left(\frac{2x^3}{x_1+x_0}-x^2\right),
\]
are \(\Ocal(\Ecal_0)\).
Clearly the second one is controllable, and for the first one we can write it as 
\[
\frac{mx}{x_1x_0}[-2x_0(x_1-x) + (x+x_1)(x-x_0)],
\]
which shows that it is also controllable.
Therefore, 
\begin{equation}\label{L3est}
L_3(x_1, x, x_0) = \frac{2x^2(x_1+x+x_0)}{x_1x_0(x_1+x_0)} + \frac{-14mx + 30m^2x^2}{1-2mx} + \Ocal(\Ecal_0).
\end{equation}

\header{Estimate of \(L_4\):} This term is the most dangerous one due to the asymmetry. It is given by 
\[
L_4(x_1, x, x_0) = \frac{x_1-x_0}{f(x_1)-f(x_0)}\frac{\r_1-\r_0}{x_1-x_0}\frac{f(x)-f(x_0)}{x-x_0}\frac{f(x_1)f'(x_1)-f(x)f'(x)}{x_1-x}.
\]
The polynomial factors can be written as follows: 
\begin{align*}
\frac{f(x)-f(x_0)}{x-x_0} &= 2x - 6mx^2 - (x-x_0) + \frac{xx_0}{x_1}\Ocal(\Ecal_0) &\eqqcolon I_1, \\
\frac{f(x_1)f'(x_1)-f(x)f'(x)}{x_1-x} &= 6x^2 - 40mx^3 + 60m^2x^4 + (x_1-x)(2x_1+4x) + x_1^2\Ocal(\Ecal_0) &\eqqcolon I_2.\\
\intertext{Also, we can reuse the estimate for \(\frac{x_1-x_0}{f(x_1)-f(x_0)}\) from \(L_3\):}
\frac{x_1-x_0}{f(x_1)-f(x_0)} &= \frac{1}{x_1+x_0} + \frac{3m}{2(1-3mx)} + \frac{1}{x_1}\Ocal(\Ecal_0) &\eqqcolon I_3.\\
\intertext{For the \(\r\) term, we claim the following:}
\frac{\r_1-\r_0}{x_1-x_0} &= -\frac{1}{x_1x_0} - \frac{2m}{x(1-2mx)} + \frac{1}{x_1x}\Ocal(\Ecal_0) &\eqqcolon I_4.
\end{align*}

We now multiply these four terms.
First of all, we can see that the lower order term of \(I_4\) is controllable: we have \(\abs{I_1I_2I_3} \le C x\cdot x_1^2\cdot \frac{1}{x_1}\).
Then, we can write the main part of \(I_4\) in the following way: 
\[
-\frac{1}{x_1x_0} - \frac{2m}{x(1-2mx)} = -\frac{1}{x_1x} - \frac{2m}{x(1-2mx)} - \frac{x-x_0}{x_1xx_0}.
\]
This allows us to control terms from the lower order term of \(I_2\): we have \(\abs{I_1I_3} \le Cx\cdot \frac{1}{x_1}\), so both \(1/x_1x\) term and \(\frac{1}{x_1x}\frac{\abs{x-x_0}}{x_0}\) term are controllable.
Similarly, the lower order terms of \(I_1\) and \(I_3\) are also controllable.

Therefore, we can write \(L_4\) as 
\[
L_4 = I_1'I_2'I_3'I_4' + \Ocal(\Ecal_0), \text{ where }
\]
\begin{align*}
I_1' &= (x+x_0) - 6mx^2, 
& I_2' &= (2x^2+2xx_1+2x_1^2) - 40mx^3 + 60m^2x^4, \\
I_3' &= \frac{1}{x_1+x_0} + \frac{3m}{2(1-3mx)}, 
& I_4' &= -\frac{1}{x_1x_0} - \frac{2m}{x(1-2mx)} = -\frac{1}{x_1x} - \frac{2m}{x(1-2mx)} - \frac{x-x_0}{x_1xx_0}.
\end{align*}
Again, the main term is fine; the product of first terms gives 
\[
\frac{2(x+x_0)(x_1^2+x_1x+x^2)}{x_1x_0(x_1+x_0)},
\] 
and the other terms can be treated as \((28mx-60m^2x^2)/(1-2mx)\) with controllable error. 
For the terms in \(I_1'I_2'I_3'I_4'\), we first control the effect of \(\frac{x-x_0}{x_1xx_0}\) in \(I_4'\): This multiplied with the other terms is bounded by \(\Ecal_0\). 
After that the other terms can be controlled straightforwardly, to give the final conclusion
\begin{equation}\label{L4est}
L_4(x_1, x, x_0) = -\frac{2(x+x_0)(x_1^2+x_1x+x^2)}{x_1x_0(x_1+x_0)} + \frac{28mx - 60m^2x^2}{1-2mx} + \Ocal(\Ecal_0).
\end{equation}

Therefore, combining \eqref{L1est}, \eqref{L2est}, \eqref{L3est} and \eqref{L4est}, we have
\[
L(x_1, x, x_0) = -\frac{(2x_1x+2xx_0-x_1x_0-3x_0^2)}{x_0(x_1+x_0)} + \Ocal(\Ecal_0) < 0,
\]
where we used the fact that \(x_0 < x < x_1\), so that 
\[
\frac{(2x_1x+2xx_0-x_1x_0-3x_0^2)}{x_0(x_1+x_0)} 
= \frac{1}{x_1+x_0}(x_1-x) + \frac{2x_1+3x_0}{x_0(x_1+x_0)}(x-x_0) \ge C\Ecal_0
\]
for sufficiently large \(\ell\).
This completes the proof of the first case.

(2) For this case, we need to use the following result of Rioseco--Sarbach \cite{RS18}:
\begin{lem}[\cite{RS18}, (24)]\label{lem:SarbachCite}
    For \(I\) given in \eqref{Areadef}, we have the following expansion: 
    \[
    c(H_0, M) = \frac{\rd I}{\rd H_0} = \frac{\pi mP^2}{\sqrt{P-6}}\left[1 + \frac{3}{4}\frac{2P^3-32P^2+165P-266}{(P-2)(P-6)^2}e^2 + \Ocal(e^4)\right].
    \]
    Here, \(P\) and \(e\) are defined as follows: 
    for \(\x_\pm = r_\pm^0(H_0, M) / m\),
    \[
    \x_- = \frac{P}{1+e}, \qquad
    \x_+ = \frac{P}{1-e}.
    \]
\end{lem}
We take \(\d(m, \Mbar)\) so that we can fit in the region where the above expansion in \(e\) is valid with a constant \(C>0\). 
We regard \(P\) and \(e\) as functions of \(r, w, \ell\) then differentiate them in \(\ell\). 
To use the above expansion, we first show that 
\[
d \le \abs{\rd_\ell P} \le D, \qquad \abs{\rd_\ell e} \le D.
\]
for constants \(d, D > 0\).

We have 
\[
P = \frac{2\x_+\x_-}{\x_++\x_-}, \qquad
e = \frac{\x_+-\x_-}{\x_++\x_-}.
\]
From \(H^2 = (1-2/\x_\pm)(1+\l/\x_\pm^2)\) with \(\l = \ell/m^2\), we have 
\[
\big(1-\frac{2m}{r}\big)\frac{1}{r^2}
= \big(1-\frac{2m}{r_\pm}\big)\frac{1}{r_\pm^2}
 + \rd_\ell \x_\pm\big(\frac{2}{\x_\pm^2}-\frac{2\l}{\x_\pm^3} + \frac{6\l}{\x_\pm^4}\big),
\]
hence 
\[
\frac{\rd_\ell \x_\pm}{\x_\pm}
= \frac{\rd_\ell V_0(r, \ell)-\rd_\ell V_0(r_\pm, \ell)}{r_\pm\rd_r V_0(r_\pm, \ell)}.
\]
By fixing \(\Mbar\) we are restricting \(r_\pm\) within a compact set with \(r_\pm \ge 10m\), in particular \(\rd_\ell\rd_r V_0\) and \(\rd_{rr}^2 V_0\) are both away from zero.
This proves that there are constants \(d_1, D_1\) with, as before, \(r_*\) denoting the point where \(V_0(r, \ell)\) attains its minimum,
\[
\frac{\rd_\ell \x_\pm}{\x_\pm} \ge d_1\frac{\abs{r-r_\pm}}{\abs{r_+-r_-}}, \qquad
\frac{\rd_\ell \x_\pm}{\x_\pm} \le D_1\frac{\abs{r-r_\pm}}{\abs{r_\pm-r_*}}.
\]

Now we can estimate the \(\ell\)-derivative of \(P\) and \(e\) to conclude. 
\begin{align*}
\rd_\ell P &= \frac{2\x_+^2\x_-^2}{(\x_++\x_-)^2}\left(\frac{1}{\x_+^2}\rd_\ell\x_+ + \frac{1}{\x_-^2}\rd_\ell\x_-\right), \\
\rd_\ell e &= \frac{2\x_+\x_-}{(\x_++\x_-)^2}\left(\frac{1}{\x_+}\rd_\ell\x_+ - \frac{1}{\x_-}\rd_\ell\x_-\right).
\end{align*}
Because the second derivative of the potential curve \(\rd_{rr}^2 V_0\) is bounded and away from zero in this regime, we have 
\[
\abs{r_+-r_*}\sim \abs{r_*-r_-}.
\]
This with above estimates proves the desired claim. 

Now we can differentiate the expansion in Lemma \ref{lem:SarbachCite} in \(\ell\) to conclude that 
\[
\rd_\ell (c(H_0, M)) = \frac{3\pi m P(P-8)}{2(P-6)^{3/2}}\rd_\ell P + \Ocal(e).
\]
Since the main part is positive and away from zero in our regime where \(\x_{\pm}\ge 10\), we can take \(\d\) sufficiently small so that \(\Ocal(e)\) term is small compared to the main term, hence we have \(\rd_\ell c(H_0, M) > 0\).
This completes the proof.
\end{proof}

Because our linearized transport operator \(\Dfr^\lin_0\) is given by \(\rd_t + \O_{\Sch}(H_0, M)\rd_{Q}\) (this will be explained later), we need to define the vector field \(Z\) which commutes with this. 
\begin{equation}\label{Zdef}
Z(t, Q, H_0, M) \coloneqq \rd_M \O_{\Sch}(H_0, M)\rd_{H_0} - \rd_{H_0} \O_{\Sch}(H_0, M)\rd_M.
\end{equation}
This gives \(Z(\O_{\Sch}(H_0, M)) = 0\), and we will see that \(\O\) defined with perturbed metric still has small \(Z\) derivative. 
Therefore, we expect that \(Z^k f\) does not grow in time, which is crucial for the nonlinear estimates.

\section{Linear Estimates}\label{sec:LinEst}
\subsection{Coordinate defined by action angle variables}
We want to use the coordinate system \((t = s, Q, H_0, M)\) which transforms the Vlasov equation \eqref{eq:vlasov-lin} into a better form. 
\(H_0, M\) are already defined by \eqref{H0def}, so we will define \(Q\) here, and show that this is indeed a coordinate system on the support of \(f\) for all time. 
Note that we use \(t\) whenever we are using the action angle coordinates, in order to distinguish two coordinate vectors \(\rd_t\) and \(\rd_s\). 
Although these two vectors are the same for the (non-dynamical) angle variables that we are using in the linear estimates, they will be different for the dynamical angle variables that we will use in the nonlinear estimates, hence we want to keep this notation consistent.
\begin{defn}[Action angle \(Q\)]\label{Qdef1} Let \(H_0, M\) satisfy the condition given in Definition \ref{Omegadef1}, and set \(R_-^0(H_0, M)\) and \(R_+^0(H_0, M)\) as there. 
For \(R\in [R_-^0(H_0, M), R_+^0(H_0, M)]\), set 
\[
\Qtil(R, H_0, M) = \int_{R_-^0(H_0, M)}^{R} \frac{A_0(R')H_0}{\a_0(R')^2w(R', H_0, M)} dR',
\]
where \(R_-^0(H_0, M)\) and \(w(R, H_0, M)\ge 0\) are defined in Definition \ref{Omegadef1}. Then, we define the action angle \(Q\) by
\begin{equation}\label{Qdef}
Q(R, w, \ell) \coloneqq \sgn(w)\O_{\Sch}\Big(H_0(R, w, \ell), M = \ell\Big)\Qtil\Big(R, H_0(R, w, \ell), M = \ell\Big).
\end{equation}
Note that \((H_0, M)\) given by \((R, w, \ell)\) automatically satisfy the condition in Definition \ref{Omegadef1}, hence the above definition is well-defined for all \((R, w, \ell)\) in the support of \(f\).
\end{defn}
Note that \(Q(R, w, \ell)\) defined here is not globally smooth, as there is a discontinuity on \(\Qtil\) across \(R = R_+^0(H_0, M)\).
Namely, \(Q\) jumps from \(-\pi\) to \(\pi\) across these points.
This situation is just the same as the discontinuity occurring in the usual polar coordinate; especially, the coordinate vector \(\rd_Q\) is smooth in the region that we are interested in, so the discontinuity does not cause any issue for our analysis.

Now, we prove that this is a valid coordinate system on \(\Scal_{\cfr}\).
We need to compute the Jacobian of the transformation. 
\begin{prop}[\((t, Q, H_0, M)\) is a coordinate system]\label{prop:anglevar}
    On \(\RR_s\times (\Scal_{\cfr}\setminus \{R = R_+^0(H_0(R, w, \ell), M = \ell)\})\), the transformation \((s, R, w, \ell) \mapsto (t, Q, H_0, M)\) is a smooth diffeomorphism. 
    The Jacobian determinant of this change of variables map is given by \(A_0(R)\O_{\Sch}(H_0, M)\).
\end{prop}
\begin{proof}
Clearly \(t, H_0, M\) are infinitely differentiable in \((s, R, w, \ell)\) coordinate. 
The proof that \(Q\) is smooth in \((s, R, w, \ell)\) coordinate is standard, and in particular follows as a subcase of Proposition \ref{prop:regDAAV} with metric coefficients taken from Schwarzschild and arbitrarily large \(N\), hence we omit the details here.

Then, we compute the Jacobian matrix of the transformation. We have, 
\[
\begin{bmatrix}
\rd_s t & \rd_s Q & \rd_s H_0 & \rd_s M \\
\rd_R t & \rd_R Q & \rd_R H_0 & \rd_R M \\
\rd_w t & \rd_w Q & \rd_w H_0 & \rd_w M \\
\rd_\ell t & \rd_\ell Q & \rd_\ell H_0 & \rd_\ell M
\end{bmatrix}
=
\begin{bmatrix}
1 & 0 & 0 & 0 \\
0 & \rd_R|_{(s, R, w, \ell)} Q & \rd_R (\a_0 E) & 0 \\
0 & \rd_w|_{(s, R, w, \ell)} Q & \frac{\a_0 w}{E} & 0 \\
0 & \rd_\ell|_{(s, R, w, \ell)} Q & \frac{\a_0}{2A_0^2R^2E} & 1
\end{bmatrix}.
\]
We make use of \(\rd_{H_0}|_{(R, H_0, M)}\) and \(\rd_R|_{(R, H_0, M)}\) to compute the Jacobian.
Whenever \(R\in (R^0_-(H_0, M), R^0_+(H_0, M))\), the vectors \(\rd_R|_{(R, H_0, M)}\) and \(\rd_{H_0}|_{(R, H_0, M)}\) are well-defined, and the following calculation makes sense:
From the definition of \(\Qtil\) in Definition \ref{Qdef1}, we have 
\begin{align*}
\rd_R|_{(R, H_0, M)}\Qtil &= \frac{A_0(R)H_0}{\a_0(R)^2w(R, H_0, M)}, \\
\rd_R|_{(R, H_0, M)}Q &= \sgn(w)\O_{\Sch}(H_0, M)\frac{A_0(R)H_0}{\a_0(R)^2w(R, H_0, M)}.
\end{align*}
From the chain rule, we have 
\begin{align*}
\rd_R Q &= \rd_R H_0 \cdot \rd_{H_0}|_{(R, H_0, M)} Q + \rd_R|_{(R, H_0, M)} Q, \\
\rd_w Q &= \rd_w H_0 \cdot \rd_{H_0}|_{(R, H_0, M)} Q,
\end{align*}
hence the Jacobian determinant is given by
\begin{align*}
\abs{\frac{\rd(t, Q, H_0, M)}{\rd(s, R, w, \ell)}}
&= \rd_w H_0 \cdot \rd_R|_{(R, H_0, M)} Q \\
&= \frac{\a_0 w}{E} \cdot \sgn(w)\O_{\Sch}(H_0, M)\frac{A_0(R)H_0}{\a_0(R)^2w(R, H_0, M)} \\
&= A_0(R)\O_{\Sch}(H_0, M).
\end{align*}
Note that the Jacobian determinant is continuous in its variables, hence this identity is true even when \(R = R^0_\pm(H_0, M)\).
Since this never touches zero, the transformation is a diffeomorphism.
\end{proof}

Now, we rewrite the linear Vlasov equation \eqref{eq:vlasov-lin} in the \((t, Q, H_0, M)\) coordinate system.
\begin{prop}[Vlasov equation in \((t, Q, H_0, M)\) coordinate]\label{prop:linvlasoveqn}
    The Vlasov equation \eqref{eq:vlasov-lin} can be rewritten in the \((t, Q, H_0, M)\) coordinate system as 
    \begin{equation}\label{eq:linvlasovA}
    \rd_t f + \O_{\Sch}(H_0, M)\rd_Q f = 0.
    \end{equation}
\end{prop}
\begin{proof}
It is clear that \(\Dfr^\lin_0 t = 1\) and \(\Dfr^\lin_0 H_0 = \Dfr^\lin_0 M = 0\). For \(\Dfr^\lin_0 Q\), we again use the vectors \(\rd_R|_{(R, H_0, M)}\), \(\rd_{H_0}|_{(R, H_0, M)}\), and \(\rd_M|_{(R, H_0, M)}\) to compute. We have 
\[
\Dfr^\lin_0 Q = \Dfr^\lin_0 R \cdot \rd_R|_{(R, H_0, M)} Q = \frac{\a_0 w}{A_0 E}\O_{\Sch}(H_0, M)\frac{A_0H_0}{\a_0^2w} = \O_{\Sch}(H_0, M).
\]
This completes the proof. 
\end{proof}

\subsection{Proof of Theorem \ref{thm:lin}} 
Now we prove Theorem \ref{thm:lin}.
\begin{proof}[Proof of Theorem \ref{thm:lin}]    
Since \eqref{eq:linvlasovA} is a linear transport equation, we have a clear solution representation in this coordinate system. This allows us to prove Theorem \ref{thm:lin} by the non-stationary phase argument.
\label{subsec:lin-decay}

Let \(f\) be a solution of the linear Vlasov equation \eqref{eq:vlasov-lin} with initial data \(f_\init\) satisfying the support condition. 
Then, in the action angle coordinate system, by Proposition \ref{prop:linvlasoveqn}, we have the following explicit solution representation in the \((t, Q, H_0, M)\) coordinate system: 
\begin{equation}\label{explsol}
f(t, Q, H_0, M) = f_\init(Q - t\O_{\Sch}(H_0, M), H_0, M).
\end{equation}
Then, the estimates \eqref{linthm-f1} and \eqref{linthm-f2} follow immediately:
For (1), because the coordinate transformation and \(\O_{\Sch}\) are smooth in \(H_0\) and \(M\), each derivative of \(f\) in \((t, Q, H_0, M)\) gives only a constant times \(t\) when the derivative hits the phase shift, so we have \eqref{linthm-f1}.
Note that this was possible as \(f\) is \(C^N\) and satisfies the size condition \eqref{linthm-init}.
For (2), we need to understand the \(Z\) derivative of \(f\). Since \(Z\O_{\Sch}(H_0, M) = 0\), we have 
\begin{multline*}
Z(f_\init(Q - t\O_{\Sch}(H_0, M), H_0, M)) \\
= (\rd_M\O_{\Sch}(H_0, M) \rd_{H_0} f_\init-\rd_{H_0}\O_{\Sch}(H_0, M) \rd_M f_\init)(Q - t\O_{\Sch}(H_0, M), H_0, M),
\end{multline*}
in particular this does not cause growth in \(t\). 
When we apply \(N\) derivatives, we will see at most \(N\)-th derivative of \(\O\) and \(N\)-th derivative of \(f_\init\), which are all bounded by a constant times \(B_\init\) by the assumption \eqref{linthm-init}, hence we have \eqref{linthm-f2}.

Now we prove (3); this is the heart of the proof. Note that the integral that we have to understand is of the form 
\[
I(s, R) = \frac{\pi}{R^2A_0(R)^2}\int_{-\infty}^{\infty}\int_0^\infty f(s, Q, H_0, M) \bar{\CC}_0(R, w, \ell) d\ell dw,
\] 
\[
\rd_s I(s, R) = \frac{\pi}{R^2A_0(R)^2}\int_{-\infty}^{\infty}\int_0^\infty (\rd_s f)(s, Q, H_0, M) \bar{\CC}_0(R, w, \ell) d\ell dw,
\]
where \(\bar{\CC}_0 \in \{E, -1/E, w, w^2/E\}\) is a function which is known to be regular, and \(f\) is given by the above explicit formula.
We use the Fourier expansion of \(f_\init\) in \(Q\) variable, which is given by 
\[
f_\init(Q, H_0, M) = \sum_{k\in \ZZ} e^{ikQ} \wh{(f_\init)}_k(H_0, M).
\]
Then, \(f\) and \(\rd_s f\) are written as 
\[
f(s, Q, H_0, M) = \sum_{k\in \ZZ} e^{ikQ}e^{-iks\O_{\Sch}(H_0, M)} \wh{(f_\init)}_k(H_0, M),
\]
\[
\rd_s f(s, Q, H_0, M) = \sum_{k\neq 0} (-ik\O_{\Sch}(H_0, M))e^{ikQ}e^{-iks\O_{\Sch}(H_0, M)} \wh{(f_\init)}_k(H_0, M).
\]
Hence, we get 
\[
\rd_s I(s, R) = \frac{\pi}{R^2A_0(R)^2}\sum_{k\neq 0} \G_k (s, R),
\]
where 
\[
\G_k (s, R) = \int_{-\infty}^{\infty}\int_0^\infty (-ik\O_{\Sch}(H_0, M))e^{ikQ}e^{-iks\O_{\Sch}(H_0, M)} \wh{(f_\init)}_k(H_0, M) \bar{\CC}_0(R, w, \ell) d\ell dw.
\]
Therefore, it suffices to prove the time decay estimate of \(\rd_R^i \G_k(s, R)\), based on the non-stationary phase argument with 
\[
e^{-iks\O_{\Sch}(H_0, M)} = \left(\frac{1}{(-iks\rd_\ell(\O_{\Sch}(H_0, M)))}\rd_\ell\right) e^{-iks\O_{\Sch}(H_0, M)}.
\]
Note that \(\rd_\ell(\O_{\Sch}(H_0, M))\) is away from zero on our support by Proposition \ref{prop:monotonic}, hence the above integration by parts is valid.
For the integration by parts, when \(\rd_\ell\) hits \(1/\rd_\ell(\O_{\Sch}(H_0, M))\) it produces higher order derivatives of \(\O_{\Sch}\), but all of them are bounded by a constant. 
We can control the integral as below, where the implicit constant depends on \(m, \cfr, N\) but not on \(k\), \(s\), or \(f_\init\). 
Here \(\CC_0\) is a generic function with bounded derivatives, which may change from line to line:
\begin{equation}\label{dRiGk}
\begin{aligned}
&\abs{\rd_R^i \G_k(s, R)} \\
&\lesssim \sum_{i_1 + \cdots + i_4 \le i}\abs{\int_{-\infty}^{\infty}\int_0^\infty k^{1+i_1+i_2}t^{i_2}e^{ikQ}e^{-iks\O_{\Sch}(H_0, M)} \wh{(\rd_{H_0}^{i_3}\rd_M^{i_4}f_\init)}_k(H_0, M) \CC_0(R, w, \ell) d\ell dw} \\
&\lesssim \sum_{\substack{i_1 + \cdots + i_4 \le i \\ j_1 + j_2 + j_3 \le N-i+i_2}} 
\frac{\abs{k}^{1+i_1}}{\abs{ks}^{N-i}} \int_{-\infty}^{\infty}\int_0^\infty \abs{k}^{j_1}\abs{\wh{(\rd_{H_0}^{i_3+j_2}\rd_M^{i_4+j_3}f_\init)}_k(H_0, M)} d\ell dw \\
&\lesssim \abs{s}^{-N+i}\int_{-\infty}^{\infty}\int_0^\infty \sum_{\substack{i_1 + \cdots + i_4 \le i \\ j_1 + j_2 + j_3 \le N-i+i_2}} 
\abs{k}^{1+i_1+j_1+i-N}\abs{\wh{(\rd_{H_0}^{i_3+j_2}\rd_M^{i_4+j_3}f_\init)}_k(H_0, M)} d\ell dw.
\end{aligned}
\end{equation}
For the penultimate inequality, we integrated by parts \(N-i+i_2\) times. Note that the exponent of \(k\) satisfies 
\(1 + i_1 + j_1 + i - N \le 1 + i_1 + i_2 \le N-1\) as \(i \le N-2\), and we also have 
\[
(2+i_1 +j_1+i-N) + (i_3+j_2) + (i_4+j_3) \le 2 + i \le N.
\]
Hence, for \(\a_1 = \max{2+i_1+j_1+i-N, 0}\) we have 
\begin{align*}
\abs{\rd_R^i \G_k(s, R)}
&\lesssim \abs{s}^{-N+i}\int_{-\infty}^{\infty}\int_0^\infty \sum_{\substack{i_1 + \cdots + i_4 \le i \\ j_1 + j_2 + j_3 \le N-i+i_2}} 
\frac{1}{k}\abs{\wh{(\rd_Q^{\a_1}\rd_{H_0}^{i_3+j_2}\rd_M^{i_4+j_3}f_\init)}_k(H_0, M)} d\ell dw \\
&\le \abs{s}^{-N+i}\int_{-\infty}^{\infty}\int_0^\infty \sum_{\a_1 + \a_2 + \a_3 \le N} 
\frac{1}{k}\abs{\wh{(\rd_Q^{\a_1}\rd_{H_0}^{\a_2}\rd_M^{\a_3}f_\init)}_k(H_0, M)} d\ell dw
\end{align*}
As we have a trivial bound for \(s\le 1\), from the second line above, 
\[
\abs{\rd_R^i \G_k(s, R)} \lesssim \sum_{i_1+\cdots+i_4\le i}\int_{-\infty}^{\infty}\int_0^\infty 
\frac{1}{k}\abs{\wh{(\rd_Q^{2+i_1+i_2}\rd_{H_0}^{i_3}\rd_M^{i_4} f_\init)}_k(H_0, M)} d\ell dw,
\]
we can write the above as 
\[
\abs{\rd_R^i \G_k(s, R)}
\le \jap{s}^{-N+i}\int_{-\infty}^{\infty}\int_0^\infty \sum_{\a_1 + \a_2 + \a_3 \le N} 
\frac{1}{k}\abs{\wh{(\rd_Q^{\a_1}\rd_{H_0}^{\a_2}\rd_M^{\a_3}f_\init)}_k(H_0, M)} d\ell dw
\]
Now we take a summation over \(k\). By the Cauchy--Schwarz inequality, we have
\begin{align*}
&\sum_{k\neq 0} \abs{\rd_R^i \G_k(s, R)} \\
&\lesssim \jap{s}^{-N+i}\int_{-\infty}^{\infty}\int_0^\infty \sum_{\a_1 + \a_2 + \a_3 \le N} 
\Big(\sum_{k\neq 0} \frac{1}{k^2}\Big)^{1/2}\Big(\sum_{k\neq 0} \abs{\wh{(\rd_Q^{\a_1}\rd_{H_0}^{\a_2}\rd_M^{\a_3}f_\init)}_k(H_0, M)}^2\Big)^{1/2} d\ell dw \\
&\lesssim \jap{s}^{-N+i}\int_{-\infty}^{\infty}\int_0^\infty \sum_{\a_1 + \a_2 + \a_3 \le N} 
\norm{\rd_Q^{\a_1}\rd_{H_0}^{\a_2}\rd_M^{\a_3}f_\init(\cdot, H_0, M)}_{L^2_Q} d\ell dw \\
&\lesssim \jap{s}^{-N+i}\sum_{\a_1 + \a_2 + \a_3 \le N} \sup_{Q, H_0, M}
\abs{\rd_Q^{\a_1}\rd_{H_0}^{\a_2}\rd_M^{\a_3}f_\init(Q, H_0, M)}.
\end{align*}
Since the change of variables \((R, w, \ell)\) to \((Q, H_0, M)\) is sufficiently regular, we can use the bound on the initial data \eqref{linthm-init}. This gives us the desired decay estimate \eqref{linthm-density}.

Notice that the \((H_0, M)\) support of the solution \eqref{explsol} remains unchanged over time.
In particular, the support in \((R, w, \ell)\) is still bounded uniformly in time. 
We fix \(R_{\textrm{out}} > 0\) to be a number that \(f(s, R, w, \ell) = 0\) for \(R > R_{\textrm{out}}\), which forces \(\r, \tr T, j, S_R \equiv 0\) there. 

For (4), we estimate the time derivative of the metric coefficient by using the known estimates on the time derivative of energy-momentum tensor. 
Taking \(\rd_s\) on \eqref{eq:linODE}, we get 
\begin{equation}\label{dslinODE}
\rd_s\vec{X}_\lin'(R) = \vec{F}_{\vec{X}}(\vec{X}_0(R), 0, R)\rd_s\vec{X}_\lin(R)
+ \vec{F}_{\vec{S}}(\vec{X}_0(R), 0, R)\rd_s\vec{S}(R).
\end{equation}
Note that \(\vec{X}_0\), metric coefficients of Schwarzschild spacetime in our foliation, is a function of \(R\) only. 
Therefore, the coefficients \(\vec{F}_{\vec{X}}(\vec{X}_0(R), 0, R)\) and \(\vec{F}_{\vec{S}}(\vec{X}_0(R), 0, R)\) are both smooth functions of \(R\), of which derivatives are bounded in the domain we consider. 
Therefore, this is a linear ODE of \(\rd_s\vec{X}_{\lin}\), with the boundary condition \(\rd_s\vec{X}_{\lin} \equiv 0\) for large \(R\). 

To get the desired estimate, first we use the estimate on \(\rd_s \vec{S}\). 
Using Gronwall's inequality on \(R\in [m/4, R_{\textrm{out}}]\), we obtain that \(\norm{\vec{X}}_{C^0}\le CB_\init\jap{s}^{-N}\).
Plugging this back to the equation, we also have \(\norm{\vec{X}}_{C^1}\le CB_\init\jap{s}^{-N}\), which proves the desired estimate for \(i\le 2\). 
Note that, the second derivative of \(\b_\lin\) can be controlled using the equation on \(\b_\lin\), which is the linearized version of \eqref{eq:MeanCurv}.

For \(i > 2\), we use induction. Say we have desired estimates for \(i < l\), where \(3\le l\le N\). 
Then, by differentiating the above ODE \(l-2\) times, we get the following: 
\[
\rd_s \vec{X}^{(l-1)}(R) = \vec{F}_{\vec{X}}(\vec{X}_0(R), 0, R)\rd_s\vec{X}_\lin^{(l-2)}(R)
+ \vec{F}_{\vec{S}}(\vec{X}_0(R), 0, R)\rd_s\vec{S}^{(l-2)}(R)
+ \text{ lower order terms }.
\]
Here, the lower order terms are already controlled by \(CB_\init \jap{t}^{-N + l-3}\).
Therefore, from the estimates that we have on \(\rd_s\vec{X}_{\lin}^{(i)}(R)\) with \(i\le l-2\) and \(\rd_s\vec{S}^{(l-2)}\), we get 
\[
\sup_R\abs{\rd_R^l\rd_s(\a_\lin, A_\lin, \b_\lin)(s, R)} \le CB_\init \jap{s}^{-N+l-2},
\]
which is the desired estimate.
This completes the proof of (4). 

Finally, we prove (5). From the proof of (3), the integral representing the energy-momentum tensor component can be written as 
\[
I(s, R) = \frac{\pi}{R^2A_0(R)^2} 
\sum_{k \in \ZZ} \G^{(0)}_k(s, R),
\]
where 
\[
\G^{(0)}_k(s, R) = \int_{-\infty}^{\infty}\int_0^\infty e^{ikQ}e^{-iks\O_{\Sch}(H_0, M)} \wh{(f_\init)}_k(H_0, M) \bar{\CC}_0(R, w, \ell) d\ell dw.
\]
Then, we define the limiting object \((\widetilde{\rho}, \widetilde{\tr T}, \widetilde{j}, \widetilde{S}_R)\) to be 
\[
(\widetilde{\rho}, \widetilde{\tr T}, \widetilde{j}, \widetilde{S}_R)(s, R) = \frac{\pi}{R^2A_0(R)^2} \G^{(0)}_0(s, R),
\]
where \(\G^{(0)}_0\) is written with corresponding \(\bar{\CC}_0\).
Then, what we have to show is that 
\[
\sum_{k\neq 0} \abs{\rd_R^i \G^{(0)}_k(s, R)} \le CB_\init \jap{s}^{-N+i}, \qquad \text{ for } i\le N-1.
\]

From the same non-stationary phase argument as in \eqref{dRiGk}, we have 
\[
\abs{\rd_R^i \G^{(0)}_k(s, R)}
\lesssim \abs{s}^{-N+i}\int_{-\infty}^{\infty}\int_0^\infty \sum_{\substack{i_1 + \cdots + i_4 \le i \\ j_1 + j_2 + j_3 \le N-i+i_2}} 
\abs{k}^{i_1+j_1+i-N}\abs{\wh{(\rd_{H_0}^{i_3+j_2}\rd_M^{i_4+j_3}f_\init)}_k(H_0, M)} d\ell dw.
\]
We have \((i_3+j_2)+(i_4+j_3)\le N\) together with 
\[
(1 + i_1+j_1+i-N) + (i_3+j_2)+(i_4+j_3) \le 1 + i \le N
\]
for \(i\le N-1\), so we have 
\[
\abs{\rd_R^i \G^{(0)}_k(s, R)}
\le \abs{s}^{-N+i}\int_{-\infty}^{\infty}\int_0^\infty \sum_{\a_1 + \a_2 + \a_3 \le N} 
\frac{1}{k}\abs{\wh{(\rd_Q^{\a_1}\rd_{H_0}^{\a_2}\rd_M^{\a_3}f_\init)}_k(H_0, M)} d\ell dw.
\]
Now we proceed in exactly the same way as in (3); replace \(\abs{s}\) to \(\jap{s}\) by adding trivial estimate for \(s\le 1\) and sum over \(k\), we get the desired estimate on the energy-momentum tensor component. 

For the metric component, we need to use the ODE estimate. 
Let \((\widetilde{\a}_{\lin}, \widetilde{A}_{\lin}, \widetilde{\b}_{\lin})\) be the metric coefficient obtained by the solution \(\widetilde{\vec{X}}_{\lin}\) of \eqref{eq:linODE} with the source term 
\[
\widetilde{\vec{S}} = (\widetilde{\r}, \widetilde{\tr T}, \widetilde{j}, \widetilde{S}_R, \Psi-\Psi_0).
\]
We subtract the ODEs for \(\vec{X}_{\lin}\) and \(\widetilde{\vec{X}}_{\lin}\) to get
\[
(\vec{X}_{\lin}-\widetilde{\vec{X}_{\lin}})'(R) = \vec{F}_{\vec{X}}(\vec{X}_0(R), 0, R)(\vec{X}_{\lin}-\widetilde{\vec{X}_{\lin}})(R)
+ \vec{F}_{\vec{S}}(\vec{X}_0(R), 0, R)(\vec{S}-\widetilde{\vec{S}})(R).
\]
This has exactly the same form with the ODE \eqref{dslinODE} in (4), with \(\vec{X}_{\lin}-\widetilde{\vec{X}_{\lin}}\) instead of \(\rd_s \vec{X}_{\lin}\) and \(\vec{S}-\widetilde{\vec{S}}\) instead of \(\rd_s\vec{S}\).
Therefore, by invoking the same estimate with there, we get 
\[
\sup_R\abs{\rd_R^i[(\a_\lin, A_\lin, \b_\lin)(s, R)-(\widetilde{\a}_\lin, \widetilde{A}_\lin, \widetilde{\b}_\lin)(R)]} \le CB_\init \jap{s}^{-N+\max{i-2, 0}},
\]
for \(i\le N+1\). This completes the proof of (5), hence the theorem.

\end{proof}


\section{Bootstrap Assumptions and Main \emph{a priori} Estimates}\label{sec:BA}
The bootstrap assumptions on the metric coefficients that we need in the entire paper are stated here. 
For here and below, the bar notation is used to denote the metric coefficients at time \(T\), i.e. \(\bar{\a}(R) = \a(T, R)\), \(\Abar(R) = A(T, R)\), \(\bar{\b}(R) = \b(T, R)\), and \(\Kbar(R) = K(T, R)\).
The symbol \(\a, A, \b, K\) without bar are the metric coefficients at time \(s < T\).
Here \(T_B\) denotes the running bootstrap time, and the bound \(T_B < T_f\) that we impose on it is exactly the final time \(T_f = \e^{-1}(\log 1/\e)^{-2}\) of Theorem \ref{thm:nonlin}; the continuity argument in Section \ref{sec:PutEverything} lets \(T_B\) exhaust \([0, T_f]\).
For fixed \(T_B < T_f\) that we bootstrap, we assume the following bounds for \(\a, A, \b\) for all \(0 < s < T < T_B\):
\begin{align}
\abs{\rd_R^I (\a-\a_0, A-A_0, \b-\b_0)} &\le \d^{3/4}\e && I \le N+2, 
\label{BA0}\\
\abs{\rd_R^I \rd_s (\a, A, \b)} &\le \d^{3/4}\e\jap{s}^{-N + \max{I-2, 0}} && I \le N, 
\label{BA1}\\
\abs{\rd_R^I (\a-\bar{\a}, A-\Abar, \b-\bar{\b})} &\le \d^{3/4}\e\jap{s}^{-N + \max{I-2, 0}} && I \le N+2, 
\label{BA2}
\end{align}
whereas for \(K\), 
\begin{align}
\abs{\rd_R^I (K-K_0)} &\le \d^{3/4}\e && I \le N+1, 
\label{BAK0}\\
\abs{\rd_R^I \rd_s K} &\le \d^{3/4}\e\jap{s}^{-N + \max{I-1, 0}} && I \le N-1, 
\label{BAK1}\\
\abs{\rd_R^I (K-\Kbar)} &\le \d^{3/4}\e\jap{s}^{-N + \max{I-1, 0}} && I \le N+1. 
\label{BAK2}
\end{align}
For the top order derivative, we lose one more \(\jap{s}\):
\begin{align}
\abs{\rd_R^{N+3} (\a-\a_0, A-A_0, \b-\b_0)} &\le \d^{3/4}\e\jap{s},  
\label{BAH0}\\
\abs{\rd_R^{N+1} \rd_s (\a, A, \b)} &\le \d^{3/4}\e, 
\label{BAH1}\\
\abs{\rd_R^{N+2} (K-K_0)} &\le \d^{3/4}\e\jap{s},
\label{BAHK0}\\
\abs{\rd_R^{N} \rd_s K} &\le \d^{3/4}\e.
\label{BAHK1}
\end{align}
Note that we always assume one less controlled derivative for \(K\), since it involves one more derivative in the definition.

Then, we can state the main bootstrap theorem as follows. The proof of this theorem will be the main part of this paper. 
\begin{thm}[]\label{thm:BT}
    Given \(m, \cfr\) and \(N\ge 6\), there is \(\d > 0\) and \(\e_0 > 0\) such that the following is true: For any \(0 < \e < \e_0\), say there is a smooth solution \((f, \a, A, \b, K)\) of \eqref{eq:Gauss}--\eqref{eq:vlasov-rwl} in time \([0, T_B)\) with \(T_B \le \e^{-1}(\log 1/\e)^{-2}\) and satisfies the bounds \eqref{BA0}--\eqref{BAHK1} on \(s, T \in [0, T_B)\). 
    Then, in fact, the solution should satisfy the bounds \eqref{BA0}--\eqref{BAHK1} with improved constant \(C\d\) instead of \(\d^{3/4}\). 
    Here, \(C\) is a constant independent of \(\e\), \(\d\), \(s\) and \(T\), and depending only on \(m, \cfr\) and \(N\ge 6\).
\end{thm}

From here on, we will focus on proving Theorem \ref{thm:BT}.
Section \ref{sec:DAAV} is devoted to the definition of the dynamical action angle coordinate system to obtain the improved bounds on the derivatives of \(f\).
In Section \ref{sec:vlasov}, we will obtain estimates on \(f\) by analyzing the Vlasov equation, and then in Section \ref{sec:Integral} we control the density variables, which are defined to be the integrated quantity of \(f\) on each fiber, by using the estimates on \(f\) and the phase mixing mechanism.
In Section \ref{sec:Elliptic}, we estimate Einstein's equation to obtain improved bounds on the metric coefficients.
In particular, the non-decaying improved estimates for the top order derivatives \eqref{BA0}, \eqref{BAH0}, and \eqref{BAH1} will be proved in Subsection \ref{subsec:metrictop}, and the decaying estimates will be proved in Subsection \ref{subsec:metricdsa} and \ref{subsec:metricaabar}.
In Section \ref{sec:PutEverything}, we will use this theorem with a standard continuity argument to prove Theorem \ref{thm:nonlin}.
Up to that point, we will always assume that we have a solution \((f, \a, A, \b, K)\) on time interval \([0, T_B)\) satisfying the bootstrap assumptions \eqref{BA0}--\eqref{BAHK1}.

In the following, whenever we use a constant \(C\), \(\lesssim\), or big-O notation, the constant in the inequality depends only on \(m, \cfr\) and \(N\ge 6\), and is independent of \(\e\), \(\d\), \(t = s\), and \(T\).

\section{Definition of Dynamical Action Angle Variables}\label{sec:DAAV}
\subsection{Bootstrap assumption needed to define the new coordinate}
Here, we want to define the dynamical action angle variables. Because we need a coordinate definition which depends on the metric coefficients, we need to make some assumptions on the metric coefficients in order to define this coordinate system.
However, we do not need to use the full strength of the bootstrap assumptions \eqref{BA0}--\eqref{BAHK1} to define the coordinate system, and the following is the list of the assumptions that we need to define the coordinate system:

For all \(0 < s < T_B\), we assume the following bounds for \(\a, A, \b\):
\begin{align}
\abs{\rd_R^I (\a-\a_0, A-A_0, \b-\b_0)} &\le \d^{3/4}\e && I \le N+2,
\label{wBA0}\\
\abs{\rd_R^I \rd_s (\a, A, \b)} &\le \d^{3/4}\e && I \le N +1
\label{wBA1}.
\end{align}
\begin{align}
\intertext{For the top order derivative, we allow to lose one \(\jap{s}\):}
\abs{\rd_R^{N+3} (\a-\a_0, A-A_0, \b-\b_0)} &\le \d^{3/4}\e\jap{s}.
\label{wBAH0}
\intertext{For the low order derivatives, we need a weak decay in time: for all \(0 < s < T < T_B\),}
\abs{\rd_R^{\le 1} (\a-\bar{\a}, A-\Abar, \b-\bar{\b})} &\le \d^{3/4}\e\jap{s}^{-2},
\label{wBAL0}\\
\abs{\rd_R^{\le 1} (K-\Kbar)} &\le \d^{3/4}\e\jap{s}^{-2},
\label{wBALK0}\\
\abs{\rd_s (\a, A, \b)} &\le \d^{3/4}\e\jap{s}^{-2}.
\label{wBAL1}
\end{align}

\subsection{Definition of conserved quantities, support of \(f\)} 
We will use the same class \(\Scal_{\cfr}\) for the initial data as in the linear estimates. 
However, contrary to the linear case, \(H\) is not a conserved quantity for the nonlinear Vlasov equation.
Still, because there is a small room given by \(\cfr\) in the support condition, we can hope that the variation of \(H\) is controlled by the weak bootstrap assumptions, so that the action angle variable can be defined for all time.
\begin{prop}[Support of \(f\)]\label{prop:suppf}
    For \(f\) solving the nonlinear Vlasov equation \(\Dfr f = 0\) with initial data \(f_\init\) with support condition \(\supp(f_\init)\subset \Scal_{\cfr}\), if \(f(s, R, w, \ell)\) is nonzero, then we have \((R, w, \ell)\in\Scal_{\frac{1}{2}\cfr}\), in particular for \(H = H(s, R, w, \ell)\) and \(M = \ell\),
    \begin{equation}\begin{aligned}\label{nonlinsupp}
        V_{\min}(M) + \frac{1}{2}\cfr \le H^2 \le 1-\frac{1}{2}\cfr, \\
        \abs{(\rd_\ell\O_{\Sch})(H, M)}, \abs{\rd_H\O_{\Sch}(H, M)} \ge \frac{1}{2}\cfr.
    \end{aligned}\end{equation}
\end{prop}
\begin{proof}
The point of this claim is to show that the difference of \(H\) and \(M\) in the support of \(f\) with \(H_0\) and \(M\) of the initial data can be controlled within \(\Ocal(\e)\).
Hence, it suffices to show that, i) the difference between \(\a_0(R)^2\big(1+w^2+\ell/A_0(R)^2R^2\big)\) and \(H(0, R, w, \ell)^2\), and ii) the variation of \(H\) along the characteristics is of scale \(\Ocal(\e)\). 
The control of the first quantity is straightforward from the definition of \(H\) and the bootstrap assumption \eqref{wBA0}.

To control the second quantity, we need to compute \(\Dfr H\).
From the definition of \(H\) in \eqref{Hdef}, we have
\begin{equation}\begin{aligned}\label{DfrH}
\Dfr H
&= \rd_s H + \big(\frac{\a w}{AE}-\b\big)\left(\a'E-\frac{\a}{E}\big(\frac{\ell}{R^3A^2}+\frac{\ell A'}{R^2A^3}\big)-(\b A)'w\right) \\
&\qquad \qquad
+ \left(\frac{\a}{A^3}\frac{\ell}{R^3E}+\frac{\a A'}{A^4}\frac{\ell}{R^2E}-\frac{\a'}{A}E+\a Kw\right)(\a\frac{w}{E}-\b A) \\
&= \rd_s (\a E) - \rd_s (\b A)w -(\rd_s A)w\big(\frac{\a w}{AE}-\b\big).
\end{aligned}\end{equation}
Under the bootstrap assumptions \eqref{wBA0}--\eqref{wBAL1}, we have \(\abs{w}\le 2H\) on the support; this is because 
\[
\abs{H} \ge \a E - \abs{\b}A\abs{w} \ge c\abs{w}.
\]
Note that \(R\) is bounded below in the support, hence \(\a\) is bounded below by a positive constant, say \(2c\). 
Then, we can bootstrap \(H \le 1\): this is initially satisfied, and under this assumption we have 
\[
\abs{\Dfr H^2} = 2H\abs{\Dfr H} \le C\abs{\rd_s (\a, A, \b)} \le C\d^{3/4}\e\jap{s}^{-2}.
\]
Integrating this in time, the variation of \(H^2\) is of scale \(\Ocal(\d^{3/4}\e)\), and in particular this is much smaller than the room \(\cfr\) given in the support condition, hence we can recover the bootstrap assumption \(H \le 1\).
This completes the proof.
\end{proof}
Therefore, we still have \(\ell\) which gives the potential curve with a single local minimum, and \(H^2\) which is above the potential at local minimum but below the potential at infinity, so we can define the action angle variables as in the linear case.

Note that at this point the \(R\)-support of \(f\) is bounded, globally in time. 
We fix \(R_{\textrm{out}}>0\) so that \(f(s, R, w, \ell) \neq 0\) then \(R \in [m/4, R_{\textrm{out}}]\).

\subsection{Dynamical action angle variables}
Now we define the dynamical action angle variables, which is the coordinate system that we will use for the nonlinear estimates.
We will use \((t, Q_T, H, M)\) to denote this coordinate system, where \(t = s\), \(M = \ell\), and \(H\) given by \eqref{Hdef}.
Here, \(T\) refers to the fixed time which is greater than \(s\): We need it to define the action angle, as we want to use an integral which refers to the solution, especially the metric coefficients at time \(T\). 
Recall that the bar notation is used to denote the metric coefficients at time \(T\); especially \(\Ebar\) is defined by 
\[
\Ebar(R, w, \ell)^2 = 1 + w^2 + \frac{\ell}{R^2\Abar^2(R)}. 
\]

In the meantime, we need to use the quantity \(H_T\) which is defined as follows: 
\begin{equation}\label{HTdef}
H_T(R, w, \ell) \coloneqq \bar{\a}(R)\Ebar(R, w, \ell) - \bar{\b}(R)\Abar(R)w.
\end{equation}
We can control the variation of \(H_T\) on the characteristics by the bootstrap assumptions.
First, \eqref{eq:Keq} at \(s = T\) tells us that 
\[
(\bar{\b}\Abar)' = (\rd_s A)|_{s=T} + \bar{\a}\Abar\Kbar
\]
Using this, we have
\begin{equation}\label{DfrHT}
\begin{aligned}
\Dfr H_T &= \left(\frac{\a w}{AE}-\b\right)
\left(\bar{\a}'\Ebar - \frac{\bar{\a}}{\Ebar}\frac{\ell}{R^3\Abar^2}-\frac{\bar{\a}}{\Ebar}\frac{\Abar'\ell}{R^2\Abar^3}-(\bar{\b}\Abar)'w\right) \\
&\qquad
+ \left(\frac{\a}{A^3}\frac{\ell}{R^3E}+\frac{\a A'}{A^4}\frac{\ell}{R^2E}-\frac{\a'}{A}E+\a Kw\right)
\left(\frac{\bar{\a}}{\Ebar}w - \bar{\b}\Abar\right) \\
&= \frac{1}{A}\Bigg[\left(\frac{\a w}{E}-\b A\right)
\left(\bar{\a}'\Ebar - \frac{\bar{\a}}{\Abar^2}\frac{\ell}{R^3\Ebar}-\frac{\bar{\a}\Abar'}{\Abar^3}\frac{\ell}{R^2\Ebar}-\bar{\a}\Abar\Kbar w - (\rd_s A)|_{s = T}w\right) \\
&\qquad
- \left(\frac{\bar{\a}}{\Ebar}w - \bar{\b}\Abar\right)
\left(\a'E-\frac{\a}{A^2}\frac{\ell}{R^3E}-\frac{\a A'}{A^3}\frac{\ell}{R^2E}-\a AKw\right)\Bigg].\\
\end{aligned}
\end{equation}
Therefore, aside from the \(\rd_s A\) term everything is controlled by \((\a-\bar{\a}, A-\Abar, \b-\bar{\b}, K-\Kbar)\) with \eqref{wBAL0} and \eqref{wBALK0}, and \(\rd_s A\) at time \(T\) is also controlled by \eqref{wBAL1}.
In particular, we have 
\[
\abs{\Dfr H_T} \le C\d^{3/4}\e \jt^{-2},
\]
hence the variation of \(H_T\) over time is small, in particular of the scale \(\Ocal(\e)\).

Now, we can define the action angle referring to the value \(H_T\) and \(M\). 
\begin{defn}[Dynamical action angle \(Q_T\)]\label{def:QT}
    Given \((s, R, w, \ell)\in \supp (f)\), we can first compute \(H_T(R, w, \ell)\) by \eqref{HTdef}. 
    Proposition \ref{prop:suppf} together with \eqref{wBAL0} tells us that \((H_T, M = \ell)\) is in the set that the following definition is valid.
    \begin{enumerate}
        \item We define the potential \(V_T(M, R)\) by
        \begin{equation}\label{VTdef}
        V_T(M, R) = (\bar{\a}^2-\bar{\b}^2\Abar^2)\Big(1+\frac{M}{R^2\Abar^2(R)}\Big).
        \end{equation}
        Note that \(V_T\) is close to \(V_0\) defined in \eqref{Vdef}, and in particular \(V_T\) also has a single local minimum, and \(H_T^2\) is above the local minimum but below the value at infinity.
        \item We define \(\big(\frac{3}{2}+\sqrt{2}\big)m < R_-(H_T, M) < R_+(H_T, M)\) to be two roots of the equation 
        \[
        H_T^2 = V_T(M, R).
        \]
        We know that there are two solutions, as we have \eqref{nonlinsupp}, and the difference of \(V_{\min}(M)\) and the potential minimum, \(H\) and \(H_T\) can be controlled by choosing \(\e_0\) small. 
        \item For \(R\in [R_-(H_T, M), R_+(H_T, M)]\) we define \(w_\pm(R, H_T, M)\) to be the two roots of the equation 
        \[
        H_T = \bar{\a}\Ebar - \bar{\b}\Abar w.
        \]
        Note that \(\Ebar\) depends on \(w\). We can explicitly solve this equation and get 
        \begin{equation}\label{wpmdef}
        w_\pm = \frac{1}{\bar{\a}^2-\bar{\b}^2\Abar^2} \bigg[H_T\bar{\b}\Abar \pm \bar{\a}\sqrt{H_T^2 - V_T(M, R)}\bigg].            
        \end{equation}
        Here we see that \(w_\pm\) degenerates at \(R_\pm(H_T, M)\), but those are not necessarily zero. 
        \item We define \(\O(H_T, M)\) to be the quantity satisfying 
        \begin{equation}\label{Odef}
        \frac{2\pi}{\O(H_T, M)} 
        = \int_{R_-(H_T, M)}^{R_+(H_T, M)} \frac{dR}{\frac{\bar{\a}w_+}{\Abar\Ebar(w_+)}-\bar{\b}} 
        - \int_{R_-(H_T, M)}^{R_+(H_T, M)} \frac{dR}{\frac{\bar{\a}w_-}{\Abar\Ebar(w_-)}-\bar{\b}}.            
        \end{equation}
        \item Finally, for given \((R, w, \ell)\), we define \(Q_T\) in the following way: We first compute \(H_T\) and \(M\) as above. 
        Then, from the definition of \(w_\pm\) we either have \(w = w_-(R, H_T, M)\) or \(w = w_+(R, H_T, M)\).
        For \(R\in [R_-(H_T, M), R_+(H_T, M))\), we first define \(\Qtil_T(R, H_T, M)\) to be the integral 
        \begin{equation}\label{Qtildef}
        \Qtil_T(R, H_T, M) \coloneqq \begin{cases}
        \displaystyle\int_{R_-(H_T, M)}^R \frac{dR'}{\frac{\bar{\a}w_+(R')}{\Abar\Ebar(w_+(R'))}-\bar{\b}} &\text{ if } w = w_+(R, H_T, M) \\[0.5cm]
        \displaystyle\int_{R_-(H_T, M)}^R \frac{dR'}{\frac{\bar{\a}w_-(R')}{\Abar\Ebar(w_-(R'))}-\bar{\b}} &\text{ if } w = w_-(R, H_T, M),
        \end{cases}
        \end{equation}
        and set 
        \[
        \Qtil_T(R_+(H_T, M), H_T, M) = \int_{R_-(H_T, M)}^{R_+(H_T, M)} \frac{dR'}{\frac{\bar{\a}w_-(R')}{\Abar\Ebar(w_-(R'))}-\bar{\b}}.
        \]
        Then, we define \(Q_T(s, R, w, \ell)\) to be 
        \[
        Q_T(s, R, w, \ell) = \O(H_T(R, w, \ell), M = \ell)\Qtil_T(R, H_T(R, w, \ell), M = \ell).
        \]
    \end{enumerate}
\end{defn}
Some remarks are in order. 
\begin{rmk}[\(Q_T\) is not globally smooth]
    Note that the value of \(Q_T\) jumps at \(R = R_+(H_T(R, w, \ell), M = \ell)\), of which size of discontinuity is \(2\pi\). However, this is just similar to the discontinuity in the usual polar coordinate. 
    In particular, we will show that the derivatives are regular enough for our purpose in the propositions below. 
\end{rmk}
\begin{rmk}
The motivation of this definition is the following identity: we have 
\[
\rd_R|_{(R, H_T(R, w, \ell), M = \ell)} Q_T
= \frac{\O(H_T, M)}{\frac{\bar{\a}w}{\Abar\Ebar(w)}-\bar{\b}}.
\]
Here we are taking \(R\) derivative with fixed \(H_T\) and \(M\), so this is not well defined at \(R = R_\pm(H_T, M)\). 
Also, this quantity depends on which value of \(w\) we are working with, which should be one of \(w_\pm(R, H_T, M)\).
This fact will be revisited when we calculate the change of variables map. 
\end{rmk}

\begin{defn}[The \((t, Q_T, H, M)\) coordinate system]\label{def:DAAV}
    For every fixed \(T < T_B\), on the region 
\[
\{(s, R, w, \ell)|s\in [0, T], (R, w, \ell)\in \Scal_{\frac{1}{2}\cfr}\},
\] 
    we can define the coordinate system \((t, Q_T, H, M)\) as above.
    To distinguish the coordinate vectors in \((s, R, w, \ell)\) and \((t, Q_T, H, M)\), we use different variables even though \(t = s\) and \(M = \ell\): 
    We use \((\rd_s, \rd_R, \rd_w, \rd_\ell)\) and \((\rd_t, \rd_{Q_T}, \rd_H, \rd_M)\) in each coordinate system. 
    In the next subsection, we actually prove that this is a well-defined coordinate system, and the change of variables map is regular enough. 
\end{defn}

Finally, as in the linear case, we need to define the vector field \(Z\) which does not cause a growth in time. 
\begin{defn}[The vector field \(Z\)]\label{def:Zdef}
    For a fixed \(T\), we define the vector field \(Z\) in the \((t, Q_T, H, M)\) coordinate system as 
    \[
    Z = (\rd_M \O_{\Sch})(H, M)\rd_H - (\rd_H \O_{\Sch})(H, M)\rd_M.
    \]
    Note that this satisfies \(Z(\O_\Sch(H, M)) = 0\), and commutes with \(\rd_{Q_T}\). 
    We have \[[\rd_H, Z] = (\rd_H\rd_M \O_\Sch)(H, M)\rd_H - (\rd_H^2 \O_\Sch)(H, M)\rd_M,\] which is still in the span of \(\rd_H\) and \(Z\) with bounded coefficients.
\end{defn}
One of the important relations between the vector fields \(\rd_R, \rd_w, \rd_\ell\) and \(\rd_{Q_T}, \rd_H, Z\) can be expressed as follows. 
\begin{align}
\rd_\ell &= \rd_\ell Q_T \rd_{Q_T} + \rd_\ell H \rd_H + \rd_M \notag\\
&= \rd_\ell Q_T \rd_{Q_T} + \Big(\rd_\ell H +\frac{\rd_M\O_\Sch(H, M)}{\rd_H\O_\Sch(H, M)}\Big)\rd_H -\frac{1}{\rd_H \O_{\Sch}(H, M)}Z, \notag \\
&= \rd_\ell Q_T \rd_{Q_T} + \frac{\rd_\ell(\O_\Sch(H, M))}{\rd_H\O_\Sch(H, M)}\rd_H -\frac{1}{\rd_H \O_{\Sch}(H, M)}Z, \notag \\
\intertext{which gives}
\rd_H &= \frac{\rd_H \O_{\Sch}(H, M)}{\rd_\ell(\O_\Sch(H, M))}\rd_\ell + \frac{1}{\rd_\ell(\O_\Sch(H, M))}Z - \frac{\rd_H\O_{\Sch}(H, M)\rd_\ell Q_T}{\rd_\ell(\O_\Sch(H, M))}\rd_{Q_T}. \label{Hlincomb}
\end{align}
Later, we will use this relation to control the \(\rd_H\) derivative by the \(\rd_{Q_T}\) and \(Z\) derivatives for the integral estimates. Note that \(\rd_\ell\) is the one which can be moved by integration by parts, and \(\rd_{Q_T}\) and \(Z\) are good derivatives.

\subsection{Regularity of the change of variables map}
The goal of this section is to prove the following proposition. 
\begin{prop}[Regularity of the change of variables map]\label{prop:regDAAV}
    Let \(T < T_B\) be fixed, and consider the quantities \((t, Q_T, H, M)\) defined on the domain given in Definition \ref{def:DAAV}.
    \begin{enumerate}
        \item For every fixed \(T < T_B\), the change of variables map from \((s, R, w, \ell)\) to \((t, Q_T, H, M)\) is \(C^{N+2}\). 
        More precisely, we have the following bounds for the derivatives:
        \begin{align}
        \norm{\rd_{R, w, \ell}^I\left(\frac{\rd(t, Q_T, H, M)}{\rd(s, R, w, \ell)}\right)} &\le C & \text{ for } I&\le N+1. \label{CoVptoA}
        \end{align}
        In particular, \(\O(H_T, M)\) is \(C^{N+2}\) with bounded derivatives in its arguments. 
        \item The Jacobian of the change of variables map is \(\Abar\O(H_T(R, w, \ell), M = \ell) + \Ocal(\e)\), in particular bounded away from zero. Therefore, \((t, Q_T, H, M)\) is indeed a coordinate system. 
        \item The following estimates hold: 
        \begin{align}
        \norm{\rd_{Q_T, H, M}^I (R, w, \ell)} &\le C & \text{ for } I&\le N+2 \label{CoVAtop}
        \end{align}
    \end{enumerate}
\end{prop}
For this, the key is to control the derivatives of \(Q_T\), which is defined in terms of the singular integral.
We first derive a better expression of the integrand, which allows us to see the singularity more explicitly. 

\begin{defn}[\(D(R, H_T, M)\)]\label{Intsimplify}
    We define the quantity \(D(R, H_T, M)\) as follows:
    \begin{equation}\label{Ddef}
    D(R, H_T, M) = H_T^2 - V_T(M, R) = H_T^2 - (\bar{\a}^2-\bar{\b}^2\Abar^2)\Big(1+\frac{M}{R^2\Abar^2(R)}\Big).
    \end{equation}
    This allows us to simplify \eqref{wpmdef} and the expression of \(\Ebar\): 
    \[
    w_\pm = \frac{\bar{\b}\Abar H_T \pm \bar{\a}\sqrt{D}}{\bar{\a}^2-\bar{\b}^2\Abar^2}, \qquad
    \Ebar(w_\pm) = \frac{H_T + \bar{\b}\Abar w_\pm}{\bar{\a}} = \frac{\bar{\a}H_T \pm \bar{\b}\Abar\sqrt{D}}{\bar{\a}^2-\bar{\b}^2\Abar^2}.
    \]
    Therefore, the integrand in the definition of \(Q_T\) can be simplified as 
    \begin{equation}\label{ptintegrand}
    \frac{1}{\frac{\bar{\a}w_\pm}{\Abar\Ebar(w_\pm)}-\bar{\b}}
    = \frac{\Abar(\pm\bar{\a}H_T + \bar{\b}\Abar\sqrt{D})}{\bar{\a}^2-\bar{\b}^2\Abar^2}\frac{1}{\sqrt{D}}
    = \pm\underbrace{\frac{\Abar\bar{\a}}{\bar{\a}^2-\bar{\b}^2\Abar^2}}_{\eqqcolon W_T^{(1)}(R)}\frac{H_T}{\sqrt{D}} 
    + \underbrace{\frac{\bar{\b}\Abar^2}{\bar{\a}^2-\bar{\b}^2\Abar^2}}_{\eqqcolon W_T^{(2)}(R)}.
    \end{equation}
    Further, with \(D\) we can simplify the expression for \(\O(H_T, M)\) given in \eqref{Odef}.
    \begin{equation}\label{Odefsimp}
    \frac{\pi}{\O(H_T, M)} = \int_{R_-(H_T, M)}^{R_+(H_T, M)}\frac{W_T^{(1)}(R)H_T}{\sqrt{H_T^2-V_T(M, R)}} dR.
    \end{equation}
\end{defn}
\begin{rmk}[Singularity of the integrand]
    For our domain of parameters, \((\rd_R V_T)(R_\pm(H_T, M), M)\) is away from \(0\), which means that \(D\) vanishes linearly at \(R = R_\pm(H_T, M)\).
    Therefore, the singularity of the integrand is of order \(\frac{1}{\sqrt{R-R_\pm(H_T, M)}}\), which makes all the integrals finite. 
    However, this might cause a problem when we take derivatives, as it might pick up a boundary term or have a non-integrable singularity.
    Still, in the case when the singularity is multiplied by a factor which vanishes at \(R = R_\pm(H_T, M)\), the singularity can be cancelled, and have a finite derivative. 
    This is what we try to do in the following lemmas. 
\end{rmk}
\begin{rmk}[Identity on \(D\)]
    The following identity is useful for the foregoing calculation: 
    \begin{equation}\label{Dsqrt}
    D(R, H_T(R, w, \ell), M = \ell) = \Ebar^2(w)\big(\frac{\bar{\a}w}{\Ebar}-\bar{\b}\Abar\big)^2
    = \Ebar^2(w)((\rd_w H_T)(R, w, \ell))^2.
    \end{equation}
\end{rmk}

From here, we state results only for \(w_+\), near \(R_-\). The other three cases follow with exactly the same argument. 
\begin{lem}[Lemma 5.10 of \cite{CL24}]\label{lem:intsep}
    Let \(\g(R, H_T, M)\) be a \(C^1\) function. Then
    \begin{multline*}
    \int_{R_-(H_T, M)}^{R} \frac{\g(R', H_T, M) dR'}{[H_T^2 - V_T(R', M)]^{1/2}} \\
    = -\frac{2\g(R_-(H_T, M), H_T, M)}{\partial_R V_T(R_-(H_T, M), M)} \sqrt{H_T^2 - V_T(R, M)} + \int_{R_-(H_T, M)}^{R} \frac{\tilde{\g}(R, H_T, M)}{[H_T^2 - V_T(R', M)]^{1/2}} dR',
    \end{multline*}
    where
    \[
    \tilde{\g}(R, H_T, M) = \g(R, H_T, M) - \frac{\g(R_-(H_T, M), H_T, M)\partial_R V_T(R, M)}{\partial_R V_T(R_-(H_T, M), M)}
    \]
    which satisfies the property that
    \[
    \lim_{R \to R_-} \tilde{\g}(R, H_T, M) = 0.
    \]
\end{lem}
\begin{proof}
The proof is exactly the same as the proof of Lemma 5.10 of \cite{CL24}, which is based on integration by parts.
\end{proof}

Note that \(\rd_R, \rd_w, \rd_\ell\) in the following lemma are in \((R, w, \ell)\) coordinates. 
In particular, \(\Ebar\) and \(H_T\) are differentiable in \((R, w, \ell)\) as much as the metric coefficients. 
\begin{lem}[Lemma 5.11 of \cite{CL24}]\label{lem:sqrtderiv}
    Suppose \(\g(R, H_T, M)\) is a \(C^1\) function. Define
    \[
    \bar{\et}(R, H_T, M) = \g(R, H_T, M)[H_T^2 - V_T(R, M)]^{1/2}
    \]
    and
    \[
    \et(R, w, \ell) = \bar{\et}(R, H_T(R, w, \ell), M = \ell).
    \]
    Then, for \(w\) satisfying \(w = w_+(R, H_T(R, w, \ell), M = \ell)\),
    \begin{align*}
    (\rd_R \et)(R, w, \ell) &= \Ebar(\rd_w H_T)(R, w, \ell)\cdot (\rd_R \g + \rd_R H_T \rd_{H_T} \g)(R, H_T(R, w, \ell), \ell) \\
    &\qquad
    + \g(R, H_T(R, w, \ell), \ell)\cdot \rd_R(\Ebar\rd_w H_T)(R, w, \ell), \\
    (\rd_w \et)(R, w, \ell) &= \Ebar(\rd_w H_T)(R, w, \ell)\cdot (\rd_w H_T\rd_{H_T}\g)(R, H_T(R, w, \ell), \ell) \\
    &\qquad
    + \g(R, H_T(R, w, \ell), \ell)\cdot \rd_w(\Ebar\rd_w H_T)(R, w, \ell), \\
    (\rd_\ell \et)(R, w, \ell) &= \Ebar(\rd_w H_T)(R, w, \ell)\cdot (\rd_\ell H_T\rd_{H_T}\g + \rd_M \g)(R, H_T(R, w, \ell), \ell) \\
    &\qquad
    + \g(R, H_T(R, w, \ell), \ell)\cdot \rd_\ell(\Ebar\rd_w H_T)(R, w, \ell).
    \end{align*}
\end{lem}
\begin{proof}
For \(w = w_+\), we have 
\[
\sqrt{H_T^2-V_T(R, M)} = \sqrt{D} = \Ebar\rd_w H_T,
\]
the desired identities are direct from the chain rule. 
\end{proof}

Considering the decomposition in \eqref{ptintegrand}, we need to analyze the derivatives of the integral of two types of integrands.
Thanks to Lemma \ref{lem:intsep}, for the first type of integrand we can assume that the function vanishes at \(R = R_-(H_T, M)\). 
\begin{lem}\label{lem:intderiv}
Suppose \(\g_1(R, H_T, M)\) is a \(C^1\) function which satisfies \(\g_1(R_-(H_T, M), H_T, M) = 0\). Define 
\[
\bar{\et}(R, H_T, M) = \int_{R_-(H_T, M)}^{R} \frac{\g_1(R', H_T, M)}{\sqrt{H_T^2 - V_T(R', M)}} dR',
\] 
and 
\[
\et(R, w, \ell) = \bar{\et}(R, H_T(R, w, \ell), M = \ell).
\]
Then, we have the following identity, for \((R, w, \ell)\) with \(\rd_R V_T(R, M = \ell) < 0\). Here, all the quantities including \(H_T\) and \(M\) on the RHS is evaluated at \((R, w, \ell)\).
\begin{align}
&\begin{aligned}
(\rd_R\et)(R, w, \ell)
&= \rd_R H_T\int_{R_-(H_T, M)}^{R} \frac{(\rd_{H_T} \g_1 + 2H_T \rd_R (\frac{\g_1}{\rd_R V_T}))(R', H_T, M)}{\sqrt{H_T^2 - V_T(R', M)}} dR' \\
&\qquad
- \frac{2\g_1}{\rd_R V_T}\rd_R (\Ebar\rd_w H_T),
\end{aligned} \label{intderivR}\\
&\begin{aligned}
(\rd_w\et)(R, w, \ell)
&= \rd_w H_T\int_{R_-(H_T, M)}^{R} \frac{(\rd_{H_T} \g_1 + 2H_T \rd_R (\frac{\g_1}{\rd_R V_T}))(R', H_T, M)}{\sqrt{H_T^2 - V_T(R', M)}} dR' \\
&\qquad
- \frac{2\g_1}{\rd_R V_T}\rd_w (\Ebar\rd_w H_T), 
\end{aligned} \label{intderivw}\\
&\begin{aligned}
(\rd_\ell\et)(R, w, \ell)
&= \rd_\ell H_T\int_{R_-(H_T, M)}^{R} \frac{(\rd_{H_T} \g_1 + 2H_T \rd_R (\frac{\g_1}{\rd_R V_T}))(R', H_T, M)}{\sqrt{H_T^2 - V_T(R', M)}} dR' \\
&\qquad
+ \int_{R_-(H_T, M)}^{R} \frac{(\rd_M \g_1 - \rd_R (\frac{\g_1 \rd_\ell V_T}{\rd_R V_T}))(R', H_T, M)}{\sqrt{H_T^2 - V_T(R', M)}} dR'
- \frac{2\g_1}{\rd_R V_T}\rd_\ell (\Ebar\rd_w H_T).  
\end{aligned} \label{intderivell}
\end{align} 
As a function of \((R, H_T, M)\), the integral satisfies the following identity. 
\begin{equation}\label{dHTlem}
\lim_{R\to R_-(H_T, M)+}\rd_{H_T}|_{(R, H_T, M)} \int_{R_-(H_T, M)}^{R} \frac{\g_1(R', H_T, M)}{\sqrt{H_T^2 - V_T(R', M)}} dR' = 0.
\end{equation}    
\end{lem}
\begin{proof}
We first compute the derivative of the integral as a function of \((R, H_T, M)\). 
Here, the point is that the derivative does not pick up the boundary term, as the function \(\g_1\) vanishes at \(R = R_-(H_T, M)\).
We use \(D(R, H_T, M) = H_T^2 - V_T(R, M)\), which is evaluated at \((R', H_T, M)\) within the integral, and at \((R, H_T, M)\) on RHS in the following calculation.
Note that we can do the integration by parts for last two identities as we have \(\rd_R V_T(R', M) < 0\) on \(R'\in [R_-(H_T, M), R]\).
\begin{align*}
\rd_R|_{(R, H_T, M)}\int_{R_-(H_T, M)}^{R}\frac{\g_1(R', H_T, M)}{\sqrt{D}} dR'
&= \frac{\g_1(R, H_T, M)}{\sqrt{D}}, \\
\rd_{H_T}|_{(R, H_T, M)}\int_{R_-(H_T, M)}^{R}\frac{\g_1(R', H_T, M)}{\sqrt{D}} dR'
&= \int_{R_-(H_T, M)}^{R} \frac{(\rd_{H_T} \g_1 + 2H_T \rd_R (\frac{\g_1}{\rd_R V_T}))(R', H_T, M)}{\sqrt{D}} dR' \\
&\qquad
- \frac{2\g_1(R, H_T, M)}{\rd_R V_T}\frac{H_T}{\sqrt{D}}, \\
\rd_M|_{(R, H_T, M)}\int_{R_-(H_T, M)}^{R}\frac{\g_1(R', H_T, M)}{\sqrt{D}} dR'
&= \int_{R_-(H_T, M)}^{R} \frac{(\rd_M \g_1 - \rd_R (\frac{\g_1 \rd_\ell V_T}{\rd_R V_T}))(R', H_T, M)}{\sqrt{D}} dR' \\
&\qquad
+ \frac{\g_1(R, H_T, M)}{\rd_R V_T}\frac{\rd_\ell V_T}{\sqrt{D}}.
\end{align*}
Now, using the chain rule we obtain the desired identities. 
The integral terms are straightforward, and to show that there is no singular term, we need the following identities: 
\begin{align*}
\frac{\g_1}{\sqrt{D}}\big(1 - \rd_R H_T \frac{2H_T}{\rd_R V_T}\big) &= -\frac{\g_1}{\rd_R V_T}\frac{\rd_R D}{\sqrt{D}} = -\frac{2\g_1}{\rd_R V_T}\rd_R (\Ebar\rd_w H_T), \\
\frac{\g_1}{\sqrt{D}}\big(-\rd_w H_T \frac{2H_T}{\rd_R V_T}\big) &= -\frac{\g_1}{\rd_R V_T}\frac{\rd_w D}{\sqrt{D}} = -\frac{2\g_1}{\rd_R V_T}\rd_w (\Ebar\rd_w H_T), \\
\frac{\g_1}{\sqrt{D}}\big(-\rd_\ell H_T \frac{2H_T}{\rd_R V_T} + \frac{\rd_\ell V_T}{\rd_R V_T}\big) 
&= -\frac{\g_1}{\rd_R V_T}\frac{\rd_\ell D}{\sqrt{D}} = -\frac{2\g_1}{\rd_R V_T}\rd_\ell (\Ebar\rd_w H_T).
\end{align*}
\end{proof}

Now, we prove Proposition \ref{prop:regDAAV}.
\begin{proof}[Proof of Proposition \ref{prop:regDAAV}]
\step{1}
Fix \(T < T_B\) and define \((t, Q_T, H, M)\) as in Definition \ref{def:DAAV}. We first show that those quantities satisfies the estimates \eqref{CoVptoA}.
Note that in the definition we refer to the metric coefficients at time \(T\), and in our timescale the weak bootstrap assumptions \eqref{wBAH0} still gives the boundedness of \(\rd_R^{N+3}\) of the metric coefficients at time \(T\).
Therefore, in this proof, it suffices to show that the desired derivatives can be controlled by \((N+3)\)-th order derivatives of the metric coefficients. 

For \(t\) and \(M\) every derivative is trivial.
For \(H\), it is \(C^{N+2}\) in \((R, w, \ell)\); especially \((N+2)\)-th order derivative of \(\rd_R H\) is bounded by \((N+3)\)-th order derivative of the metric coefficients. 
For \(\rd_s H = \rd_s (\a E- \b Aw)\), the weak bootstrap assumption \eqref{wBA1} proves that this quantity is \(C^{N+1}\).
For \(Q_T\), this quantity is independent of \(s\), hence it suffices to show that \(Q_T\) is \(C^{N+2}\) in \((R, w, \ell)\).

First, notice that the lower limit of the integral \(R_-(H_T, M)\) has bounded \(C^{N+3}\) derivatives in \((R, w, \ell)\); first we have the following identity from the definition of \(R_-\): 
\begin{align}
\rd_{H_T}(R_-(H_T, M)) &= \frac{2H_T}{\rd_R V_T(R_-(H_T, M), M)}, \label{dHTR-}\\
\rd_M(R_-(H_T, M)) &= -\frac{\rd_\ell V_T(R_-(H_T, M), M)}{\rd_R V_T(R_-(H_T, M), M)}. \label{dMR-}
\end{align} 
Since \(H_T\) and \(M\) have bounded \((N+3)\)-th order derivatives in \((R, w, \ell)\) and so is \(V_T\), and since we are working on the regime that \(\rd_R V_T\neq 0\) at \(R_-(H_T, M)\), we can conclude that \((N+3)\)-th order derivatives of \(R_-(H_T, M)\) can be controlled.
This is also true for \(R_+\).

As \(Q_T\) is defined as a product of \(\O(H_T, M)\) and \(\Qtil(R, H_T, M)\), we first prove that \(\Qtil\) given by \eqref{Qtildef} is \(C^{N+2}\) in \((R, w, \ell)\). Then, since \(R_+\) is \(C^{N+3}\), the normalization \(\O\) should also be \(C^{N+2}\) in \((R, w, \ell)\), and the desired regularity of \(Q_T\) follows.
We also restrict to the case when \(w = w_+(R, H_T, M)\), without loss of generality. 
Using the representation \eqref{Intsimplify} and Lemma \ref{lem:intsep}, we can write \(\Qtil\) as a sum of three terms: 
\begin{align*}
\Qtil(R, H_T, M)
&= \int_{R_-(H_T, M)}^{R} \frac{W_T^{(1)}(R')H_T}{\sqrt{H_T^2-V_T(R', M)}} dR' + \int_{R_-(H_T, M)}^{R} W_T^{(2)}(R') dR' \\
&= -\frac{2W_T^{(1)}(R_-(H_T, M))H_T}{\rd_R V_T(R_-(H_T, M), M)} \sqrt{H_T^2-V_T(R, M)} \\
&\quad + \int_{R_-(H_T, M)}^{R} \frac{\tilde{W}_T^{(1)}(R', H_T, M)H_T}{\sqrt{H_T^2-V_T(R', M)}} dR' + \int_{R_-(H_T, M)}^{R} W_T^{(2)}(R') dR',
\end{align*}
The last term is the easiest: we can compute the first derivative as follows. 
\begin{align*}
\rd_R \int_{R_-(H_T, M)}^{R} W_T^{(2)}(R') dR' &= W_T^{(2)}(R)-\rd_R H_T \rd_{H_T}R_-(H_T, M) W_T^{(2)}(R_-(H_T, M)), \\
\rd_w \int_{R_-(H_T, M)}^{R} W_T^{(2)}(R') dR' &= -\rd_w H_T \rd_{H_T}R_-(H_T, M) W_T^{(2)}(R_-(H_T, M)), \\
\rd_\ell \int_{R_-(H_T, M)}^{R} W_T^{(2)}(R') dR' &= -(\rd_\ell H_T \rd_{H_T}R_-(H_T, M)+\rd_M R_-(H_T, M))W_T^{(2)}(R_-(H_T, M)).
\end{align*}
These are all sufficiently regular. For the first two terms, we need to use Lemma \ref{lem:sqrtderiv} and Lemma \ref{lem:intderiv} repeatedly to see that both of them are \(C^{N+2}\) in \((R, w, \ell)\).
We notice that the \(I\)-th derivative of \(\Qtil\) requires \((I+1)\)-th order derivative of the metric coefficients: this is mainly because of the term \(\rd_R\big(\frac{\g_1}{\rd_R V_T}\big)\).
This also tells us that the term that requires the highest-order derivative always appears through \(V_T\). 
For the first term, the induction with Lemma \ref{lem:sqrtderiv} works; using the identity once produces a term with the same structure which we can induct. 

For the second term, Lemma \ref{lem:intderiv} allows us to argue in the same way for \((R, w, \ell)\) with \(\rd_R V_T(R, M = \ell) < 0\). 
To cover the other points, we proceed as follows: set \(R_c(M)\) to be the point where \(V_T(R, M)\) as a function of \(R\) attains its local minimum, and choose a fixed small number \(\bfr>0\) (independent of \(M\)).
Since \(\rd_{RR}^2 V_T\) is bounded from below near \(R_c(M)\) in our domain of interest, we know that \(\abs{\rd_R V_T(R, M)}\) is bounded away from zero for \(R\) such that \(\abs{R-R_c(M)} \geq \bfr\).
Hence, we can separate three cases: if \(R \in [R_-(H_T, M), R_c(M)-\bfr]\), then the above argument proves that the derivatives are bounded up to \((N+2)\)-th order. \\
For \(R \in [R_c(M)-\bfr, R_c(M)+\bfr]\), we separate the integral as follows: 
\[
\int_{R_-(H_T, M)}^{R_c(M)-\bfr} \frac{\tilde{W}_T^{(1)}(R', H_T, M)H_T}{\sqrt{H_T^2-V_T(R', M)}} dR'
+\int_{R_c(M)-\bfr}^{R} \frac{\tilde{W}_T^{(1)}(R', H_T, M)H_T}{\sqrt{H_T^2-V_T(R', M)}} dR'.
\]
The derivative of first term can be calculated as in Lemma \ref{lem:intderiv}; the difference is that there is no \(R\) derivative here.
Since \(\abs{\rd_R V_T(R, M)}\) is bounded from below, the derivatives are bounded. 
For the second term, the integrand is regular, in particular \(\sqrt{H_T^2-V_T}\) is bounded below, so we need one less derivative to control this.
Note that it does not matter that \(\rd_R V_T = 0\) at \(R_c\), as we do not put it on the denominator in this interval. 
\\
For \(R \in [R_c(M)+\bfr, R_+(H_T, M)]\), from the argument for the second case it suffices to control the integral on \([R_c(M)+\bfr, R]\); this is equal to 
\[
\int_{R_c(M)+\bfr}^{R_+(H_T, M)} \frac{\tilde{W}_T^{(1)}(R', H_T, M)H_T}{\sqrt{H_T^2-V_T(R', M)}} dR'
- \int_{R}^{R_+(H_T, M)} \frac{\tilde{W}_T^{(1)}(R', H_T, M)H_T}{\sqrt{H_T^2-V_T(R', M)}} dR'.
\]
The first term can be controlled by the symmetric argument as the second case, and the second term can be controlled by the symmetric argument for the first case with symmetric version of Lemma \ref{lem:intderiv}.

\step{2}
Now we estimate the Jacobian of the change of variables map. We have
\[
\begin{bmatrix}
\rd_s t & \rd_s Q_T & \rd_s H & \rd_s M \\
\rd_R t & \rd_R Q_T & \rd_R H & \rd_R M \\
\rd_w t & \rd_w Q_T & \rd_w H & \rd_w M \\
\rd_\ell t & \rd_\ell Q_T & \rd_\ell H & \rd_\ell M
\end{bmatrix}
=
\begin{bmatrix}
1 & 0 & \rd_s H & 0 \\
0 & \rd_R Q_T & \rd_R H & 0 \\
0 & \rd_w Q_T & \rd_w H & 0 \\
0 & \rd_\ell Q_T & \rd_\ell H & 1
\end{bmatrix},
\]
hence the Jacobian is \(\rd_w H \cdot \rd_R Q_T - \rd_R H \cdot \rd_w Q_T\).
Away from \(R = R_\pm (H_T, M)\), we can calculate this as 
\begin{equation}\label{JacCalc}
\begin{aligned}
&\rd_w H \cdot \rd_R Q_T - \rd_R H \cdot \rd_w Q_T \\
&= \rd_w H\cdot\rd_R|_{(R, H_T, M)} Q_T + \rd_w H\cdot\rd_R H_T \rd_{H_T}|_{(R, H_T, M)} Q_T - \rd_R H\cdot\rd_w H_T \rd_{H_T}|_{(R, H_T, M)} Q_T \\
&\qquad
+\underbrace{\Abar\O(H_T, M) - \rd_w H_T\cdot\rd_R|_{(R, H_T, M)} Q_T}_{= 0} \\
&= \Abar\O(H_T, M) + \rd_w (H-H_T) (\rd_R|_{(R, H_T, M)} Q_T) \\
&\qquad
+ (\rd_w H\rd_R H_T - \rd_R H\rd_w H_T)(\rd_{H_T}|_{(R, H_T, M)} Q_T) \\
&= \Abar\O(H_T, M) + \rd_w (H-H_T)\rd_R Q_T - \rd_R(H-H_T)\rd_w Q_T.
\end{aligned}
\end{equation}
From the weak bootstrap assumption \eqref{wBAL0} and the conclusions in Step 1 we know that the second and third term are of size \(\Ocal(\e)\), hence the Jacobian is \(\Abar\O(H_T, M) + \Ocal(\e)\), which is bounded away from zero.

\step{3} Now this is a direct consequence of the inverse function theorem, with the estimates proved in Step 1 and Step 2.
\end{proof}

\subsection{Refined estimates on the period function}
We previously defined the function \(\O\) in the way that it depends on the metric coefficients at time \(T\). 
Of course we assume that the metric coefficients at time \(T\) are close to the Schwarzschild metric, but we still need to estimate the difference between \(\O\) and \(\O_{\Sch}\) to derive estimates from the transport equation.
There are two reasons for this: 
First, we need the monotonicity of \(\O\) in \(\ell\); we proved that \(\rd_\ell (\O_{\Sch}(H, M))\) has a sign in Proposition \ref{prop:monotonic}, and we want to claim that \(\rd_\ell (\O(H, M))\) also has a sign as it is close to \(\rd_\ell (\O_{\Sch}(H, M))\).
Second, we defined the vector field \(Z\) in terms of the derivatives of \(\O_{\Sch}\) so that 
\[
Z(\O_{\Sch}(H, M)) = 0,
\]
to keep \(Z\) as a regular vector field.

From the definition of \(Z\) we know that for any \(i_1, i_2\ge 0\), we have \(\rd_H^{i_1}Z^{i_2}(\O_{\Sch}(H, M)) = 0\).
In order to do the estimates on the Vlasov equation, we need to show that \(Z(\O(H_T, M))\) (and its derivatives) is small.
The following proposition partially answers this question, in the sense that the difference between derivatives of \(\O\) and \(\O_{\Sch}\) are small.
\begin{prop}[Refined estimates on the period function]\label{prop:PeriodRefined}
    For \(I\le N + 1\), we have the following estimate: 
    \[
    \sum_{i_1+i_2\le I} \sup_{(H_T, M)} \abs{\rd_{H_T}^{i_1}\rd_M^{i_2}(\O(H_T, M) - \O_{\Sch}(H_T, M))} \lesssim \sum_{i\le N+2} \sup_R \abs{\rd_R^i(\bar{\a}-\a_0, \Abar-A_0, \bar{\b}-\b_0)(R)},
    \]
    where the supremum on the LHS is taken over \((H_T, M)\) in the domain of interest.
    In particular, for \(I\le N\) we have 
    \begin{equation}\label{Orefined}
    \abs{\rd_{H_T}^{i_1}\rd_M^{i_2}(\O(H_T, M) - \O_{\Sch}(H_T, M))} \lesssim \d^{3/4}\e.    
    \end{equation}
\end{prop}
\begin{proof}
Throughout we suppress the time argument: all metric coefficients are evaluated at time \(T\), so that the barred coefficients \(\bar{\a}, \Abar, \bar{\b}\) are the solution coefficients at time \(T\), while \(\a_0, A_0, \b_0 = 0\) are the Schwarzschild coefficients \eqref{sch-metric}. We abbreviate the right-hand side of the asserted estimate by
\[
\Phi \coloneqq \sum_{i\le N+2}\sup_R \abs{\rd_R^i(\bar{\a}-\a_0, \Abar-A_0, \bar{\b}-\b_0)(R)}.
\]
The argument follows the proof of Proposition 5.21 of \cite{CL24}. 
For \(\l\in[0, 1]\) set
\[
\a_\l = \a_0 + \l(\bar{\a}-\a_0), \quad A_\l = A_0 + \l(\Abar-A_0), \quad \b_\l = \l\bar{\b},
\]
and, in analogy with \eqref{VTdef} and \eqref{ptintegrand},
\[
V_\l(M, R) = (\a_\l^2-\b_\l^2 A_\l^2)\Big(1+\frac{M}{R^2 A_\l^2}\Big), \qquad
W_\l^{(1)}(R) = \frac{A_\l\a_\l}{\a_\l^2-\b_\l^2 A_\l^2}.
\]
Let \(\big(\frac{3}{2}+\sqrt{2}\big)m < R_-^\l(H_T, M) < R_+^\l(H_T, M)\) be the two roots of \(H_T^2 = V_\l(M, R)\), and define
\[
\frac{\pi}{\O_\l(H_T, M)} = \int_{R_-^\l(H_T, M)}^{R_+^\l(H_T, M)} \frac{W_\l^{(1)}(R)\,H_T}{\sqrt{H_T^2-V_\l(M, R)}}\,dR.
\]
By \eqref{sch-metric} and the weak bootstrap assumption \eqref{wBA0}, for every \(i\le N+2\),
\[
\abs{\rd_R^i(\a_\l-\a_0, A_\l-A_0, \b_\l)} = \l\,\abs{\rd_R^i(\bar{\a}-\a_0, \Abar-A_0, \bar{\b})} \le \Phi
\quad\text{uniformly in }\l\in[0, 1].
\]
In particular each \(V_\l\) is an \(\Ocal(\Phi)\) perturbation of \(V_0\), with a single non-degenerate minimum and with \(\rd_R V_\l(R_\pm^\l, M)\) bounded away from \(0\) uniformly in \(\l\), exactly as in Proposition \ref{prop:suppf} and the discussion following \eqref{wpmdef}. 
By construction we have \(\O_1 = \O\) and \(\O_0 = \O_{\Sch}\), moreover \(\rd_\l \a_\l = \bar{\a}-\a_0\), etc., so that
\begin{equation}\label{lambdaderiv}
\abs{\rd_R^i\rd_\l V_\l},\ \abs{\rd_R^i\rd_\l W_\l^{(1)}} \lesssim \Phi \qquad (i\le N+1),
\end{equation}
uniformly in \(\l\), since each is a polynomial expression in the coefficients (with denominators bounded away from \(0\)) at least one of whose factors is \(\rd_\l(\a_\l, A_\l, \b_\l) = (\bar{\a}-\a_0, \Abar-A_0, \bar{\b})\).

Finally we pass from the half-periods to \(\O\) itself. By Proposition \ref{prop:regDAAV} (and the analogous, simpler statement for the background) the maps \((H_T, M)\mapsto \pi/\O_\l\) are \(C^{N+2}\), and by Propositions \ref{prop:suppf} and \ref{prop:monotonic} the frequencies \(\O_\l\) are bounded above and below by positive constants, uniformly in \(\l\). Since
\[
\O - \O_{\Sch} = -\frac{\O\,\O_{\Sch}}{\pi}\Big(\frac{\pi}{\O}-\frac{\pi}{\O_{\Sch}}\Big),
\]
and \(\O, \O_{\Sch}\) together with their \((H_T, M)\)-derivatives up to order \(N+1\) are bounded, the Leibniz rule reduces the first asserted estimate to
\[
\sup_{(H_T, M)}\sum_{i_1 + i_2 \le N+1}\abs{\rd_{H_T}^{i_1}\rd_M^{i_2}\Big(\frac{\pi}{\O(H_T, M)}-\frac{\pi}{\O_{\Sch}(H_T, M)}\Big)} \lesssim \Phi.
\]
Writing \(\frac{\pi}{\O}-\frac{\pi}{\O_{\Sch}} = \int_0^1 \frac{d}{d\l}\frac{\pi}{\O_\l}\,d\l\), it suffices to prove
\begin{equation}\label{lambdatemp}
\sup_{(H_T, M)}\sum_{i_1+i_2\le N+1}\abs{\rd_{H_T}^{i_1}\rd_M^{i_2}\frac{d}{d\l}\frac{\pi}{\O_\l(H_T, M)}} \lesssim \Phi
\qquad\text{ uniformly in }\l\in[0, 1].
\end{equation}
The point of the interpolation is that, by \eqref{lambdaderiv}, every \(\l\)-derivative produces a factor of the deviation \(\Phi\), which is the smallness on the right-hand side.

To localize the singularities, let \(R_c(M)\) be the minimizer of \(V_0(R, M)\) as a function of \(R\) and let \(\bfr>0\) be the fixed small constant of the proof of Proposition \ref{prop:regDAAV}, chosen so that \(\abs{\rd_R V_\l(R, M)}\) is bounded away from \(0\) for \(\abs{R-R_c(M)}\ge \bfr\) and all \(\l\in[0, 1]\). 
(Note that we want a bit more on the choice of \(\bfr\), as this means the local minimum of \(V_\l(R, M)\) is in the interval \([R_c(M)-\bfr, R_c(M)+\bfr]\).)
Split
\[
\frac{\pi}{\O_\l}
= \underbrace{\int_{R_-^\l}^{R_c-\bfr}}_{\eqqcolon\,\Ical_\l^{(1)}}
+ \underbrace{\int_{R_c-\bfr}^{R_c+\bfr}}_{\eqqcolon\,\Ical_\l^{(2)}}
+ \underbrace{\int_{R_c+\bfr}^{R_+^\l}}_{\eqqcolon\,\Ical_\l^{(3)}}
\frac{W_\l^{(1)}(R)\,H_T}{\sqrt{H_T^2-V_\l(M, R)}}\,dR,
\]
and prove \eqref{lambdatemp} for each of \(\Ical_\l^{(1)}, \Ical_\l^{(2)}, \Ical_\l^{(3)}\) separately.

\step{1} (On \([R_-^\l, R_c-\bfr]\): \(\Ical_\l^{(1)}\).)
Applying Lemma \ref{lem:intsep} with \(\g = W_\l^{(1)} H_T\), and then differentiating in \(\l\) and integrating by parts in \(R\) (via \(\rd_\l(H_T^2-V_\l)^{-1/2} = \frac{\rd_\l V_\l}{\rd_R V_\l}\rd_R(H_T^2-V_\l)^{-1/2}\), as is done repeatedly above), we obtain
\begin{align*}
\Ical_\l^{(1)}
&= -\frac{2 W_\l^{(1)}(R_-^\l)\,H_T}{\rd_R V_\l(R_-^\l, M)}\sqrt{H_T^2-V_\l(M, R_c-\bfr)}
+ \int_{R_-^\l}^{R_c-\bfr} \frac{\tilde{W}_\l^{(1)}(R)\,H_T}{\sqrt{H_T^2-V_\l(M, R)}}\,dR, \\
\frac{d}{d\l}\Ical_\l^{(1)}
&= \big(\text{regular terms, evaluated at } R_c-\bfr\big) \\
&\qquad
+ \int_{R_-^\l}^{R_c-\bfr} \bigg(\rd_\l \tilde{W}_\l^{(1)}(R) - \rd_R\Big[\tilde{W}_\l^{(1)}(R)\,\frac{\rd_\l V_\l(M, R)}{\rd_R V_\l(R, M)}\Big]\bigg)\frac{H_T}{\sqrt{H_T^2-V_\l(M, R)}}\,dR,
\end{align*}
where \(\tilde{W}_\l^{(1)}\) is the desingularized weight of Lemma \ref{lem:intsep}, so that \(\tilde{W}_\l^{(1)}(R_-^\l) = 0\) and the moving-endpoint contribution at \(R_-^\l\) vanishes; the displayed boundary terms are regular because \(H_T^2-V_\l\) and \(\rd_R V_\l(R_-^\l, M)\) are bounded away from \(0\). 
Every term on the right carries a factor \(\rd_\l V_\l\) or \(\rd_\l W_\l^{(1)}\) (the former also through \(\rd_\l R_-^\l = -\rd_\l V_\l(M, R_-^\l)/\rd_R V_\l(R_-^\l, M)\)), hence is \(\Ocal(\Phi)\) by \eqref{lambdaderiv}, and the surviving integral is again of the form to which Lemma \ref{lem:intsep} applies. 
The derivatives \(\rd_{H_T}^{i_1}\rd_M^{i_2}\) are then taken exactly as in the proof of Proposition \ref{prop:regDAAV}, each distributed using the analogues of \eqref{dHTR-}--\eqref{dMR-} for \(R_-^\l\) and remove singularity via Lemma \ref{lem:intsep}, each application costing one further \(R\)-derivative of the coefficients. 
Since \(i_1+i_2\le N+1\), at most \(N+2\) coefficient derivatives occur, all entering through \(V_\l\) and \(W_\l^{(1)}\) and therefore can be controlled by \(\Phi\). 
This proves \eqref{lambdatemp} for \(\Ical_\l^{(1)}\).

\step{2} (On \([R_c-\bfr, R_c+\bfr]\): \(\Ical_\l^{(2)}\).)
On \([R_c-\bfr, R_c+\bfr]\) the denominator \(H_T^2-V_\l\) is bounded away from \(0\) uniformly in \(\l\), and the split points \(R_c\pm\bfr\) are independent of \(\l\), so
\[
\frac{d}{d\l}\Ical_\l^{(2)}
= \int_{R_c-\bfr}^{R_c+\bfr}\bigg(\frac{(\rd_\l W_\l^{(1)})\,H_T}{\sqrt{H_T^2-V_\l}}
+ \frac{W_\l^{(1)} H_T\,(\rd_\l V_\l)}{2\,(H_T^2-V_\l)^{3/2}}\bigg)dR.
\]
By \eqref{lambdaderiv} the integrand and its \(\rd_{H_T}^{i_1}\rd_M^{i_2}\) derivatives are \(\Ocal(\Phi)\), which costs one less derivative than in Step 1. 
Note that \(\rd_M\) picks the boundary term but this is also regular, as the denominator is away from zero. 
This gives \eqref{lambdatemp} for \(\Ical_\l^{(2)}\).

\step{3} (On \([R_c+\bfr, R_+^\l]\): \(\Ical_\l^{(3)}\).)
This is identical to Step 1 upon using the symmetric version of the argument near \(R_+^\l\).

\medskip
Combining the above items gives \eqref{lambdatemp}, hence the first asserted estimate with right-hand side \(\Phi\). 
Finally, for \(i_1+i_2 = I\le N\) the weak bootstrap assumption \eqref{wBA0} yields \(\Phi\lesssim \d^{3/4}\e\), which is \eqref{Orefined}. 
\end{proof}

\subsection{Vlasov equation in the new coordinate}
We can write the Vlasov equation in the new coordinate system as follows.
\begin{prop}[Vlasov equation in dynamical action angle variables]\label{prop:vlasovDAAV}
    The Vlasov equation \(\Dfr f = 0\) can be rewritten in the dynamical action angle coordinate system as 
    \begin{equation}\label{vlasov-DAAV}
    \rd_t f + \O(H, M)\rd_{Q_T} f  + \Pcal(t, T, Q_T, H, M)\rd_{Q_T} f + \Dfr H \rd_H f = 0.
    \end{equation}
    where \(\O\) is the period function defined by the metric coefficients at time \(T\). Here, \(\Pcal\) is given by
    \begin{equation}
    \begin{aligned}\label{Pcaldef}
    \Pcal
    &= \frac{1}{A}\left[\rd_w(H -H_T)\rd_R Q_T - \rd_R(H-H_T)\rd_w Q_T 
    - (A-\Abar)\O(H_T, M) - w\rd_s A\cdot \rd_w Q_T\right] \\
    &\qquad + \O(H_T, M) - \O(H, M)\,
    \end{aligned}
    \end{equation}
    where each of the terms on the RHS is evaluated at \((s = t, R(Q_T, H, M), w(Q_T, H, M), \ell = M)\).
\end{prop}
\begin{rmk}[The coefficient of \(\rd_{Q_T}\)]
    We choose \(\O(H, M)\) to be the coefficient of the main linear term.
    Given the definition of \(Q_T\), it is more natural to use \(\O(H_T, M)\) and regard the difference as a nonlinear error term, but as we use vector fields \(\rd_{Q_T}\), \(\rd_H\), and \(Z\) for the transport estimate, writing the equation as \eqref{vlasov-DAAV} simplifies calculations.
    The \(\Pcal\) term is the nonlinear error that is small and decaying. 
    From our bootstrap assumptions, we have the estimate \eqref{Pcalest} on \(\Pcal\), which is sufficient for our desired control on \(f\). 
\end{rmk}
\begin{proof}
We have \(\Dfr t = 1\), \(\Dfr M = 0\), so the only thing that we have to prove is that \(\Dfr Q_T = \O(H, M) + \Pcal\).
\begin{align*}
&\Dfr Q_T - \O(H, M) \\
&= \big(\frac{\a w}{AE} - \b\big)\rd_R Q_T
+ \big(\frac{\a}{A^3}\frac{\ell}{R^3E} + \a\frac{A'}{A^4}\frac{\ell}{R^2E}-\frac{\a'}{A}E+ \a Kw\big)\rd_w Q_T - \O(H, M)\\
&= \frac{1}{A}\rd_w H \rd_R Q_T - \frac{1}{A}\rd_R H\rd_w Q_T - \frac{w}{A}\rd_s A\rd_w Q_T - \O(H, M) \\
&= \frac{1}{A}\left[\rd_w(H -H_T)\rd_R Q_T - \rd_R(H-H_T)\rd_w Q_T\right]
    + \Big(\frac{\Abar}{A}-1\Big)\O(H_T, M) - \frac{w}{A}\rd_s A\rd_w Q_T \\
&\qquad + \O(H_T, M) - \O(H, M) \\
&= \Pcal.
\end{align*}
We used the identity \eqref{JacCalc} in the penultimate identity. 
This gives the desired conclusion.
\end{proof}
\begin{rmk}[The choice of \((t, Q_T, H, M)\) coordinates]\label{rmk:coordchoice}
    It seems alternative choices for \(Q_T\) and \(H\) as our coordinates do not give the same decay.
    We might consider the action angle \(Q\) in \eqref{Qdef} or the energy \(H_0\) in \eqref{H0def} on the Schwarzschild background, but then coefficients that depend on the metric difference \((\a-\a_0, A-A_0, \b-\b_0)\) will appear, which we do not expect to decay in time. 
    We may consider the action angle that is defined with the metric coefficient at time \(s\) (not the fixed time \(T\)), but then there is no reason to believe that this angle is a regular function of \((s, R, w, \ell)\): the \(\rd_s\) derivative picks up the derivative of metric coefficients, so the integral is divergent at the endpoint. 
    We may also consider using \(H_T\) instead of \(H\), but then the coefficient of \(\rd_{H_T}\), which is \(\Dfr H_T\) given in \eqref{DfrHT}, involves \(\rd_R(\a-\bar{\a}, A-\Abar)\), which decays one power slower considering our bootstrap assumptions \eqref{BA2}.
\end{rmk}

\section{Estimates on the Distribution Function and Its Derivatives}\label{sec:vlasov}
In this section, we use the Vlasov equation in dynamical action angle coordinates \eqref{vlasov-DAAV} to derive the growth estimates on \(\rd f\). 
As in the linear case, there is a distinction between good and bad derivatives, and we will see that \(\rd_{Q_T}, Z\) do not cause additional growth, while \(\rd_H\) causes growth in time \(\jt\). 
In particular, for the top order derivative \(\rd^{N+1} f\), we will see that we lose one more \(\jt\) as we assumed so in the bootstrap assumptions for the metric coefficients. 
The loss of decay for the metric coefficients will be apparent in the next section. 

First, we define the vector field \(Y_H\) which depends on \(t\). 
\begin{defn}[The vector field \(Y_H\)]\label{def:YH}
    At \((t, Q_T, H, M)\), we define \(Y_H\) by 
    \begin{equation}\label{YHdef}
    Y_H = t\rd_H(\O(H, M))\rd_{Q_T} + \rd_H.
    \end{equation} 
\end{defn}
\begin{rmk}
    The point of this definition is that the commutator between \(Y_H\) and the linear part of the transport operator \(\Dfr^\lin \coloneqq \rd_t + \O(H, M)\rd_{Q_T}\) vanishes.
    This corresponds to the vector field \(t\rd_x + \rd_v\) for the Vlasov equation in the flat spacetime. 
\end{rmk}

Now we can state the main result of this section.
\begin{thm}[Estimates on \(f\) derivatives]\label{thm:Vlasov}
    Under the bootstrap assumptions \eqref{BA0}--\eqref{BAHK1}, we have the following estimates on \(\rd f\).
    \begin{enumerate}
        \item For \(i_1 + i_2 + i_3 \le N\), we have 
        \begin{equation}\label{flow}
        \norm{(\rd_{Q_T}^{i_1}\rd_H^{i_2}Z^{i_3} f)(t, Q_T, H, M)}_{L^\infty} 
        \lesssim \d\e\jt^{i_2}.
        \end{equation}
        \item For \(1 + i_1 + i_2 + i_3 \le N\), we have 
        \begin{equation}\label{flowYH}
        \norm{(Y_H\rd_{Q_T}^{i_1}\rd_H^{i_2}Z^{i_3} f)(t, Q_T, H, M)}_{L^\infty} 
        \lesssim \d\e\jt^{i_2}.
        \end{equation}
        \item For \(i_1 + i_2 + i_3 =  N + 1\), we have
        \begin{equation}\label{ftop}
        \norm{(\rd_{Q_T}^{i_1}\rd_H^{i_2}Z^{i_3} f)(t, Q_T, H, M)}_{L^\infty} 
        \lesssim \d\e\jt^{i_2+1}.
        \end{equation}
    \end{enumerate}
    As before the implicit constant in the above estimates depends on \(N\) and the background parameters \(m, \cfr\), but is independent of \(\d\), \(\e\), and \(t\).
\end{thm}
\begin{rmk}
    Contrary to the Vlasov--Poisson on torus case, we do not gain faster decay from using more \(Y_H\) derivatives: the main reason is that \(\rd_{Q_T}\) which is on the tangent space of the spacetime might have zero spatial part, hence the vector field \(Y_H\) does not give better decay of the integrated quantities. 
    This limits us to use at most one \(Y_H\) derivative for \(f\), but nevertheless this fact that we can use it once is important when we estimate integrals of \(f\) to get the metric coefficients.
    This will be clearer in Section \ref{sec:Integral}. 
\end{rmk}

The proof of this theorem is again a bootstrap argument.
The next proposition displays the bootstrap assumptions for \(f\) derivatives. 
\begin{prop}[Bootstrap argument for \(f\) derivatives]\label{prop:BTVlasov}
    For given small \(\d, \e > 0\), say we have the following estimates on \(f\) for \(t\in [0, T')\) with \(T' < T_B\):
    \begin{align}
    \norm{(\rd_{Q_T}^{i_1}\rd_H^{i_2}Z^{i_3} f)(t, Q_T, H, M)}_{L^\infty} 
    &\le \d^{3/4}\e\jt^{i_2}, 
    &&\text{ for } i_1 + i_2 + i_3 \le N,
    \label{fBAlow}\\
    \norm{(Y_H\rd_{Q_T}^{i_1}\rd_H^{i_2}Z^{i_3} f)(t, Q_T, H, M)}_{L^\infty} 
    &\le \d^{3/4}\e\jt^{i_2}, 
    &&\text{ for } 1 + i_1 + i_2 + i_3 \le N,
    \label{fBAlowYH}\\
    \norm{(\rd_{Q_T}^{i_1}\rd_H^{i_2}Z^{i_3} f)(t, Q_T, H, M)}_{L^\infty} 
    &\le \d^{3/4}\e\jt^{i_2+1}, 
    &&\text{ for } i_1 + i_2 + i_3 = N + 1.
    \label{fBAtop}
    \end{align}
    Then, there is a constant \(C > 0\) that is independent of \(\d\), \(\e\), and \(t\) such that the improved version of the above estimates with the constant \(C\d\) instead of \(\d^{3/4}\) holds.
\end{prop}
The rest of this section is devoted to the proof of this proposition, which directly implies Theorem \ref{thm:Vlasov} by the standard continuity argument.
In each subsection, we analyze commutator terms coming from the \(\O(H, M)\rd_{Q_T}\) term, the nonlinearity \(\Pcal\rd_{Q_T} f\), and the \(\Dfr H\rd_H f\) term separately, and finally the initial data term. 
In the last subsection, we will combine all the estimates to retrieve the bootstrap assumptions.

\subsection{Setting up commutator estimates}
To derive the estimate on \(f\) derivatives, we commute our Vlasov equation \eqref{vlasov-DAAV} with \(Y_H^{i_0}\rd_{Q_T}^{i_1}\rd_H^{i_2}Z^{i_3}\).
We need the following identity for \(i_0 =1\) and \(i_1+i_2+i_3\le N-1\), or \(i_0 = 0\) and \(i_1 +i_2+i_3 \le N+1\). 
\begin{prop}\label{prop:commutedeqn}
    For \(f\) satisfying \eqref{vlasov-DAAV}, the derivative \(Y_H^{i_0}\rd_{Q_T}^{i_1}\rd_H^{i_2}Z^{i_3} f\) satisfies the following equation: 
    \begin{equation}\begin{aligned}\label{commutedeqn}
    &\Big(\rd_t + \O(H, M)\rd_{Q_T} + \Pcal\rd_{Q_T} + \Dfr H\rd_H\Big)(Y_H^{i_0}\rd_{Q_T}^{i_1}\rd_H^{i_2}Z^{i_3} f) \\
    &= Y_H^{i_0}[\O(H, M)\rd_{Q_T}, \rd_{Q_T}^{i_1}\rd_H^{i_2}Z^{i_3}]f 
    + [\Pcal\rd_{Q_T}, Y_H^{i_0}\rd_{Q_T}^{i_1}\rd_H^{i_2}Z^{i_3}]f 
    + [\Dfr H\rd_H, Y_H^{i_0}\rd_{Q_T}^{i_1}\rd_H^{i_2}Z^{i_3}]f \\
    &\eqqcolon T_1 + T_2 + T_3.
    \end{aligned}\end{equation}
\end{prop}
\begin{proof}
This follows from the direct calculation. 
For \(\Dfr^\lin\), we first commute with \(Y_H\), and for the commutator with \(\rd_{Q_T}^{i_1}\rd_H^{i_2}Z^{i_3}\) we used the fact that each of those vectors commutes with \(\rd_t\). 
(Note that \(Z\) is defined in terms of \(H\) and \(M\), not \(H_T\).)
\end{proof}
We analyze \(T_1\) in Subsection \ref{subsec:commutator}, \(T_2\) in Subsection \ref{subsec:Pcal}, and \(T_3\) in Subsection \ref{subsec:DfrH}.
Comparing \(T_2\) with \(T_3\), note that \(\rd_{Q_T}\) is a good derivative whereas \(\rd_H\) causes additional loss of \(\jt\). 
This is compensated from the fact that \(\Dfr H\) has one less derivative than \(\Pcal\), hence gains one more \(\jt\) than \(\Pcal\) does.
For \(T_1\), the situation when we commute with \(\rd_H\) and \(\rd_{Q_T}\) or \(Z\) can be compared similarly.
When we commute with \(\rd_H\), the \(\rd_H\) on \(f\) is exchanged to \(\rd_{Q_T}\), hence the commutator term integrated in time will not cause additional growth. 
On the other hand, when we commute with \(\rd_{Q_T}\) or \(Z\), the smallness of commutator term follows from the fact that \(\rd_{Q_T}\) or \(Z\) of \(\O(H, M)\) is small (or zero). This will be elaborated more in the next subsection. 

\subsection{Estimates on the derivative of \(\O(H, M)\)}\label{subsec:commutator}
In this subsection, we use \((X, Y)\) to be the dummy variables for the first and second variable of \(\O\) respectively; i.e. we will use \(\rd_X \O\) and \(\rd_Y \O\) to denote the partial derivative of \(\O\) with respect to the first and second variable.
We have the following expression on \(Z(\O(H, M))\):
\begin{align}\label{ZO}
\begin{aligned}
Z(\O(H, M)) &= ((\rd_M\O_\Sch)(H, M)-(\rd_M\O)(H, M))(\rd_H\O_\Sch)(H, M) \\
&\qquad
- ((\rd_H\O_\Sch)(H, M)-(\rd_H\O)(H, M))(\rd_M\O_\Sch)(H, M).
\end{aligned}
\end{align}
To get the estimates on the higher order derivatives of \(\O\), we need the following lemma. 
\begin{lem}\label{lem:Oderivest}
    For \(\O(H, M)\) defined in \eqref{Odef}, we have the following estimates:
    \begin{enumerate}
        \item For \(i_0 + i_1 + i_2 + i_3 \le N\), we have 
        \begin{align}\label{Oderivlow}
            \norm{Y_H^{i_0}\rd_{Q_T}^{i_1}\rd_H^{i_2}Z^{i_3}(\O(H, M))}
            &\le C\d^{3/4}\e &&\text{ if } i_1 + i_3 > 0, \\
            \norm{Y_H^{i_0}\rd_{Q_T}^{i_1}\rd_H^{i_2}Z^{i_3}(\O(H, M))}
            &\le C &&\text{ if } i_1 + i_3 = 0.
        \end{align}
        \item For \(i_1 + i_2 + i_3 = N+1\), we have 
        \begin{equation}\label{Oderivtop}
            \norm{\rd_{Q_T}^{i_1}\rd_H^{i_2}Z^{i_3}(\O(H, M))}
            \le C.
        \end{equation}
    \end{enumerate}
\end{lem}
\begin{proof}
First, note that \(\O\) is a function of \(H\) and \(M\) only; hence for \(Y_H\), it suffices to consider the \(\rd_H\) part, which means without loss of generality we can assume \(i_0 = 0\), and for \(Q_T\) every derivative is zero, so the result for the case \(i_1 > 0\) is trivial. 
Moreover, since we know that \(\O\) as well as \(\O_\Sch\) has bounded \((N+2)\)-th order derivative, any \(\rd_H\) and \(Z\) derivative that we consider is bounded. 
Therefore, the only inequality that we need to prove is 
\[
\norm{\rd_H^{i_2}Z^{i_3}(\O(H, M))} \le C\d^{3/4}\e,
\]
for \(i_3 > 0\), \(i_2 + i_3 \le N\). 

From \eqref{ZO}, we have 
\begin{align*}
\rd_H^{i_2}Z^{i_3}(\O(H, M))
&= (\rd_H^{i_2}Z^{i_3-1})\Big[((\rd_M\O_\Sch)(H, M)-(\rd_M\O)(H, M))(\rd_H\O_\Sch)(H, M)\Big] \\
&\qquad
- (\rd_H^{i_2}Z^{i_3-1})\Big[((\rd_H\O_\Sch)(H, M)-(\rd_H\O)(H, M))(\rd_M\O_\Sch)(H, M) \big].
\end{align*}
From Proposition \ref{prop:PeriodRefined} with \(I \le N\), together with the fact that \(\O_\Sch\) has bounded derivatives, we know that the above quantity is bounded by \(C\d^{3/4}\e\).
Note that we used the full strength of Proposition \ref{prop:PeriodRefined} here; for if \(i_2 + i_3 = N+1\) we cannot run the same argument. 
\end{proof}

With this lemma, we can control the commutator term \(T_1\) in \eqref{commutedeqn} by the smallness of the derivatives of \(\O(H, M)\) and the bootstrap assumptions for \(f\) derivatives.
Note that we keep the derivatives of \(f\) on RHS, which will be cleared during the main induction step.
The main reason of doing so is because we do not have enough smallness for these linear terms in the equation. 
\begin{lem}\label{lem:VlasovT1}
    For \(T_1\) defined in \eqref{commutedeqn} with given \(i_0, i_1, i_2, i_3\), we have the following estimates: 
    \begin{align}
    \abs{T_1} &\lesssim \d\e^2\jt^{i_2} + \sum_{\substack{j_0\le i_0, j_1\le i_1, \\j_2 < i_2, j_3 \le i_3}} \abs{Y_H^{j_0}\rd_{Q_T}^{1+j_1}\rd_H^{j_2}Z^{j_3} f} 
    &&(i_0 + i_1 + i_2 + i_3 \le N), \label{T1low} \\
    \abs{T_1} &\lesssim \d\e^2\jt^{i_2+1} + \sum_{\substack{j_0\le i_0, j_1\le i_1, \\j_2 < i_2, j_3 \le i_3}} \abs{Y_H^{j_0'}\rd_{Q_T}^{1+j_1'}\rd_H^{j_2'}Z^{j_3'} f}
    + \abs{\rd_{Q_T}f}
    &&(i_0 = 0, i_1 + i_2 + i_3 = N+1). \label{T1top}
    \end{align}
\end{lem}
\begin{proof}
For \(T_1\), we have an identity 
\[
T_1 = Y_H^{i_0}[\O(H, M), \rd_{Q_T}^{i_1}\rd_H^{i_2}Z^{i_3}]\rd_{Q_T} f,
\]
because \(\rd_{Q_T}\) commutes with \(\rd_H\) and \(Z\). 
This allows us to control \(T_1\) as follows. 
Here, we sum over \((j_0, j_1, j_2, j_3)\) and \((j_0', j_1', j_2', j_3')\) with \(j_0+j_0' \le i_0\), \(j_1+j_1' \le i_1\), \(j_2+j_2' \le i_2\), and \(j_3+j_3' \le i_3\), hence the sum of indices is at most \(i_0+i_1+i_2+i_3\). 
If \(i_0+i_1+i_2+i_3\le N\), we have
\begin{align*}
\abs{T_1} &\le C\sum_{j_1+j_2+j_3 > 0} \abs{Y_H^{j_0}\rd_{Q_T}^{j_1}\rd_H^{j_2}Z^{j_3}(\O(H, M))} \cdot \abs{Y_H^{j_0'}\rd_{Q_T}^{1+j_1'}\rd_H^{j_2'}Z^{j_3'} f} \\
&\le C\d^{3/4}\e \cdot \d^{3/4}\e\jt^{i_2} 
+ C\cdot \sum_{j_2' < i_2}\abs{Y_H^{j_0'}\rd_{Q_T}^{1+j_1'}\rd_H^{j_2'}Z^{j_3'} f} \\
&\le C\d\e^2\jt^{i_2} + C \sum_{j_2' < i_2} \abs{Y_H^{j_0'}\rd_{Q_T}^{1+j_1'}\rd_H^{j_2'}Z^{j_3'} f}.
\end{align*}
There are at most \(N\) derivatives on \(f\) in the above sum, and we can control it by the bootstrap assumptions \eqref{fBAlow} and \eqref{fBAlowYH}.
Note that for the second term, the situation that we cannot use the improved bound for \(\O\) is when \(j_1 = j_3 = 0\), and in this case \(j_2 > 0\) and \(j_2' < i_2\). \\
Else if \(i_1 + i_2 + i_3 = N+1\), there are at most \(N+1\) derivatives on \(f\), hence we can proceed as follows: 
\begin{align*}
\abs{T_1} &\le C\sum_{j_1+j_2+j_3 > 0} \abs{\rd_{Q_T}^{j_1}\rd_H^{j_2}Z^{j_3}(\O(H, M))} \cdot \abs{\rd_{Q_T}^{1+j_1'}\rd_H^{j_2'}Z^{j_3'} f} \\
&\le C\d^{3/4}\e \cdot \d^{3/4}\e\jt^{i_2 + 1} 
+ C \sum_{j_2' < i_2} \abs{\rd_{Q_T}^{1+j_1'}\rd_H^{j_2'}Z^{j_3'} f}
+ C\abs{\rd_{Q_T}f} \\
&\le C\d\e^2\jt^{i_2+1} + C \sum_{j_2' < i_2} \abs{\rd_{Q_T}^{1+j_1'}\rd_H^{j_2'}Z^{j_3'} f}
+ C\abs{\rd_{Q_T}f},
\end{align*}
where we used the bootstrap assumption \eqref{fBAlow}--\eqref{fBAtop}. 
We have an extra term here for the case when all the derivatives hit \(\O\), where we cannot use the improved bound for \(\O\).
\end{proof}

\subsection{Estimates on \(\Pcal\)}\label{subsec:Pcal}
\(\Pcal\) has been defined in \eqref{Pcaldef} as the coefficient of \(\rd_{Q_T}f\) in the Vlasov equation \eqref{vlasov-DAAV}.
From \eqref{BA0}, \eqref{BA1}, \eqref{BA2}, and \eqref{BAH1}, and Proposition \ref{prop:regDAAV}, we have the following decay estimates for \(\Pcal\) and its derivatives:
\begin{equation}\label{Pcalest}
    \abs{\rd_{Q_T, H, M}^I\Pcal} \lesssim \d^{3/4}\e\jt^{-N+\max{I-1, 0}}, \quad \text{ for } I \le N+1.
\end{equation}
Note that for the term \(\O(H_T, M)-\O(H, M)\) we used mean value theorem, hence the fact that \(\O\) has bounded \(N+2\) derivatives was necessary.
With this, we can estimate \(T_2\) as follows: 
\begin{lem}\label{lem:vlasovT2}
    For \(T_2\) defined in \eqref{commutedeqn} with given \(i_0, i_1, i_2, i_3\), we have the following estimate: 
    \begin{equation}\label{T2}
        \abs{T_2} \lesssim \begin{cases}
            \d\e^2\jt^{i_2} &\text{ for } i_0 + i_1 + i_2 + i_3 \le N, \\
            \d\e^2\jt^{i_2+1} &\text{ for } i_0 = 0, i_1 + i_2 + i_3 = N+1.
        \end{cases} 
    \end{equation}
\end{lem}
\begin{proof}
We can proceed as follows: note that \(\rd_{Q_T}\) commutes with \(Y_H\), \(\rd_H\) and \(Z\) so we have the first identity. 
Here, we sum over \((j_0, j_1, j_2, j_3)\) and \((j_0', j_1', j_2', j_3')\) with \(j_0+j_0'\le i_0\), \(j_1+j_1'\le i_1\), \(j_2+j_2'\le i_2\), and \(j_3 + j_3' \le i_3\), hence the sum of indices is at most \(i_0+i_1+i_2+i_3\).
\begin{align*}
\abs{T_2} &= \abs{[\Pcal, Y_H^{i_0}\rd_{Q_T}^{i_1}\rd_H^{i_2}Z^{i_3}]\rd_{Q_T}f} \\
&\le C\sum_{j_0 + j_1 + j_2 + j_3 > 0} \abs{Y_H^{j_0}\rd_{Q_T}^{j_1}\rd_H^{j_2}Z^{j_3}(\Pcal)} \cdot \abs{Y_H^{j_0'}\rd_{Q_T}^{1+j_1'}\rd_H^{j_2'}Z^{j_3'} f}.
\end{align*}
If \(i_0 + i_1 +i_2 + i_3 \le N\), the \(\Pcal\) term is bounded by \(\d^{3/4}\e\), and there are at most \(N\) derivatives on \(f\) in the above sum, hence from \eqref{Pcalest} and \eqref{fBAlowYH} we have 
\[
\abs{T_2}\lesssim \d\e^2\jt^{i_2},
\]
which is the desired estimate. \\
For \(i_0 = 0\), \(i_1 + i_2 + i_3 \le N+1\) case, because one more \(\jt\) is acceptable we can proceed as follows: 
\begin{equation*}
\abs{T_2} \lesssim \d^{3/4}\e \cdot \d^{3/4}\e\jt^{i_2 + 1} \lesssim \d\e^2\jt^{i_2 + 1},
\end{equation*}
which follows from \eqref{Pcalest} and \eqref{fBAtop}.
This completes the proof. 
\end{proof}

\subsection{Estimates on \(\Dfr H\)}\label{subsec:DfrH}
\(\Dfr H\) has been calculated in \eqref{DfrH}. From \eqref{BA0}, \eqref{BA1}, and \eqref{BAH1}, we have the following decay estimates for \(\Dfr H\) and its derivatives:
\begin{equation}\label{DfrHest}
    \abs{\rd_{Q_T, H, M}^I(\Dfr H)} \lesssim 
    \begin{cases}
        \d^{3/4}\e\jt^{-N+\max{I-2, 0}} &\text{ if } I \le N, \\
        \d^{3/4}\e &\text{ if } I = N+1. 
    \end{cases}
\end{equation}
With this, we can estimate \(T_3\) as follows: 
\begin{lem}\label{lem:vlasovT3}
    For \(T_3\) defined in \eqref{commutedeqn} with given \(i_0, i_1, i_2, i_3\), we have the following estimate: 
    \begin{equation}\label{T3}
        \abs{T_3} \lesssim \begin{cases}
            \d\e^2\jt^{i_2} &\text{ for } i_0 + i_1 + i_2 + i_3 \le N, \\
            \d\e^2\jt^{i_2+1} &\text{ for } i_0 = 0, i_1 + i_2 + i_3 = N+1.
        \end{cases} 
    \end{equation}
\end{lem}
\begin{proof}
We have the following commutation identities:
\begin{equation}\begin{aligned}\label{rdHcommutator}
[\rd_H, Y_H] &= t(\rd_{XX}^2\O)(H, M)\rd_{Q_T}, \\
[\rd_H, Z] &= \big(\rd_{XY}^2\O_\Sch - \frac{\rd_Y\O_\Sch}{\rd_X\O_\Sch}\rd_{XX}^2\O_\Sch\big)\rd_H
+ \frac{\rd_{XX}^2\O_\Sch}{\rd_X\O_\Sch} Z.
\end{aligned}\end{equation}
This with simple induction allows us to write any product of \(i_2\) \(\rd_H\)'s and \(i_3\) \(Z\)'s as a linear combination of \(\rd_H^{i_2'}Z^{i_3'}\), where \(i_2' \le i_2\) and \(i_3' \le i_3\), with bounded coefficients. 
Still, it is better to separate the case when \(i_0 = 1\). 
\begin{equation*}
T_3 = \Dfr H [\rd_H, Y_H]\rd_{Q_T}^{i_1}\rd_H^{i_2}Z^{i_3}f
+ [\Dfr H, Y_H \rd_{Q_T}^{i_1}\rd_H^{i_2}Z^{i_3}]\rd_H f
+ (\Dfr H) Y_H \rd_{Q_T}^{i_1}\rd_H^{i_2} [\rd_H, Z^{i_3}]f,
\end{equation*}
so we can estimate \(T_3\) as follows:
\begin{align*}
\abs{T_3} &\le Ct\abs{\Dfr H}\cdot \abs{\rd_{Q_T}^{1+i_1}\rd_H^{i_2}Z^{i_3}f} 
+ C\sum_{\substack{j_0 + j_1 + j_2 + j_3 > 0 \\ j_2''+j_3\le j_3'+1, j_2''\le 1}} \abs{Y_H^{j_0}\rd_{Q_T}^{j_1}\rd_H^{j_2}Z^{j_3}(\Dfr H)} \cdot \abs{Y_H^{j_0'}\rd_{Q_T}^{j_1'}\rd_H^{j_2'+j_2''}Z^{j_3''} f} \\
&\qquad
+ C\sum_{\substack{i_2'\le 1, \\i_2' + i_3' < 1 + i_3}} \abs{\Dfr H} \cdot \abs{Y_H\rd_{Q_T}^{i_1}\rd_H^{i_2+i_2'}Z^{i_3'} f}.
\end{align*}
Under the condition \(i_0 + i_1 + i_2 + i_3\le N\), from \eqref{DfrHest}, \eqref{fBAlow}, and \eqref{fBAlowYH} this can be bounded by 
\begin{align*}
\abs{T_3} &\le C\d^{3/4}\e\jt^{-N+1}\cdot \d^{3/4}\e\jt^{i_2} + C\d^{3/4}\e\jt^{-1}\cdot \d^{3/4}\e\jt^{1+i_2} \\
&\qquad
+ C\d^{3/4}\e\jt^{-N}\cdot \d^{3/4}\e\jt^{i_2+1} \\
&\le \d\e^2\jt^{i_2}.
\end{align*}

Now we consider the case when \(i_0 = 0\); in this case the first term above is absent: 
\[
T_3 = [\Dfr H, \rd_{Q_T}^{i_1}\rd_H^{i_2}Z^{i_3}]\rd_H f
+ (\Dfr H)\rd_{Q_T}^{i_1}\rd_H^{i_2} [\rd_H, Z^{i_3}]f.
\]
This allows us to estimate \(T_3\) as follows:
\begin{align*}
\abs{T_3} &\le 
C\sum_{\substack{j_1 + j_2 + j_3 > 0 \\ j_2''+j_3\le j_3'+1, j_2''\le 1}} \abs{\rd_{Q_T}^{j_1}\rd_H^{j_2}Z^{j_3}(\Dfr H)} \cdot \abs{\rd_{Q_T}^{j_1'}\rd_H^{j_2'+j_2''}Z^{j_3''} f}
+ C\abs{\Dfr H} \cdot \sum_{\substack{i_2'\le 1, \\i_2' + i_3' < 1 + i_3}} \abs{\rd_{Q_T}^{i_1}\rd_H^{i_2+i_2'}Z^{i_3'} f}.
\end{align*}
For the case \(i_1 + i_2 + i_3\le N\), we can use \eqref{DfrHest} and \eqref{fBAlow} to conclude that 
\[
\abs{T_3} \le C\d^{3/4}\e\jt^{-2} \cdot \d^{3/4}\e\jt^{1+i_2} + C\d^{3/4}\e\jt^{-N}\cdot \d^{3/4}\e\jt^{i_2+1}
\le C\d\e^2\jt^{i_2}.
\]
Note that in any case the number of derivatives on \(f\) is at most \(N\). \\
For the case \(i_1 + i_2 + i_3 = N+1\), we can use \eqref{DfrHest} and \eqref{fBAtop}: here the number of derivatives on \(f\) is at most \(N+1\), but we have an extra \(\jt\) to spare, hence we can proceed as follows:
\begin{align*}
\abs{T_3}&\le C\d^{3/4}\jt^{-2}\e\cdot\d^{3/4}\e\jt^{2+i_2} + C\d^{3/4}\e\jt^{-N}\cdot \d^{3/4}\e\jt^{i_2+2} \\
&\qquad
+ C\d^{3/4}\e\cdot \d^{3/4}\e\jt \\
&\le \d\e^2\jt^{i_2+1} 
\end{align*}
which is the desired estimate. 
Note that the last term is needed for the case when all the derivatives hit \(\Dfr H\) so that we need to use the top order estimate of \eqref{DfrHest}.
This completes the proof.
\end{proof}

\subsection{Estimates for the initial data term}
In the next subsection, we will use the Vlasov equation \eqref{vlasov-DAAV} to retrieve the bootstrap assumptions for \(\rd f\).
We already have control for the source terms \(T_1\), \(T_2\), and \(T_3\) from Lemma \ref{lem:VlasovT1}, Lemma \ref{lem:vlasovT2}, and Lemma \ref{lem:vlasovT3}.
The only remaining term is the initial data term, which can be estimated as follows:
\begin{lem}\label{lem:vlasovinit}
    For initial data \(f_\init\) with \(\supp(f_\init)\subset \Scal_{\cfr}\) and the size condition \eqref{nonlinthm-initf}, we have the following estimate for the initial data term in \eqref{commutedeqn}:
    \begin{equation}\label{init}
        \sup_{(Q_T, H, M)} \abs{Y_H^{i_0}\rd_{Q_T}^{i_1}\rd_H^{i_2}Z^{i_3}f_\init} 
        \le \d\e, \quad \text{ for } i_0 + i_1 + i_2 + i_3 \le N+1.
    \end{equation}
\end{lem}
\begin{proof}
This is a direct consequence of the size condition \eqref{nonlinthm-initf} and the regularity of the change of variables from the standard coordinates to the dynamical action angle variables, which is guaranteed by Proposition \ref{prop:regDAAV}.
Note that at \(t = 0\) we have \(Y_H = \rd_H\).
\end{proof}

\subsection{Proof of Proposition \ref{prop:BTVlasov}}
Now we have all the ingredients to prove Proposition \ref{prop:BTVlasov}.
Basically we integrate \eqref{commutedeqn} along the characteristic curves, and use the estimates for \(T_1\), \(T_2\), \(T_3\), and the initial data term to retrieve the bootstrap assumptions for \(\rd f\).
Because of the commutators arising from the linear terms, we need to be careful in the order of retrieving the bootstrap assumptions for \(\rd f\); for this we need to induct on \(i_2\) first.
\begin{proof}[Proof of Proposition \ref{prop:BTVlasov}]
We first retrieve \eqref{flow} and \eqref{flowYH}. Setting \(i_0 \le 1\) and \(i_0 + i_1 + i_2 + i_3 \le N\), from \eqref{T1low}, \eqref{T2}, \eqref{T3}, and \eqref{init}, we have 
\begin{align*}
\abs{Y_H^{i_0}\rd_{Q_T}^{i_1}\rd_H^{i_2}Z^{i_3}f} 
&\lesssim \d\e + \int_0^t \d\e^2\jap{s}^{i_2}\,ds 
+ \int_0^t \sum_{\substack{j_0\le i_0, j_1\le i_1, \\j_2 < i_2, j_3 \le i_3}} \abs{Y_H^{j_0}\rd_{Q_T}^{1+j_1}\rd_H^{j_2}Z^{j_3} f} \, ds.
\end{align*}
We induct on \(i_2\): For \(i_2 = 0\), the second integral above is absent, hence we have \(\abs{Y_H^{i_0}\rd_{Q_T}^{i_1}Z^{i_3}f} \le C\d\e\).
For given \(i_2 > 0\), if we have improved inequality \eqref{flow} and \eqref{flowYH} up to \(i_2-1\), then we can control the above as 
\[
\abs{Y_H^{i_0}\rd_{Q_T}^{i_1}\rd_H^{i_2}Z^{i_3}f} 
\lesssim \d\e\jap{t}^{i_2} 
+ \int_0^t \d\e\jap{s}^{i_2-1}\, ds \le C\d\e\jap{t}^{i_2}.
\]
This completes the retrieval of \eqref{flow} and \eqref{flowYH}.

For \eqref{ftop}, we set \(i_0 = 0\) and \(i_1 + i_2 + i_3 = N+1\); for this case there is one extra term in the estimate for \(T_1\), for which we already improved the bound.
From \eqref{T1top}, \eqref{T2}, \eqref{T3}, and \eqref{init}, we have
\[
\abs{\rd_{Q_T}^{i_1}\rd_H^{i_2}Z^{i_3}f}
\lesssim \d\e + \int_0^t \d\e^2\jap{s}^{i_2+1}\, ds
+ \int_0^t \sum_{\substack{j_1\le i_1, \\j_2 < i_2, j_3 \le i_3}} \abs{\rd_{Q_T}^{1+j_1}\rd_H^{j_2}Z^{j_3} f} \, ds
+ \int_0^t \abs{\rd_{Q_T}f} \, ds.
\]
The first and third integrals can be controlled by \(C\d\e\jt^{i_2+1}\), and we again need an induction on \(i_2\) to control the second integral.
For \(i_2 = 0\), the second integral is absent, hence we have 
\[
\abs{\rd_{Q_T}^{i_1}Z^{i_3}f} \le C\d\e\jt.
\]
For given \(i_2 > 0\), if we have improved inequality \eqref{ftop} up to \(i_2-1\), then we can control the above as 
\[
\abs{\rd_{Q_T}^{i_1}\rd_H^{i_2}Z^{i_3}f}
\lesssim \d\e\jap{t}^{i_2+1} + \int_0^t \d\e\jap{s}^{i_2} \, ds \le C\d\e\jap{t}^{i_2+1}.
\]
This completes the retrieval of \eqref{ftop}, hence concluding the proof of Proposition \ref{prop:BTVlasov}.
\end{proof}

\section{Estimates on the Density Variables}\label{sec:Integral}
In this section, we control the density variables \(\r, \tr T, j, S_R\) defined by \eqref{rhodef}--\eqref{SRdef} based on the estimates for \(f\) obtained in Proposition \ref{prop:BTVlasov} and estimates on the metric coefficients assumed in \eqref{BA0}--\eqref{BAHK1}.

\subsection{Density estimates setup}
The goal of this section is to prove the following proposition.
\begin{prop}[Density estimates]\label{prop:Density}
    Under the bootstrap assumptions \eqref{BA0}--\eqref{BAHK1}, we have the following estimates for the density variables \(\r, \tr T, j, S_R\) and their derivatives:
    \begin{align}
        \abs{\rd_R^I(\r, \tr T, j, S_R)} &\lesssim \d\e, &&\text{ for } I \le N, \label{qlow}\\
        \abs{\rd_R^I(\r, \tr T, j, S_R)} &\lesssim \d\e\jap{s}, &&\text{ for } I = N+1, \label{qtop} \\
        \abs{\rd_R^I\rd_s(\r, \tr T, j, S_R)} &\lesssim \d\e\jap{s}^{-N+I}, &&\text{ for } I \le N-2, \label{dsqlow} \\
        \abs{\rd_R^I\rd_s(\r, \tr T, j, S_R)} &\lesssim \d\e, &&\text{ for } I = N-1, \label{dsqtop} \\
        \abs{\rd_R^I(\r-\bar{\r}, \tr T-\overline{\tr T}, j-\jbar, S_R-\Sbar_R)} &\lesssim \d\e\jap{s}^{-N+I}, &&\text{ for } I \le N-1. \label{qqbarlow}        
    \end{align}
\end{prop}
Note that the desired estimates do not decay for \eqref{qlow}, \eqref{qtop}, and \eqref{dsqtop}.
These do not require to decompose \(f\) into its Fourier modes, and can be proved by direct estimates on the integral representation of the density variables. 
This will be done in the next subsection. 
For \eqref{dsqlow} and \eqref{qqbarlow} we actually exploit the phase mixing mechanism to obtain decay, and for this we need to decompose \(f\) into its Fourier modes.
This will be done in the other subsections.

To deal with four density variables at the same time, we will write them in a unified way with the following symbol: 
\begin{equation}\label{qdef}
\qbf(s, R) = \frac{\pi}{R^2A^2}\int_{-\infty}^\infty\int_0^\infty 
f(s, R, w, \ell) \bar{\CC}(s, R, w, \ell) \, d\ell dw,
\end{equation}
where \(\bar{\CC}\in \{E, -1/E, w, w^2/E\}\).
Note that the \(s\)-dependence of \(\bar{\CC}\) is only through the metric coefficients \(A\), so its \(\rd_s\) derivative can be controlled by \(\rd_s A\).

\subsection{Non-decaying estimates for the density variables} 
In this subsection, we establish the non-decaying estimates for the density variables, i.e. \eqref{qlow}, \eqref{qtop}, and \eqref{dsqtop}.
For this, we will use the fact that we integrate over \(\ell\), hence we can integrate by parts in \(\ell\). 
Hence, the following identity which represents \(\rd_R\) as a linear combination of \(\rd_\ell\) and good derivatives \(\rd_{Q_T}\), \(Z\) is useful.
\begin{lem}\label{lem:drintodell}
    The following identities hold among vectors:
    \begin{align}
    \rd_H &= \frac{\rd_X\O_\Sch(H, M)}{\rd_\ell(\O_\Sch(H, M))}
\left(\rd_\ell - (\rd_\ell Q_T)\rd_{Q_T} + \frac{1}{\rd_X\O_\Sch(H, M)}Z\right), \label{dHintoell} \\
    \rd_R &= (\rd_R Q_T)\rd_{Q_T} + (\rd_R H)\frac{\rd_X\O_\Sch(H, M)}{\rd_\ell(\O_\Sch(H, M))}
\left(\rd_\ell - (\rd_\ell Q_T)\rd_{Q_T} + \frac{1}{\rd_X\O_\Sch(H, M)}Z\right). \label{drintodell}
    \end{align}
\end{lem}
\begin{proof}
The proof of \eqref{dHintoell} is already given at \eqref{Hlincomb}. For \(\rd_R\), we have
\[
\rd_R = (\rd_R Q_T)\rd_{Q_T} + (\rd_R H)\rd_H
\]
which gives the desired identity \eqref{drintodell} by plugging in \eqref{dHintoell}.
Note here that \(\rd_\ell(\O_\Sch(H, M))\) is nonzero on the domain we consider.
\end{proof}
\begin{rmk}
    In this identity, the coefficients are guaranteed to have bounded \(N+1\)-st derivatives by Proposition \ref{prop:regDAAV}, hence do not cause any issue for the following analysis. 
\end{rmk}

Now we analyze \(\rd_R^I \qbf(s, R)\) and prove \eqref{qlow} and \eqref{qtop} using this identity.
\begin{proof}[Proof of \eqref{qlow} and \eqref{qtop}]
Notice that functions \(\bar{\CC}\) have bounded \((N+2)\)-th derivatives in \(R\), and can be indefinitely differentiated in \(w\) and \(\ell\). 
When \(\rd_R\) hits an integral of the form \(\qbf(s, R)\), then either it hits the term outside integral, or it hits \(f\), or it hits \(\bar{\CC}\).
If it hits \(f\), we can decompose \(\rd_R\) with \eqref{drintodell}, and integrate by parts \(\rd_\ell\) to \(\bar{\CC}\). 
Therefore, for \(I\le N+1\) we can bound \(\rd_R^I \qbf(s, R)\) as follows. 
\[
\abs{\rd_R^I \qbf(s, R)}
\le C\sum_{i_1 + i_2 + i_3 + i_4 + i_5\le I} \abs{\rd_R^{i_1} \big(\frac{1}{R^2A^2}\big)} 
\int_{-\infty}^\infty\int_0^\infty \abs{\rd_{Q_T}^{i_2}Z^{i_3} f}\abs{\rd_R^{i_4}\rd_\ell^{i_5} \bar{\CC}} d\ell dw.
\]
Note here that the implicit constant \(C\) absorbs all the derivative of coefficients in the identity \eqref{drintodell}, hence we crucially use the fact that the coefficients are bounded up to their \((N+2)\)-th derivative.
Now we invoke \eqref{BA1} to control the derivative of \(A\) and \(\bar{\CC}\), and Theorem \ref{thm:Vlasov} to control the derivative of \(f\).
This gives us that
\begin{align*}
\abs{\rd_R^I \qbf(s, R)}
&\le C\sum_{i_1 + i_2 + i_3 + i_4 + i_5\le I} \abs{\rd_R^{i_1} \big(\frac{1}{R^2A^2}\big)} 
\norm{\rd_{Q_T}^{i_2}Z^{i_3} f}_{L^\infty}\norm{\rd_R^{i_4}\rd_\ell^{i_5} \bar{\CC}}_{L^\infty}  \\
&\le \begin{cases}
    C\d\e &\text{ if } I \le N, \\
    C\d\e\jt &\text{ if } I = N+1. 
\end{cases}
\end{align*}
\end{proof}

The proof for \eqref{dsqtop} is not very different; here we need to use the equation to write \(\rd_s f\) in terms of the other derivatives.
\begin{proof}[Proof of \eqref{dsqtop}]
For \eqref{dsqtop} on \(\rd_s \qbf(s, R)\), we first write its integral form: 
\begin{align*}
\rd_s \qbf(s, R)
&= -\frac{2\pi \rd_s A}{R^2 A^3}\int_{-\infty}^\infty\int_0^\infty 
f(s, R, w, \ell) \bar{\CC}(s, R, w, \ell) \, d\ell dw \\
&\qquad
+ \frac{\pi}{A^2R^2}\int_{-\infty}^\infty\int_0^\infty 
(\rd_s f)(s, R, w, \ell) \bar{\CC}(s, R, w, \ell) \, d\ell dw \\
&\qquad
+ \frac{\pi}{A^2R^2}\int_{-\infty}^\infty\int_0^\infty 
f(s, R, w, \ell) \rd_s \bar{\CC}(s, R, w, \ell) \, d\ell dw \\
&\eqqcolon \Ical_1 + \Ical_2 + \Ical_3.
\end{align*}
Note that \(\Ical_1\) and \(\Ical_3\) can be controlled in exactly the same way as in the proof of \eqref{qlow} and \eqref{qtop}.
\(\rd_s A\) has bounded (actually small and decaying) \(N\)-th derivative by \eqref{BA1}, and \(\rd_s \bar{\CC}\) can be controlled by \(\rd_s A\) as well.

For \(\Ical_2\), we write this as follows using the equation: 
\[
\rd_s f = \rd_t f + (\rd_s H)\rd_H f
= -\O(H, M)\rd_{Q_T}f - \Pcal(t, T, Q_T, H, M)\rd_{Q_T}f - (\Dfr H - \rd_s H)\rd_H f,
\]
We need to show that \(\rd_R^I \Ical_2(s, R)\) can be controlled by \(\d\e\) for \(I \le N-1\). 
For the \(\rd_H f\) term, we can just use \eqref{dHintoell} to write \(\rd_H\) as a linear combination of \(\rd_\ell\), \(\rd_{Q_T}\), and \(Z\), and then integrate by parts in \(\ell\) to move the \(\rd_\ell\) derivative to \(\bar{\CC}\).
\begin{align*}
\Ical_2(s, R)
&= -\frac{\pi}{A^2R^2}\int_{-\infty}^\infty\int_0^\infty 
(\O(H, M)+\Pcal(t, T, Q_T, H, M))\rd_{Q_T} f\cdot \bar{\CC} \, d\ell dw \\
&\qquad
-\frac{\pi}{A^2R^2}\int_{-\infty}^\infty\int_0^\infty 
(\Dfr H - \rd_s H) \bar{\CC}\cdot\frac{\rd_X\O_\Sch}{\rd_\ell(\O_\Sch)}
\left(- (\rd_\ell Q_T)\rd_{Q_T} + \frac{1}{\rd_X\O_\Sch}Z\right)f \, d\ell dw \\
&\qquad
+ \frac{\pi}{A^2R^2}\int_{-\infty}^\infty\int_0^\infty 
\rd_\ell\left((\Dfr H - \rd_s H) \bar{\CC}\cdot\frac{\rd_X\O_\Sch}{\rd_\ell(\O_\Sch)}\right)
f \, d\ell dw
\end{align*}
Then, we use the same trick as above; since we are applying at most \(N-1\) derivatives, the coefficients are differentiated at most \(N\) times, which are all bounded by Proposition \ref{prop:regDAAV}, \eqref{Pcalest}, \eqref{DfrHest}, and \eqref{BA1}.
Note also that 
\[
\Dfr H - \rd_s H = -(\rd_s A)w\big(\frac{\a w}{AE}-\b\big).
\]
Therefore, we can control \(\rd_R^I \Ical_2(s, R)\) by \(\d\e\), by \eqref{flow} with \(i_2 = 0\). 
\end{proof}

\subsection{Decomposition of the density variables}
Now, we have to prove decay for the density variables, specifically \eqref{dsqlow} and \eqref{qqbarlow}.
There are two sources of decay: the first one is the decay of the metric coefficients, and the second one is the phase mixing mechanism.
As we see in Section \ref{sec:LinEst}, the linear Vlasov equation does not have a decaying solution, but the phase mixing mechanism can still give us decay, and this is what we will exploit for the linear part.
For the nonlinear part, it is a product of a decaying derivative of metric coefficients, and a derivative of \(f\) with controlled growth;
hence we need to use the first source of decay to control the nonlinear part.
In this subsection, we will decompose the density variables into a linear part and a nonlinear part, and then we will estimate them separately in the next subsections.

We first rewrite the Vlasov equation as follows:
\[
\rd_t f + \O(H, M)\rd_{Q_T} f = (\Rcal_1+\Rcal_2)(t, T, Q_T, H, M),
\]
where
\begin{align*}
\Rcal_1(t, T, Q_T, H, M) &= -\Pcal(t, T, Q_T, H, M)\rd_{Q_T} f, \\
\Rcal_2(t, T, Q_T, H, M) &= -\Dfr H \rd_H f.
\end{align*}
Note that the coefficients in the RHS are all small and decaying; this will be discussed in more detail later.
Then, by Duhamel's principle, we can write the solution \(f\) to this equation as follows:
\begin{align*}
f(t, Q_T, H, M) &= f_\init(Q_T-t\O(H, M), H, M) \\
&\qquad
+ \sum_{b = 1, 2}\int_0^t \Rcal_b (\t, T, Q_T-(t-\t)\O(H, M), H, M)d\t.
\end{align*}
Therefore, we can define \(\qbf_L\) and \(\qbf_N\) with \(\qbf = \qbf_L + \qbf_N\) in the following way: 
\begin{align}
\qbf_L(s, R) &= \frac{\pi}{R^2A^2}\int_{-\infty}^\infty\int_0^\infty f_\init(Q_T - s\O(H, M), H, M) \bar{\CC} \,d\ell dw, \label{qLdef}\\
\qbf_N(s, R) &= \frac{\pi}{R^2A^2}\sum_{b=1, 2}\int_{-\infty}^\infty\int_0^\infty\int_0^s \Rcal_b(\t, T, Q_T-(s-\t)\O(H, M), H, M) \bar{\CC} \, d\t d\ell dw. \label{qNdef}
\end{align}
Note that, the variables in these integrands depend completely on \(t = s, R, w, \ell\), meaning that it has nothing to do with \(\t\). 
The quantity \(H\) there refers to the metric coefficient at time \(t = s\), not \(\t\).

\subsection{Estimates on the linear part}\label{subsec:density-linear}
In this section we prove the decay for the linear part \(\qbf_L\), which is the contribution from the initial data.
The strategy of the proof is essentially the same as what we did in Section \ref{sec:LinEst} after noting the fact that \(\O\) defined in terms of the metric coefficients at time \(T\) is \(C^{N+2}\).
The main result of this subsection is the following. 
\begin{prop}[Control on \(\qbf_L\)]\label{prop:qL}
    For \(\qbf_L(s, R)\) defined in \eqref{qLdef} with \(\bar{\CC}\in \{E, -1/E, w, w^2/E\}\), we have the following estimates:
    \begin{align}
        \abs{\rd_R^I\rd_s \qbf_L(s, R)} &\lesssim \d\e\jap{s}^{-N+I}, &&\text{ for } I \le N-2, \label{dsqLlow}\\
        \abs{\rd_R^I(\qbf_L(s, R)-\qbf_L(T, R))} &\lesssim \d\e\jap{s}^{-N+I}, &&\text{ for } I \le N-1. \label{qLqbarlow}
    \end{align}
\end{prop}

Because we have to use the non-stationary phase argument, we need to decompose the frequency components of \(f_\init\).
We denote each component by \(\fhat_{\init, k}(H, M)\), i.e. 
\[
f_\init(Q_T, H, M) = \sum_{k\in \ZZ} e^{ikQ_T}\fhat_{\init, k}(H, M).
\]

Then, the quantities that we want to analyze can be written as follows:
\begin{lem}\label{lem:qLform}
    For \(\qbf_L(s, R)\) defined in \eqref{qLdef}, we have the following identity:
    \begin{align*}
    \rd_s \qbf_L(s, R)
    &= \sum_{k\in \ZZ} -\frac{2\pi\rd_s A}{R^2A^3}\int_{-\infty}^\infty\int_0^\infty e^{-ik s\O(H, M)}e^{ikQ_T}\fhat_{\init, k}(H, M) \bar{\CC} \,d\ell dw \\
    &\quad - \sum_{k\in \ZZ} \frac{\pi}{R^2A^2}\int_{-\infty}^\infty\int_0^\infty ik\O(H, M)e^{-ik s\O(H, M)}e^{ikQ_T}\fhat_{\init, k}(H, M) \bar{\CC} \,d\ell dw \\
    &\quad + \sum_{k\in \ZZ} \frac{\pi}{R^2A^2}\int_{-\infty}^\infty\int_0^\infty e^{-ik s\O(H, M)}e^{ikQ_T}\fhat_{\init, k}(H, M) \rd_s \bar{\CC} \,d\ell dw \\
    &\quad + \sum_{k\in \ZZ} \frac{\pi}{R^2A^2}\int_{-\infty}^\infty\int_0^\infty e^{-ik s\O(H, M)}e^{ikQ_T} \\
    &\qquad \hspace{20mm}
    \cdot(-iks\O_X(H, M)\fhat_{\init, k}(H, M) + \wh{\rd_H f}_{\init, k}(H, M))\rd_s H \bar{\CC} \,d\ell dw, \\
    \intertext{}
    \qbf_L(s, R) - \qbf_L(T, R) &= \sum_{k\in \ZZ\setminus\{0\}} \frac{\pi}{R^2A^2}\int_{-\infty}^\infty\int_0^\infty e^{-ik s\O(H, M)}e^{ikQ_T}\fhat_{\init, k}(H, M) \bar{\CC}(s) \,d\ell dw \\
    &\qquad
    - \sum_{k\in \ZZ\setminus\{0\}} \frac{\pi}{R^2\Abar^2}\int_{-\infty}^\infty\int_0^\infty e^{-ik T\O(H_T, M)}e^{ikQ_T}\fhat_{\init, k}(H_T, M) \bar{\CC}(T) \,d\ell dw \\
    &\qquad
    + \frac{\pi}{R^2\Abar^2}\int_{-\infty}^\infty\int_0^\infty (\fhat_{\init, 0}(H, M)-\fhat_{\init, 0}(H_T, M)) \bar{\CC}(T) \,d\ell dw \\
    &\qquad
    + \frac{\pi}{R^2}\int_{-\infty}^{\infty}\int_0^{\infty} \fhat_{\init, 0}(H, M)
    \left(\frac{\CC(s)}{A^2}-\frac{\CC(T)}{\Abar^2}\right) \,d\ell dw.  
    \end{align*}
\end{lem}
\begin{proof}
We first have 
\[
f_\init(Q_T-s\O(H, M), H, M)
= \sum_{k\in \ZZ} e^{-ik s\O(H, M)}e^{ikQ_T}\fhat_{\init, k}(H, M).
\]
Then, we can write \(\qbf_L\) as follows:
\[
\qbf_L(s, R) = \sum_{k\in \ZZ} \frac{\pi}{R^2A^2}\int_{-\infty}^\infty\int_0^\infty e^{-ik s\O(H, M)}e^{ikQ_T}\fhat_{\init, k}(H, M) \bar{\CC} \,d\ell dw.
\]
First we have to take \(s\) derivative; for the integrand we fix \(R, w, \ell\) and take a derivative in \(s\).
The integrand depends on \(s\) not only explicitly but also implicitly through \(H\), and \(\bar{\CC}\) contains metric coefficients which depends on \(s\). 
Collecting terms with \(\rd_H\) at last, we obtain the expression above. 

For the second identity, the above identity just separates the \(k = 0\) mode and the \(k\neq 0\) modes.
The RHS is just a rearrangement of the definition of \(\qbf_L\).
\end{proof}

\begin{rmk}
For the above expression of \(\qbf_L(s, R)-\qbf_L(T, R)\), we keep the \(k\neq 0\) term as we can show that it is small and decay by itself, and for the \(k = 0\) term we do not see the difference of exponential term, so we can use the fact that \(H-H_T\) is small to control it. 
This will be discussed when we prove Proposition \ref{prop:qL}.    
\end{rmk}

We prove the following lemma which contains the effect of phase mixing mechanism. 
This controls the terms in the form repeated above. 
\begin{lem}\label{lem:qL:stationaryphase}
    Let \(\bar{\CC}(t, R, w, \ell)\) be a function with bounded derivative up to the \(N\)-th order, and \(g\) is a \(C^N\) function with \(N\)-th derivative bounded.
    Then, for \(k\neq 0\) and \(\a = 0, 1\), we have the decay of the following integral for \(I \le N-1-\a\): 
    \begin{multline}\label{stationaryphase}
    \abs{\rd_R^I \int_{-\infty}^{\infty}\int_0^\infty
    k^\a e^{-iks\O(H, M)}e^{ikQ_T}\ghat_k(H, M) \bar{\CC} \, d\ell dw} \\
    \lesssim \jap{s}^{-N+I} \sum_{i_1+i_2+i_3\le N}\int_{-\infty}^\infty\int_0^\infty \frac{1}{\abs{k}}\abs{\wh{(\rd_{Q_T}^{i_1}\rd_H^{i_2}\rd_M^{i_3}g)}_k(H, M)} \, d\ell dw.
    \end{multline}
    For \(k = 0\) we can still prove the boundedness of this integral for \(I \le N\): 
    \begin{equation}\label{stationaryphase0}
    \abs{\rd_R^I \int_{-\infty}^{\infty}\int_0^\infty
    \ghat_0(H, M) \bar{\CC} \, d\ell dw}
    \lesssim \sum_{i_1+i_2\le I} \int_{-\infty}^\infty\int_0^\infty 
    \abs{\wh{(\rd_H^{i_1}\rd_M^{i_2}g)}_0(H, M)} \, d\ell dw.
    \end{equation}
\end{lem}
\begin{proof} For \(k\neq 0\), we use the following identity: 
\[
e^{-iks\O(H, M)} = \left(\frac{1}{-iks\rd_\ell(\O(H, M))}\right)\rd_\ell e^{-iks\O(H, M)}.
\]
With this, we can control the integral as follows. Here, \(\CC\) denotes a generic bounded function, which may change from line to line.
\begin{align*}
&\abs{\rd_R^I \int_{-\infty}^{\infty}\int_0^\infty
k^\a e^{-iks\O(H, M)}e^{ikQ_T}\ghat_k(H, M) \bar{\CC} \, d\ell dw} \\
&\lesssim \sum_{i_1 + \cdots + i_4 \le I} \abs{\int_{-\infty}^{\infty}\int_0^\infty
k^{\a+i_1+i_2}s^{i_2}e^{-iks\O(H, M)}e^{ikQ_T}\wh{(\rd_H^{i_3}\rd_M^{i_4}g)}_k(H, M) \CC \, d\ell dw} \\
&\lesssim \sum_{\substack{i_1 + \cdots + i_4 \le I \\ j_1 + j_2 + j_3 \le N-I+i_2}}
\frac{\abs{k}^{\a+i_1}}{\abs{ks}^{N-I}}
\int_{-\infty}^{\infty}\int_0^\infty \abs{k}^{j_1}\abs{\wh{(\rd_H^{i_3+j_2}\rd_M^{i_4+j_3}g)}_k(H, M)} \, d\ell dw \\
&\lesssim \jap{s}^{-N+I} \sum_{\substack{i_1 + \cdots + i_4 \le I \\ j_1 + j_2 + j_3 \le N-I+i_2}}
\int_{-\infty}^\infty\int_0^\infty \abs{k}^{\a+i_1+j_1-N+I}
\abs{\wh{(\rd_H^{i_3+j_2}\rd_M^{i_4+j_3}g)}_k(H, M)} \, d\ell dw.
\end{align*}
Then, counting the index, we have 
\[
(1+\a+i_1+j_1-N+I) + (i_3 + j_2) + (i_4 + j_3)
\le 1 + \a + I \le N,
\]
which proves that the above is bounded by the RHS of \eqref{stationaryphase}.

For the second inequality, we can just control the integral by the triangle inequality, and then use the fact that \(\bar{\CC}\) has bounded derivative up to the \(N\)-th order to control the \(R\) derivative.
\end{proof}

\begin{proof}[Proof of Proposition \ref{prop:qL}]
For \(\rd_s \qbf_L (s, R)\), there are four terms to control.
For the first term, \(k = 0\) term is bounded and the decay follows from \(\rd_s A\): for \(I \le N-2\) we have
\begin{align*}
&\abs{\rd_R^I\left(-\frac{2\pi\rd_s A}{R^2A^3}
\int_{-\infty}^\infty\int_0^\infty \fhat_{\init, 0}(H, M) \bar{\CC} \,d\ell dw\right)} \\
&\lesssim \sum_{i_1 + i_2 \le I} \abs{\rd_R^{i_1} \rd_s A}
\abs{\rd_R^{i_2} \int_{-\infty}^\infty\int_0^\infty \fhat_{\init, 0}(H, M) \bar{\CC} \,d\ell dw} \\
&\lesssim \sum_{i_1 + i_2 + i_3 \le I} \d^{3/4}\e\jap{s}^{-N + \max{i_1-2, 0}}
\int_{-\infty}^\infty\int_0^\infty \abs{\wh{(\rd_H^{i_2}\rd_M^{i_3}f_\init)}_0(H, M)} \, d\ell dw 
\end{align*}
which is of course much smaller than we need. 
Note that we used \eqref{BA1} and \eqref{stationaryphase0} at the last inequality, and the last integral can be controlled by \eqref{nonlinthm-initf}.
For \(k\neq 0\) terms, we need to sum them over \(\ZZ\); for this we proceed as follows: 
\begin{align*}
&\abs{\rd_R^I\sum_{k\in \ZZ\setminus\{0\}} \frac{2\pi\rd_s A}{R^2A^3}\int_{-\infty}^\infty\int_0^\infty e^{-ik s\O(H, M)}e^{ikQ_T}\fhat_{\init, k}(H, M) \bar{\CC} \,d\ell dw} \\
&\lesssim \sum_{i_1+i_2\le I} \abs{\rd_R^{i_1}\rd_s A}
\sum_{k\in \ZZ\setminus\{0\}}\abs{
    \rd_R^{i_2}\int_{-\infty}^\infty\int_0^\infty e^{-ik s\O(H, M)}e^{ikQ_T}\fhat_{\init, k}(H, M) \bar{\CC} \,d\ell dw} \\
&\lesssim \sum_{i_1+i_2\le I} \d^{3/4}\e\jap{s}^{-N + \max{i_1-2, 0}}
\sum_{k\in \ZZ\setminus\{0\}}\sum_{j_1 + j_2 +j_3\le N}
\int_{-\infty}^\infty\int_0^\infty \frac{1}{\abs{k}}\abs{\wh{(\rd_{Q_T}^{j_1}\rd_H^{j_2}\rd_M^{j_3}f_\init)}_k(H, M)} \, d\ell dw \\ 
&\lesssim \d^{3/4}\e\jap{s}^{-N+I}\sum_{j_1 + j_2 +j_3\le N}\int_{-\infty}^\infty\int_0^\infty
\bigg(\sum_{k\in \ZZ\setminus\{0\}}\frac{1}{\abs{k}^2}\bigg)^{1/2}
\norm{\wh{(\rd_{Q_T}^{j_1}\rd_H^{j_2}\rd_M^{j_3}f_\init)}_k(H, M)}_{\ell^2_k} \, d\ell dw \\
&\lesssim \d^{3/4}\e\jap{s}^{-N+I}\sum_{j_1 + j_2 +j_3\le N}
\sup_{(Q_T, H, M)}\abs{\rd_{Q_T}^{j_1}\rd_H^{j_2}\rd_M^{j_3}f_\init(Q_T, H, M)} \\
&\lesssim \d\e\jap{s}^{-N+I}.
\end{align*}

The third term can be controlled in the same way, after noticing that 
\[
\rd_s \bar{\CC} \in \left\{
-\frac{\ell\rd_s A}{R^2A^3E}, \, \frac{\ell\rd_s A}{R^2A^3E^3}, \, 0, \, \frac{w^2\ell\rd_s A}{R^2A^3E^3}
\right\}.
\]
For the fourth term, because \(\rd_s H\) can be written as a linear combination of \(\rd_s (\a, A, \b)\) we can treat this term similarly; but here we need to control the terms with one more derivative \(\wh{(\rd_H f)}\) and \(k\fhat\) together with multiplied \(s\).
Since we actually have \(f_\init\in C^{N+1}\) with smallness \(\e\) from \(\rd_s H\), we can control this term in the same way as above.

For the second term, we need to use the full power of \eqref{stationaryphase} with \(\a = 1\). For \(I \le N-2\),
\begin{align*}
&\abs{\rd_R^I \sum_{k\in \ZZ} \frac{\pi}{R^2A^2}\int_{-\infty}^\infty\int_0^\infty ik\O(H, M)e^{-ik s\O(H, M)}e^{ikQ_T}\fhat_{\init, k}(H, M) \bar{\CC} \,d\ell dw} \\
&\lesssim \sum_{k\in \ZZ\setminus\{0\}} \abs{\rd_R^I \int_{-\infty}^\infty\int_0^\infty ke^{-ik s\O(H, M)}e^{ikQ_T}\fhat_{\init, k}(H, M) \CC \,d\ell dw} \\
&\lesssim \jap{s}^{-N+I}\sum_{k\in \ZZ\setminus\{0\}}\sum_{i_1 + i_2 + i_3 \le N} \int_{-\infty}^\infty\int_0^\infty \frac{1}{\abs{k}}\abs{\wh{(\rd_{Q_T}^{i_1}\rd_H^{i_2}\rd_M^{i_3}f_\init)}_k(H, M)} \, d\ell dw \\
&\lesssim \jap{s}^{-N+I}\sum_{i_1 + i_2 + i_3 \le N} \int_{-\infty}^\infty\int_0^\infty 
\norm{\wh{(\rd_{Q_T}^{i_1}\rd_H^{i_2}\rd_M^{i_3}f_\init)}_k(H, M)}_{\ell^2_k} \, d\ell dw \\
&\lesssim \jap{s}^{-N+I}\sum_{i_1 + i_2 + i_3 \le N}\sup_{(Q_T, H, M)}\abs{\rd_{Q_T}^{i_1}\rd_H^{i_2}\rd_M^{i_3}f_\init}.
\end{align*} 
The last line can be controlled by \eqref{nonlinthm-initf}, which gives us the desired smallness and decay.

Now we turn to the second estimate on \(\qbf_L(s, R) - \qbf_L(T, R)\).
For the \(k\neq 0\) terms, we can just use the same argument as above with \eqref{stationaryphase} \(\a = 0\) to show that they are small and decaying. 
For the \(k = 0\) term, we control two terms separately. 
We have 
\begin{align*}
&\frac{\pi}{R^2\Abar^2}\int_{-\infty}^\infty\int_0^\infty (\fhat_{\init, 0}(H, M)-\fhat_{\init, 0}(H_T, M)) \bar{\CC}(T) \,d\ell dw \\
&= \int_0^1 \frac{\pi}{R^2\Abar^2}\int_{-\infty}^\infty\int_0^\infty
\wh{(\rd_H f_\init)}_0(H_T + \th(H-H_T), M)(H-H_T) \bar{\CC}(T) \,d\ell dw d\th,
\end{align*}
and 
\begin{align*}
&\abs{\rd_R^I \int_{-\infty}^\infty\int_0^\infty
\wh{(\rd_H f_\init)}_0(H_T + \th(H-H_T), M)(H-H_T) \bar{\CC}(T) \,d\ell dw} \\
&\lesssim \sum_{i_1 + i_2 + i_3 \le I} \int_{-\infty}^\infty\int_0^\infty
\abs{\wh{(\rd_H^{1+i_1}\rd_M^{i_2} f_\init)}_0(H_T + \th(H-H_T), M) \rd_R^{i_3}(H-H_T)} \,d\ell dw \\
&\lesssim \sum_{i_1 + i_2 + i_3 \le I} \sup_{Q_T, H, M}\abs{\rd_H^{1+i_1}\rd_M^{i_2} f_\init}\abs{\rd_R^{i_3}(H-H_T)},
\end{align*}
which is smaller than the desired bound by \eqref{nonlinthm-initf} and \eqref{BA2}.
For the last term, this can be controlled exactly in the same way as \(k = 0\) for the first term of \(\rd_s \qbf\), after noticing that 
\[
\rd_R^I \big(\frac{\CC(s)}{A^2}-\frac{\CC(T)}{\Abar^2}\big)
\le C\sum_{i\le I}\rd_R^i(A-\Abar)
\le C\d^{3/4}\e\jap{s}^{-N+\max{I-2, 0}}
\]
by \eqref{BA2}.
\end{proof}

\subsection{Setting up the nonlinear estimates}
We recall that the terms we are considering here are of the form 
\[
\qbf_N(s, R) = \frac{\pi}{R^2A^2}\sum_{b = 1, 2}
\int_{-\infty}^\infty\int_0^\infty\int_0^s \Rcal_b(\t, T, Q_T-(s-\t)\O(H, M), H, M) \bar{\CC} \, d\t d\ell dw,
\]
and the estimates we want to prove are the following:
\begin{prop}[Control on \(\qbf_N\)]\label{prop:qN}
    For \(\qbf_N(s, R)\) defined in \eqref{qNdef} with \(\bar{\CC}\in \{E, -1/E, w, w^2/E\}\), we have the following estimates:
    \begin{align}
        \abs{\rd_R^I\rd_s \qbf_N(s, R)} &\lesssim \d\e\jap{s}^{-N+I}, &&\text{ for } I \le N-2, \label{dsqNlow}\\
        \abs{\rd_R^I(\qbf_N(s, R)-\qbf_N(T, R))} &\lesssim \d\e\jap{s}^{-N+I}, &&\text{ for } I \le N-1. \label{qNqNbarlow}
    \end{align}
\end{prop}
The goal of the remaining part of this section is to prove the above proposition.

To analyze \(\Rcal_1, \Rcal_2\) in an integrated way, we use the following notation: we write \(\Rcal_b = p_b \rd_b f\), where 
\[
p_b = \begin{cases}
-\Pcal(t, T, Q_T, H, M), &\text{ for } b = 1, \\
-\Dfr H, &\text{ for } b = 2,
\end{cases}
\qquad
\rd_b = \begin{cases}
\rd_{Q_T}, &\text{ for } b = 1, \\
\rd_H, &\text{ for } b = 2.
\end{cases}
\] 
For \(p_b\), we have the following estimates:
\begin{align*}
\abs{\rd^I p_1 (s, T, Q_T, H, M)} &\lesssim \d^{3/4}\e\jap{s}^{-N + \max{I-1, 0}} 
&\text{ for } I \le N+1,
\\
\abs{\rd^I p_2 (s, T, Q_T, H, M)} &\lesssim \d^{3/4}\e\jap{s}^{-N + \max{I-2, 0}} 
&\text{ for } I \le N, \\
&\lesssim \d^{3/4}\e &\text{ for } I = N+1.
\end{align*} 
The estimates for \(p_1\) and \(p_2\) were already stated in \eqref{Pcalest} and \eqref{DfrHest}.

Note that, in general \(\Rcal_2\) is worse to control. 
For \(\Rcal_1\), the \(p\) part loses one \(\jap{s}\); but this is compensated by the fact that \(\rd_{Q_T}\) does not lose any \(s\) decay.
Therefore, \(\Rcal_1\) is better than \(\Rcal_2\) in the sense that the first derivative might not lose \(\jap{s}\) for \(\Rcal_1\) when it hits \(p_1\), but for \(\Rcal_2\) we lose \(\jap{s}\) coming from \(\rd_H f\).

Now, we compute \(\rd_s \qbf_N(s, R)\) and \(\qbf_N(s, R)-\qbf_N(T, R)\). 
We separate the terms as in the following propositions. 
\begin{prop}\label{prop:dsqNdecomp}
    We have the following decomposition for \(\rd_s \qbf_N(s, R)\):
    \[
    \rd_s \qbf_N (s, R) = \frac{\pi}{R^2A^2}\sum_{a = 1}^5\sum_{b = 1}^2
    \Tcal_{a, b}(s, R),
    \]
    where \(\Tcal_{a, b}\) are defined as follows:
    \begin{align*}
    \Tcal_{1, b}(s, R) &\coloneqq \int_{-\infty}^\infty\int_0^\infty \Rcal_b(s, T, Q_T, H, M) \bar{\CC} \,d\ell dw, \\
\Tcal_{2, b}(s, R) &\coloneqq -\frac{2\rd_s A}{A}\int_{-\infty}^\infty\int_0^\infty\int_0^s \Rcal_b(\t, T, Q_T-(s-\t)\O(H, M), H, M) \bar{\CC} \, d\t d\ell dw, \\
\Tcal_{3, b}(s, R) &\coloneqq \int_{-\infty}^\infty\int_0^\infty\int_0^s \Rcal_b(\t, T, Q_T-(s-\t)\O(H, M), H, M) \rd_s \bar{\CC} \, d\t d\ell dw, \\
\Tcal_{4, b}(s, R) &\coloneqq \int_{-\infty}^\infty\int_0^\infty\int_0^s 
(-(s-\t)\rd_X \O\rd_{Q_T}\Rcal_b + \rd_H \Rcal_b)(\t, T, Q_T-(s-\t)\O(H, M), H, M) \\
&\hspace{30mm} \cdot \rd_s H \bar{\CC} \,d\t d\ell dw, \\
\Tcal_{5, b}(s, R) &\coloneqq -\int_{-\infty}^\infty\int_0^\infty\int_0^s 
\O(H, M)\rd_{Q_T}\Rcal_b(\t, T, Q_T-(s-\t)\O(H, M), H, M) \bar{\CC} \, d\t d\ell dw.
    \end{align*}
\end{prop}
\begin{proof}
The first term follows when \(\rd_s\) hits the limit of the innermost integral. 
The second term follows when \(\rd_s\) hits the \(A^2\) in the denominator outside the integral.
The third term follows when \(\rd_s\) hits \(\bar{\CC}\).
For the fourth term, this follows when \(\rd_s\) hits the implicit \(s\) dependence in the integrand, which is through \(H\); this term contains \(\rd_s H\) which will be the main source of decay.
The last term follows when \(\rd_s\) hits the explicit \(s\) in the phase of the integrand \(Q_T - (s-\t)\O(H, M)\).
\end{proof}

Now we turn to \(\qbf_N(s, R) - \qbf_N(T, R)\).
Note that the \(\t\) integral in the definition of \(\qbf_N\) is from \(0\) to \(s\) for \(\qbf_N(s, R)\) and from \(0\) to \(T\) for \(\qbf_N(T, R)\); therefore, on \([0, s]\) we should control the difference \(\Dcal_{1, b}(s, R)-\Dcal_{1, b}(T, R)\) and \(\Dcal_{2, b}\), and on \([s, T]\) we need to use the decay to control the remaining terms \(\Dcal_{3, b}\).
Also, as we have seen in the linear part, the \(k = 0\) mode and \(k\neq 0\) modes behave differently; we need to separate them into \(\Dcal_{1, b}(s, R)-\Dcal_{1, b}(T, R)\) and \(\Dcal_{2, b}\) and use different arguments to control them.
This leads us to decompose in the frequency components at this point. 
\begin{prop}\label{prop:qqbardecomp}
We have the following decomposition for \(\qbf_N(s, R) - \qbf_N(T, R)\):
\[
\qbf_N(s, R) - \qbf_N(T, R) = \sum_{b = 1, 2} \big[\Dcal_{1, b}(s, R) - \Dcal_{1, b}(T, R) + \Dcal_{2, b}(s, T, R) - \Dcal_{3, b}(s, T, R)\big],
\]
where \(\Dcal_{a, b}\) are defined as follows:
\begin{equation}\begin{aligned}\label{DibFT}
\Dcal_{1, b}(s, R) &\coloneqq \sum_{k\neq 0, l}\frac{\pi}{R^2A^2}\int_{-\infty}^\infty\int_0^\infty\int_0^s e^{ikQ_T}\wh{(p_b)}_l \cdot e^{-ik(s-\t)\O}\wh{(\rd_b f)}_{k-l}\bar{\CC}(s) \, d\t d\ell dw, \\
\Dcal_{2, b}(s, T, R) &\coloneqq \sum_l \frac{\pi}{R^2A^2}\int_{-\infty}^\infty\int_0^\infty\int_0^s \wh{(p_b)}_l \wh{(\rd_b f)}_{-l}(H, M)\bar{\CC}(s) \, d\t d\ell dw \\
&\qquad
- \sum_l \frac{\pi}{R^2\Abar^2}\int_{-\infty}^\infty\int_0^\infty\int_0^s \wh{(p_b)}_l \wh{(\rd_b f)}_{-l}(H_T, M)\bar{\CC}(T) \, d\t d\ell dw, \\
\Dcal_{3, b}(s, T, R) &\coloneqq \sum_l\frac{\pi}{R^2\Abar^2}\int_{-\infty}^\infty\int_0^\infty\int_s^T \wh{(p_b)}_l \wh{(\rd_b f)}_{-l}\bar{\CC}(T) \, d\t d\ell dw.
\end{aligned}\end{equation}
\end{prop}

We end this subsection by summarizing the following two subsections which will be devoted to controlling the nonlinear part.
For \(\Tcal_{a, b}\), clearly \(\Tcal_{5, b}\) is the worst term to control.
For \(\Tcal_{1, b}\), we can prove the desired decay without decomposing it into frequency components, which will be shown in the end of this section. 

For the other terms, we need to decompose them into frequency components as follows: 
\begin{equation}\begin{aligned}\label{TibFT}
\Tcal_{2, b}(s, R) &= \sum_{k, l\in \ZZ} -\frac{2\rd_s A}{A}\int_{-\infty}^\infty\int_0^\infty\int_0^s e^{ikQ_T}e^{-ik(s-\t)\O}\wh{(p_b)}_l\wh{(\rd_b f)}_{k-l} \bar{\CC} \, d\t d\ell dw, \\
\Tcal_{3, b}(s, R) &= \sum_{k, l\in \ZZ} \int_{-\infty}^\infty\int_0^\infty\int_0^s e^{ikQ_T}e^{-ik(s-\t)\O}\wh{(p_b)}_l\wh{(\rd_b f)}_{k-l} \rd_s \bar{\CC} \, d\t d\ell dw, \\
\Tcal_{4, b}(s, R) &= \sum_{k, l\in \ZZ} \int_{-\infty}^\infty\int_0^\infty\int_0^s 
-(s-\t)\rd_X \O\cdot ike^{ikQ_T}e^{-ik(s-\t)\O}\wh{(p_b)}_l\wh{(\rd_b f)}_{k-l} \rd_s H \bar{\CC} \\
&\hspace{40mm} + e^{ikQ_T}e^{-ik(s-\t)\O}\rd_H(\wh{(p_b)}_l\wh{(\rd_b f)}_{k-l}) \rd_s H \bar{\CC} \,d\t d\ell dw, \\
\Tcal_{5, b}(s, R) &= -\sum_{k, l\in \ZZ} \int_{-\infty}^\infty\int_0^\infty\int_0^s 
ik\O e^{ikQ_T}e^{-ik(s-\t)\O}\wh{(p_b)}_l\wh{(\rd_b f)}_{k-l} \bar{\CC} \, d\t d\ell dw.
\end{aligned}\end{equation}
For \(\Tcal_{2, b}\), \(\Tcal_{3, b}\), and \(\Tcal_{4, b}\), because of the \(s\)-decaying term \(\rd_s A\), \(\rd_s \bar{\CC}\) and \(\rd_s H\), these terms are better than \(\Tcal_{5, b}\).
Thus, we mainly focus on controlling \(\Tcal_{5, b}\) in the next section, then the estimates on the other terms will follow easily.

For \(\Dcal_{a, b}\), \(\Dcal_{1, b}\) is the worst term to control, but this is strictly easier than \(\Tcal_{5, b}\) because there is one less \(k\), and the estimate we need is the same. 
Of course we want one more derivative on \(\Dcal\), but we will see that the top order term is not hard to control. 
For the zero mode terms, we can just use the decay to control them, which will be shown in the last section.

In the next subsection, we prove some general integral inequalities, and in the last subsection these will be used to control \(\Tcal_{a, b}\) and \(\Dcal_{a, b}\).

This is the promised control on \(\Tcal_{1, b}(s, R)\).
\begin{lem}\label{lem:T1b}
    For \(\Tcal_{1, b}(s, R)\) defined in Proposition \ref{prop:dsqNdecomp} with \(b = 1, 2\), we have the following estimates:
    \begin{equation}\label{T1b}
        \abs{\rd_R^I \Tcal_{1, b}(s, R)} \lesssim \d^{7/4}\e\jap{s}^{-N+I}.
    \end{equation}
\end{lem}
\begin{proof}
\(\rd_R\) might hit \(p_b\) or \(\rd_b f\). When \(b = 2\), we proceed as follows: when \(\rd_H\) hits \(f\) then we use \eqref{dHintoell} and integrate by parts in \(\ell\), so that we have
\begin{align*}
\abs{\rd_R^I \Tcal_{1, b}(s, R)} &\lesssim
\sum_{i_1 + i_2 + i_3 \le I+1}
\int_{-\infty}^\infty\int_0^\infty \abs{\rd_R^{i_1}p_b} \abs{\rd_{Q_T}^{i_2}Z^{i_3} f} \, d\ell dw \\ 
&\lesssim \sum_{i_1 + i_2 + i_3 \le I+1} \d^{3/4}\e\jap{s}^{-N + \max{i_1-2, 0}}\cdot \d\e \\
&\lesssim \d^{7/4}\e\jap{s}^{-N+I}.
\end{align*}
We display the proof for \(b = 1\) case to show that it is indeed easier than \(b = 2\) case: this argument applies for the rest of this section. 
The inequality goes as 
\begin{align*}
\abs{\rd_R^I \Tcal_{1, b}(s, R)} &\lesssim
\sum_{i_1 + i_2 + i_3 \le I+1}
\int_{-\infty}^\infty\int_0^\infty \abs{\rd_R^{i_1}p_b} \abs{\rd_{Q_T}^{1+i_2}Z^{i_3} f} \, d\ell dw \\ 
&\lesssim \sum_{i_1 + i_2 + i_3 \le I} \d^{3/4}\e\jap{s}^{-N + \max{i_1-1, 0}}\cdot \d\e \\
&\lesssim \d^{7/4}\e\jap{s}^{-N+I}.
\end{align*}
\end{proof}

\subsection{Estimates on the nonlinear part I: General integral inequalities}
For \(\rd_s \qbf_N(s, R)\), the main term that we have to control is \(\Tcal_{5, b}(s, R)\). 
Say if \(\rd_R^I\) hits the integrand, then the derivative might hit \(e^{ikQ_T}\) to produce one more \(k\), or \(\wh{(p_b)}_l\) so that we lose one \(\jap{s}\) decay, or more difficultly, \(e^{-ik(s-\t)\O}\wh{(\rd_b f)}_{k-l}\). 
As we did in the Subsection \ref{subsec:density-linear}, we can use Lemma \ref{lem:drintodell} so that when \(\rd_R^I\) hits \(e^{-ik(s-\t)\O}\wh{(\rd_b f)}_{k-l}\), we ignore \(\rd_{Q_T}\) derivatives, integrate by parts \(\rd_\ell\) derivatives, and consider \(Z\) derivatives only. 
In the sense of Lemma \ref{lem:transported}, this produces a term with \(Z\) hitting \(f\), or an extra \(k\). 
Because \(k\) can be decomposed into \(k-l\) and \(l\), it can act as \(\rd_{Q_T}\) on either \(p\) or \(f\). 
This motivates us to define the following symbol. 
\begin{equation}\label{Idef}
\Ical_{k, l}^{i_1, i_2, i_3}(\t, s, R)
\coloneqq \int_{-\infty}^\infty\int_0^\infty e^{ikQ_T}\wh{(\rd_R^{i_1}p_b)}_l \cdot \big(e^{-ik(s-\t)\O}\wh{(\rd_b \rd_{Q_T}^{i_2}Z^{i_3}f)}_{k-l}\big)\CC(s) \, d\ell dw
\end{equation}
Note that \(\CC\) is a generic function with bounded derivatives which may change from line to line.
The derivative on \(p_b\) is originally \(\rd_H\) or \(\rd_M\) on the Fourier side, but this is in turn a derivative \(\rd_R\) with some constant multiplied at the physical space, so we can just write \(\rd_R\) for simplicity.

The term \(e^{-ik(s-\t)\O}\wh{(\rd_b f)}_{k-l}(H, M)\) is the free transport of \(f\) at time \(\t\) to time \(s\) by the flow generated by \(\rd_t + ik\O(H, M)\) on the Fourier side, or \(\rd_t + \O(H, M)\rd_{Q_T}\) on the physical side.
Because this vector field commutes with \(Z\) (in the sense that the commutator is small) and \(Y_H\), the derivative of the free transport of \(f\) for these vectors has a controlled bound, which is the key to control \(\Tcal_{5, b}\).
This will be described in Lemma \ref{lem:transported}. We first define a vector field on the Fourier side as follows: 
\begin{equation}\label{Yketdef}
Y_{k, \et\t}(s, H, M) \coloneqq i(ks+\et\t)(\rd_X \O)(H, M) + \rd_H.
\end{equation}
We are in particular interested in the cases \(\et = -l\) and \(\et = -k\), which will be used to control the derivatives of the free transport of \(f\) and gain \(ks + \et\t\).

\begin{lem}\label{lem:transported}
    We have the following identities for a regular function \(g(t, Q_T, H, M)\): 
    \begin{align}
    Y_{k, -l\t}(e^{-ik(s-\t)\O}\ghat_{k-l}) &= e^{-ik(s-\t)\O}\wh{(Y_H g)}_{k-l}, \label{Y1ontrans}\\
    Y_{k, -k\t}(e^{-ik(s-\t)\O}\ghat_{k-l}) &= e^{-ik(s-\t)\O}\wh{(\rd_H g)}_{k-l}, \label{Y2ontrans}\\
    Z(e^{-ik(s-\t)\O}\ghat_{k-l}) &= e^{-ik(s-\t)\O}[\wh{(Z g)}_{k-l} - ik(s-\t)Z\O\cdot\ghat_{k-l}]. \label{Zontrans}
    \end{align}
\end{lem}
\begin{proof}
This is a straightforward computation using the definition of \(Y_{k, -l\t}\), \(Y_{k, -k\t}\) and \(Z\).
\end{proof}

\begin{rmk}
    We define \(Z\) in the way that \(Z\O\) is of the size \(\Ocal(\e)\), so in particular \((s-\t)Z\O\) is bounded in our time scale.
    Therefore, for the integrand that we consider, the \(Z\) derivative of the free transport of \(f\) essentially produces three types of terms, one with \(Z\) hitting \(f\), one with a derivative hitting \(p_b\), and one with \(\rd_{Q_T}\) hitting \(f\).
\end{rmk}

The following lemma is the main integral inequality that we will use to control \(\Tcal_{5, b}\).
Within this lemma, we omit the arguments \(s\) and \(R\) for \(\Ical^{i_1, i_2, i_3}_{k, l}\). 
\begin{lem}\label{lem:T5comp}
    For \(I \le N-2\) and \(k, l\in \ZZ\) with \(k\neq 0\), 
    there are \(a_l(\t), b_{k-l}(\t)\) with \(\sum_l a_l(\t)^2, \sum_{k-l} b_{k-l}(\t)^2 \le 1\) for each \(\t\in [0, s]\) such that the following inequality on \(\Ical_{k, l}^{i_1, i_2, i_3}(\t, s, R)\) is true.
    \begin{enumerate}
        \item On \(\t\in [s/2, s]\), for \(i_1, i_2, i_3\) with \(i_1 + i_2 + i_3 \le I+1\), we have the following estimates:
    \begin{align}
        \abs{\Ical_{k, l}^{i_1, i_2, i_3}(\t)} &\lesssim \d^{\frac{7}{4}}\e^2\jap{s}^{-N+I}\frac{a_l}{\jap{l}}\frac{b_{k-l}}{\jap{k-l}},
        &&\text{ for } I\ge 1, i_1 \le I, \label{T5complt}\\
        \abs{\Ical_{k, l}^{i_1, i_2, i_3}(\t)} &\lesssim \d^{\frac{7}{4}}\e^2\jap{s}^{-N+I+1}\frac{a_l}{\jap{l}}\frac{b_{k-l}}{\jap{k-l}},
        &&\text{ }  \label{T5compltsloss}\\
        \abs{\Ical_{k, l}^{i_1, i_2, i_3}(\t)} &\lesssim \d^{\frac{7}{4}}\e^2\jap{s}^{-N+I+1} \frac{\jap{l}}{\jap{ks-l\t}}\frac{a_l}{\jap{l}}\frac{1}{\jap{k-l}^3}.
        &&\text{ for } i_2 + i_3 \le 1. \label{T5compltY}
    \end{align}
        \item On \(\t\in [0, s/2]\), for \(i_1, i_2, i_3\) with \(i_1 + i_2 + i_3 \le I\), we have the following estimates:
    \begin{align}
        k\abs{\Ical_{k, l}^{i_1, i_2, i_3}(\t)} &\lesssim \d^{\frac{7}{4}}\e^2\jap{s}^{-N+I+1}\jap{\t}^{-1}\frac{1}{k^2}a_lb_{k-l}, &&\text{ for } 0 < I \le N-4, \label{T5compst} \\
        k\abs{\Ical_{k, l}^{i_1, i_2, i_3}(\t)} &\lesssim \d^{\frac{7}{4}}\e^2\jap{s}^{-2}\jap{\t}^{-1}\frac{1}{k}\min\left(\frac{1}{\jap{l}}, \frac{1}{\jap{k-l}}\right)a_l b_{k-l}, &&\text{ for } I = N-3, \label{T5compstN2}\\
        k\abs{\Ical_{k, l}^{i_1, i_2, i_3}(\t)} &\lesssim \d^{\frac{7}{4}}\e^2\jap{s}^{-1}\jap{\t}^{-1}\frac{a_l}{\jap{l}}\frac{b_{k-l}}{\jap{k-l}}, &&\text{ for } I = N-2, \label{T5compstN1}\\
        k\abs{\Ical_{k, l}^{i_1, i_2, i_3}(\t)} &\lesssim \d^{\frac{7}{4}}\e^2\jap{s}^{-N+2}\jap{\t}^{-1} 
        \frac{1}{k^2}\frac{1}{\jap{ks-l\t}}\left(\frac{1}{\jap{l}^2}+\frac{1}{\jap{k-l}^2}\right), &&\text{ for } I = 0. \label{T5compstI0}
    \end{align}
    \end{enumerate}
\end{lem}
\begin{proof}
We prove the statement for \(\t\in [s/2, s]\).
For these terms, note that the decay in \(\t\) can be converted into the decay in \(s\).
To get the first two estimates, we use the form \eqref{Idef} directly: 
\begin{align*}
\abs{\Ical^{i_1, i_2, i_3}_{k, l}(\t)} 
&\lesssim \int_{-\infty}^\infty\int_0^\infty \abs{\wh{(\rd_R^{i_1}p_b)}_l}\cdot \abs{\wh{(\rd_b \rd_{Q_T}^{i_2}Z^{i_3}f)}_{k-l}} \, d\ell dw \\
&\lesssim \frac{1}{\jap{l}\jap{k-l}}\left(\int_{-\infty}^\infty\int_0^\infty \abs{\wh{(\rd_R^{1+i_1}p_b)}_l}^2 \, d\ell dw\right)^{1/2}
\left(\int_{-\infty}^\infty\int_0^\infty \abs{\wh{(\rd_b \rd_{Q_T}^{1+i_2}Z^{i_3}f)}_{k-l}}^2 \, d\ell dw\right)^{1/2} \\
&\lesssim \frac{1}{\jap{l}\jap{k-l}}a_l\norm{\rd_R^{1+i_1} p_b}_{L^\infty} b_{k-l}\norm{\rd_b \rd_{Q_T}^{1+i_2}Z^{i_3}f}_{L^\infty}.
\end{align*}
We first analyze the case when the number of derivatives on \(f\) is at most \(N\).
If \(I \ge 1\) and \(i_1 \le I\), then we have 
\[
\norm{\rd_R^{1+i_1} p_b}_{L^\infty} \norm{\rd_b \rd_{Q_T}^{1+i_2}Z^{i_3}f}_{L^\infty}
\lesssim \d^{7/4}\e^2\jap{s}^{-N+I},
\]
which proves \eqref{T5complt}, and else we still have 
\[
\norm{\rd_R^{1+i_1} p_b}_{L^\infty} \norm{\rd_b \rd_{Q_T}^{1+i_2}Z^{i_3}f}_{L^\infty}
\lesssim \d^{7/4}\e^2\jap{s}^{-N+\max{i_1-1, 0} + 1},
\]
which gives us \eqref{T5compltsloss}.
If the number of derivatives on \(f\) is \(N+1\), the only possible case is that \(I = N-2\), \(i_1 = 0\), and \(i_2+i_3 = N-1\). 
There we have 
\[
\norm{\rd_R^{1+i_1} p_b}_{L^\infty} \norm{\rd_b \rd_{Q_T}^{1+i_2}Z^{i_3}f}_{L^\infty}
\lesssim \d^{7/4}\e^2\jap{s}^{-N + 2},
\]
which is far better than we need in \eqref{T5complt} and \eqref{T5compltsloss}.

Now we prove \eqref{T5compltY} by using one \(Y_{k, -l\t}\), assuming \(i_2 + i_3\le 1\). We have 
\[
1 = \frac{Y_{k, -l\t} - \rd_H}{i(ks-l\t)\rd_X \O},
\]
so we have 
\[
\Ical^{i_1, i_2, i_3}_{k, l}(\t) = \frac{1}{ks-l\t}\int_{-\infty}^\infty\int_0^\infty e^{ikQ_T}\wh{(\rd_R^{i_1}p_b)}_l \cdot (Y_{k, -l\t}-\rd_H)\big(e^{-ik(s-\t)\O}\wh{(\rd_b \rd_{Q_T}^{i_2}Z^{i_3}f)}_{k-l}\big)\CC(s) \, d\ell dw.
\]
We know that \(Y_{k, -l\t}\) is converted to \(Y_H\) hitting \(f\) by \eqref{Y1ontrans}, and \(\rd_H\) can be treated as before, meaning that we use \eqref{dHintoell} to convert it into \(\rd_\ell\) and \(Z\). 
From \eqref{Zontrans}, we know that \(Z\) hitting the free transport of \(f\) produces a term with \(Z\) hitting \(f\), and a term with an extra \(k\).
Since \(k = (k-l) + l\), it can be understood as \(\rd_{Q_T}\) hitting either \(p_b\) or \(f\).
As an upshot, we obtain that 
\[
\abs{\Ical^{i_1, i_2, i_3}_{k, l}(\t)} \lesssim \sum_{\a_1 + \a_2 + \a_3 + \a_4\le 1} \frac{1}{\jap{ks-l\t}}
\int_{-\infty}^\infty\int_0^\infty \abs{\wh{(\rd_R^{i_1+\a_1}p_b)}_l}\cdot \abs{\wh{(\rd_b Y_H^{\a_2}\rd_{Q_T}^{i_2+\a_3}Z^{i_3+\a_4}f)}_{k-l}} \, d\ell dw.
\]

If \(\a_1 = 0\), then this is easy to control: we gain \(\jap{l}\) by taking one more derivative (which is possible because \(i_1 \le I+1\le N-1\)), and gain \(\jap{k-l}^3\) by taking three more derivatives on \(f\).
This gives us
\begin{align*}
(\a_1 = 0 \text{ terms}) &\lesssim \frac{1}{\jap{ks-l\t}} \frac{\norm{\rd_R^{i_1 + 1}p_b}_{L^\infty}}{\jap{l}}\frac{\norm{\rd_b Y_H^{\a_2}\rd_{Q_T}^{3+i_2+\a_3}Z^{i_3+\a_4}f}_{L^\infty}}{\jap{k-l}^2} \\
&\lesssim \d^{7/4}\e^2\jap{s}^{-N+I+1} \frac{\jap{l}}{\jap{ks-l\t}} \frac{1}{\jap{l}^2}\frac{1}{\jap{k-l}^3},
\end{align*}
which is fine. Note that at this point we want to use the estimate on \(f\) which is not top order, so the condition \(N\ge 6\) should be used.

Now we consider the case when \(\a_1 = 1\), \(\a_2 = \a_3 = \a_4 = 0\). 
We do not gain more \(l\) from \(\wh{p}_l\), but to get the summability in \(l\) we need to use the same trick as before, which gives
\begin{align*}
(\a_1 = 1 \text{ term}) &\lesssim \frac{1}{\jap{ks-l\t}}\frac{1}{\jap{k-l}^3}\left(\int_{-\infty}^\infty\int_0^\infty \abs{\wh{(\rd_R^{1+i_1}p_b)}_l}^2 \, d\ell dw\right)^{1/2} \\
&\hspace{40mm}\cdot 
\left(\int_{-\infty}^\infty\int_0^\infty \abs{\wh{(\rd_b \rd_{Q_T}^{3+i_2}Z^{i_3}f)}_{k-l}}^2 \, d\ell dw\right)^{1/2} \\
&\lesssim \frac{\jap{l}}{\jap{ks-l\t}} \frac{1}{\jap{l}}\frac{1}{\jap{k-l}^3} \cdot a_l \norm{\rd_R^{i_1 + 1}p_b}_{L^\infty}\norm{\rd_b \rd_{Q_T}^{3+i_2}Z^{i_3}f}_{L^\infty} \\
&\lesssim \d^{7/4}\e^2\jap{s}^{-N+I + 1} \frac{\jap{l}}{\jap{ks-l\t}} \frac{a_l}{\jap{l}}\frac{1}{\jap{k-l}^3}.
\end{align*}

Now we prove the estimates for \(\t\in [0, s/2]\).
In this interval, the decay in \(\t\) cannot be converted into the decay in \(s\), so we need to use the vector field \(Y_{k, -k\t}\) to gain \(s\). 
Because this is the only source of decay, we need to use \(N_1 \coloneqq N-I-1\) of \(Y_{k, -k\t}\): 
\[
1 = \left(\frac{Y_{k, -k\t}-\rd_H}{ik(s-\t)\rd_X\O}\right)^{N_1},
\]
for each use of \(\t\) we gain \(k(s-\t)\) and lose one derivative, hence one \(s\). 
We have, 
\begin{multline*}
k\Ical^{i_1, i_2, i_3}_{k, l}(\t) \\
= \frac{k}{\jap{k(s-\t)}^{N_1}}\int_{-\infty}^\infty\int_0^\infty e^{ikQ_T}\wh{(\rd_R^{i_1}p_b)}_l \cdot (Y_{k, -k\t}-\rd_H)^{N_1}\big(e^{-ik(s-\t)\O}\wh{(\rd_b \rd_{Q_T}^{i_2}Z^{i_3}f)}_{k-l}\big)\CC(s) \, d\ell dw.
\end{multline*}
For \(Y_{k, -k\t}-\rd_H\), we deal with the same way as before; for \(Y_{k, -k\t}\) we use \eqref{Y2ontrans} to convert it into \(\rd_H\) hitting \(f\), and for \(\rd_H\) we use \eqref{dHintoell} to convert it into \(\rd_\ell\) and \(Z\), which gives one more derivative on \(p\) or \(\rd_{Q_T}\) or \(Z\) on \(f\). 
This proves that 
\begin{align*}
\abs{k\Ical^{i_1, i_2, i_3}_{k, l}(\t)}
&\lesssim \frac{\jap{s}^{-N+I+1}}{k^{N_1-1}}\sum_{\a_1+\a_2+\a_3+\a_4\le N_1}
\int_{-\infty}^\infty\int_0^\infty \abs{\wh{(\rd_R^{i_1+\a_1}p_b)}_l}\abs{{\wh{(\rd_b \rd_{Q_T}^{i_2+\a_2}\rd_H^{\a_3}Z^{i_3+\a_4}f)}_{k-l}}} \, d\ell dw.
\end{align*}
If \(0 < I \le N-4\), \(N_1 -1\ge 2\), so we can use \(1/k^2\) for the summability, meaning that we do not have to lose more derivative for the sake of summability, hence in this case there is no chance that there are \(N+1\) derivatives on \(f\).
Therefore, the integral satisfies 
\begin{multline*}
\int_{-\infty}^\infty\int_0^\infty \abs{\wh{(\rd_R^{i_1+\a_1}p_b)}_l}\abs{{\wh{(\rd_b \rd_{Q_T}^{i_2+\a_2}\rd_H^{\a_3}Z^{i_3+\a_4}f)}_{k-l}}} \, d\ell dw \\
\lesssim
\d^{\frac{7}{4}}\e^2\jap{\t}^{-N +\max{i_1+\a_1-2, 0} + \a_3 + 1}a_lb_{k-l}.
\end{multline*}
We have \(\a_3 + 1\le N-1\), so the \(\t\)-decay is sufficient. 
This establishes \eqref{T5compst}.

For \(I = N-3\), i.e. \(N_1 = 2\), we have only one \(k\) in the denominator, hence we need to gain one more \(\jap{l}\) or \(\jap{k-l}\) for the summability by losing a derivative.
If we choose to gain \(\jap{l}\), then we have
\begin{align*}
\int_{-\infty}^\infty\int_0^\infty (\cdots) d\ell dw
&\le \frac{1}{\jap{l}}\int_{-\infty}^\infty\int_0^\infty \abs{\wh{(\rd_R^{i_1+\a_1+1}p_b)}_l}\abs{{\wh{(\rd_b \rd_{Q_T}^{i_2+\a_2}\rd_H^{\a_3}Z^{i_3+\a_4}f)}_{k-l}}} \, d\ell dw \\
&\le \d^{7/4}\e^2\jap{\t}^{-N +\max{i_1+\a_1-1, 0} + \a_3 + 1}\frac{1}{\jap{l}}a_lb_{k-l},
\end{align*}
which still has a sufficient \(\t\)-decay. 
If we choose to gain \(\jap{k-l}\), then we have
\begin{align*}
\int_{-\infty}^\infty\int_0^\infty (\cdots) d\ell dw
&\le \frac{1}{\jap{k-l}}\int_{-\infty}^\infty\int_0^\infty \abs{\wh{(\rd_R^{i_1+\a_1}p_b)}}_l\abs{{\wh{(\rd_b \rd_{Q_T}^{i_2+\a_2+1}\rd_H^{\a_3}Z^{i_3+\a_4}f)}_{k-l}}} \, d\ell dw \\
&\le \d^{7/4}\e^2\jap{\t}^{-N +\max{i_1+\a_1-2, 0} + \a_3 + 2}\frac{1}{\jap{k-l}}a_lb_{k-l}.
\end{align*}
Note that in this case there might be \(N+1\) derivatives on \(f\); still, because \(\a_3 \le N_1 = 2\), we have sufficient decay in \(\t\).
This establishes \eqref{T5compstN2}.

For \(I = N-2\), i.e. \(N_1 = 1\), we have no \(k\) in the denominator, so we need to gain one more \(\jap{l}\) and \(\jap{k-l}\) each for the summability by losing two derivatives.
\begin{align*}
&\int_{-\infty}^\infty\int_0^\infty (\cdots) d\ell dw \\
&\le \frac{1}{\jap{l}\jap{k-l}}\int_{-\infty}^\infty\int_0^\infty \abs{\wh{(\rd_R^{i_1+\a_1+1}p_b)}_l}\abs{{\wh{(\rd_b \rd_{Q_T}^{i_2+\a_2+1}\rd_H^{\a_3}Z^{i_3+\a_4}f)}_{k-l}}} \, d\ell dw \\
&\le \begin{cases}
\d^{7/4}\e^2\jap{\t}^{-N +\max{i_1+\a_1-1, 0} + \a_3 + 1}\frac{1}{\jap{l}\jap{k-l}}a_lb_{k-l}, &\text{ if } 
i_2 + i_3 + \a_2 + \a_3 + \a_4 \le N-2, \\
\d^{7/4}\e^2\jap{\t}^{-N +\max{i_1+\a_1-1, 0} + \a_3 + 2}\frac{1}{\jap{l}\jap{k-l}}a_lb_{k-l}, &\text{ if } 
i_2 + i_3 + \a_2 + \a_3 + \a_4 = N-1.
\end{cases} 
\end{align*}
In the first case we already have a sufficient decay in \(\t\), and in the second case we have \(i_1 + \a_1 = 0\) and \(\a_3 + 2 \le 3\), so there is still a sufficient decay in \(\t\).
This establishes \eqref{T5compstN1}.

Finally, let \(I = 0\). In this case from the usual way we will see the \(N\) \(\rd_H\) derivatives on \(f\), which will not give an integrability in \(\t\).
Alternatively, we use one \(Y_{k, -l\t}\) and \(N-2\) \(Y_{k, -k\t}\), which gives 
\begin{align*}
&k\Ical^{0, 0, 0}_{k, l}(\t) \\
&= \frac{k}{\jap{ks-l\t}\jap{k(s-\t)}^{N-2}} \\
&\qquad \times
\int_{-\infty}^\infty\int_0^\infty e^{ikQ_T}\wh{(p_b)}_l \cdot (Y_{k, -l\t}-\rd_H)(Y_{k, -k\t}-\rd_H)^{N-2}\big(e^{-ik(s-\t)\O}\wh{(\rd_b f)}_{k-l}\big)\CC(s) \, d\ell dw.
\end{align*}
We treat the integral in the same way; \(Y_{k, -l\t}\) and \(Y_{k, -k\t}\) are converted following \eqref{Y1ontrans} and \eqref{Y2ontrans}, and \(\rd_H\) is converted following \eqref{dHintoell}.
This gives us
\begin{multline*}
\abs{k\Ical^{0, 0, 0}_{k, l}(\t)} \\
\lesssim \frac{1}{k^2\jap{ks-l\t}}\jap{s}^{-N+2}
\sum_{\substack{\a_1+\cdots+\a_5\le N-1\\\a_2 \le 1, \a_4\le N-2}}
\int_{-\infty}^\infty\int_0^\infty \abs{\wh{(\rd_R^{\a_1}p_b)}_l}\abs{{\wh{(\rd_b Y_H^{\a_2}\rd_{Q_T}^{\a_3}\rd_H^{\a_4}Z^{\a_5}f)}_{k-l}}} \, d\ell dw.
\end{multline*}
If every derivative falls on \(f\), then when we get two \(\jap{l}\) from \(p\) we have 
\begin{align*}
\int_{-\infty}^\infty\int_0^\infty (\cdots) d\ell dw
&\le \frac{1}{\jap{l}^2}\int_{-\infty}^\infty\int_0^\infty \abs{\wh{(\rd_R^2p_b)}_l}\abs{{\wh{(\rd_b Y_H^{\a_2}\rd_{Q_T}^{\a_3}\rd_H^{\a_4}Z^{\a_5}f)}_{k-l}}} \, d\ell dw \\
&\le \d^{7/4}\e^2\jap{\t}^{-N + \a_4 + 1}\frac{1}{\jap{l}^2},
\end{align*}
which is satisfactory because \(\a_4 \le N-2\). \\
If there are \(N-1\) derivatives on \(f\), then we have \(\a_1\le 1\), so we can gain one free \(\jap{l}\) from \(p\) and one \(\jap{k-l}\) from \(f\), which gives us
\begin{align*}
\int_{-\infty}^\infty\int_0^\infty (\cdots) d\ell dw
&\le \frac{1}{\jap{l}\jap{k-l}}\int_{-\infty}^\infty\int_0^\infty \abs{\wh{(\rd_R^{1+\a_1}p_b)}_l}\abs{{\wh{(\rd_b Y_H^{\a_2}\rd_{Q_T}^{1+\a_3}\rd_H^{\a_4}Z^{\a_5}f)}_{k-l}}} \, d\ell dw \\
&\le \d^{7/4}\e^2\jap{\t}^{-N + \a_4 + 1}\frac{1}{\jap{l}\jap{k-l}},
\end{align*}
and this is also satisfactory. \\
Finally, when there are less than \(N-1\) derivatives on \(f\), we can gain two \(\jap{k-l}\) from \(f\) for the summability.
\begin{align*}
\int_{-\infty}^\infty\int_0^\infty (\cdots) d\ell dw
&\le \frac{1}{\jap{k-l}^2}\int_{-\infty}^\infty\int_0^\infty \abs{\wh{(\rd_R^{\a_1}p_b)}_l}\abs{{\wh{(\rd_b Y_H^{\a_2}\rd_{Q_T}^{2+\a_3}\rd_H^{\a_4}Z^{\a_5}f)}_{k-l}}} \, d\ell dw \\
&\le \d^{7/4}\e^2\jap{\t}^{-N + \max{\a_1-2, 0} + \a_4 + 1}\frac{1}{\jap{k-l}^2},
\end{align*}
which is good.
This completes the proof of \eqref{T5compstI0}.
\end{proof}

To use the bound in this lemma, we will need the following calculation. 
For \(k\neq 0\), 
\begin{equation}\label{intineq}
\begin{aligned}
\int_{s/2}^s \frac{\jap{l}}{\jap{ks-l\t}}d\t &\lesssim 
\begin{cases}
\int_{(k-l)s}^{(k-l/2)s}\jap{u}^{-1}du &\text{ if } l > 0, \\
\jap{ks}^{-1}s &\text{ if } l = 0, \\
\int^{(k-l)s}_{(k-l/2)s}\jap{u}^{-1}du &\text{ if } l < 0. 
\end{cases} \\
&\lesssim 1 + \log\jap{s} + \log\jap{l} + \log\jap{k-l}.
\end{aligned}
\end{equation}
There is a room of summability for \(\jap{l}\) and \(\jap{k-l}\), so the logarithmic term can be absorbed. 
\(\log\jap{s}\) will be controlled with one other \(\jap{s}\). 
Similarly, on \([0, s/2]\) we have 
\begin{equation}\label{intineq1}
\int_0^{s/2} \frac{\jap{l}}{\jap{ks-l\t}}d\t 
\lesssim 1 + \log\jap{s} + \log\jap{k} + \log\jap{l}.
\end{equation}

\subsection{Estimates on the nonlinear part II: Controlling \(\Tcal_{i, j}\) and \(\Dcal_{i, j}\)}
By estimating each component with Lemma \ref{lem:T5comp}, we can control \(\Tcal_{5, b}\) as follows:
\begin{lem}\label{lem:T5b}
    For \(\Tcal_{5, b}(s, R)\) defined in Proposition \ref{prop:dsqNdecomp} with \(b = 1, 2\), we have the following estimates:
    \begin{equation}\label{T5b}
        \abs{\rd_R^I \Tcal_{5, b}(s, R)} \lesssim \d^{7/4}\e\jap{s}^{-N+I}
    \end{equation}    
\end{lem}       
\begin{proof}
\step{1}
We first prove the following inequality: 
\begin{equation}\begin{aligned}\label{T5byI}
\abs{\rd_R^I \Tcal_{5, b}(s, R)} &\lesssim \sum_{i_1 + i_2 + i_3 \le I+1} \sum_{k\neq 0, l} \int_{s/2}^s \abs{\Ical_{k, l}^{i_1, i_2, i_3}(\t, s, R)} \, d\t \\
&\qquad + 
\sum_{i_1 + i_2 + i_3 \le I} \sum_{k\neq 0, l} \int_0^{s/2} \abs{k\Ical_{k, l}^{i_1, i_2, i_3}(\t, s, R)} \, d\t. \\
\end{aligned}\end{equation}
When \(\rd_R\) hits the integrand of \(\Tcal_{5, b}\), three situations are possible; it might hit \(e^{ikQ_T}\) and produce one more \(k\), or it might hit \(\wh{(p_b)}_l\), or it might hit \(e^{-ik(s-\t)\O}\wh{(\rd_b f)}_{k-l}\).
Therefore, we have 
\begin{align*}
\abs{\rd_R^I \Tcal_{5, b}} &\lesssim \sum_{j_1 + j_2 + j_3 \le I} \sum_{k\neq 0, l} 
\abs{k^{1+j_1}\int_{-\infty}^\infty\int_0^\infty \int_0^s e^{ikQ_T}\wh{(\rd_R^{j_2}p_b)}_l \cdot \rd_R^{j_3}(e^{-ik(s-\t)\O} \wh{(\rd_b f)}_{k-l})\CC(s) \, d\t d\ell dw}. \\
\intertext{For the last term, we can use \eqref{drintodell} and integrate by part \(\rd_\ell\) away, so that only \(Z\) derivative hits the free transport. \(\rd_\ell\) might hit \(e^{ikQ_T}\) to produce one more \(k\), or hit \(\wh{(p_b)}_l\). Therefore, we have}
&\lesssim \sum_{j_1 + j_2 + j_3 \le I} \sum_{k\neq 0, l} 
\abs{k^{1+j_1}\int_{-\infty}^\infty\int_0^\infty \int_0^s e^{ikQ_T}\wh{(\rd_R^{j_2}p_b)}_l \cdot Z^{j_3}(e^{-ik(s-\t)\O} \wh{(\rd_b f)}_{k-l})\CC(s) \, d\t d\ell dw}. \\
\intertext{Now, when we control the \(Z\) derivative of the free transport of \(f\), we can use \eqref{Zontrans} to convert it into \(Z\) hitting \(f\) and an extra \(k\) term. Then, the \(k\) term can be understood as \(\rd_{Q_T}\) hitting either \(p_b\) or \(f\). On \([0, s/2]\), we keep one \(k\) for a technical purpose. This proves that}
&\lesssim \sum_{i_1 + i_2 + i_3 \le I+1} \sum_{k\neq 0, l} \abs{\int_{-\infty}^\infty\int_0^\infty \int_{s/2}^s e^{ikQ_T}\wh{(\rd_R^{i_1}p_b)}_l \cdot (e^{-ik(s-\t)\O} \wh{(\rd_b\rd_{Q_T}^{i_2}Z^{i_3} f)}_{k-l})\CC(s) \, d\t d\ell dw} \\
&+ \sum_{i_1 + i_2 + i_3 \le I} \sum_{k\neq 0, l}
\abs{k\int_{-\infty}^\infty\int_0^\infty \int_0^{s/2} e^{ikQ_T}\wh{(\rd_R^{i_1}p_b)}_l \cdot (e^{-ik(s-\t)\O} \wh{(\rd_b\rd_{Q_T}^{i_2}Z^{i_3} f)}_{k-l})\CC(s) \, d\t d\ell dw}.
\end{align*}

\step{2}
Now, we invoke the known estimates on \(\Ical_{k, l}^{i_1, i_2, i_3}(\t, s, R)\) in Lemma \ref{lem:T5comp} to control the right hand side of \eqref{T5byI}, which gives us the desired estimate for \(\Tcal_{5, b}\).
We do this separately on \([s/2, s]\) and \([0, s/2]\).

First of all, in the case when \(I \ge 1\) and \(i_1 \le I\), we can use \eqref{T5complt} to get the desired decay in \(s\):
we have 
\begin{equation}\label{temp:easiest}
\sum_{k\neq 0, l} \int_{s/2}^s \abs{\Ical_{k, l}^{i_1, i_2, i_3}(\t)} \, d\t
\lesssim \d^{7/4}\e^2\jap{s}^{-N+I}\int_{s/2}^s \left(\sum_l \frac{a_l(\t)}{\jap{l}}\right)\left(\sum_{k-l}\frac{b_{k-l}(\t)}{\jap{k-l}}\right) d\t
\lesssim \d^{7/4}\e\jap{s}^{-N+I},
\end{equation}
where for the second inequality we used Cauchy--Schwarz and the fact that \(\sum_l a_l^2, \sum_{k-l} b_{k-l}^2 \le 1\).

Therefore, the case we have to cover is \(I = 0\) or \(i_1 = I+1\). 
Note that in either case we have \(i_2+i_3\le 1\), so the number of derivatives on \(f\) is always less than \(N\). 
Here we separate the sum over \(k\) and \(l\) into three parts: i) \(\jap{l}\le \jap{s}\), ii) \(\jap{s} < \jap{l} < 2\jap{k-l}\jap{s}^{2/3}\), iii) \(2\jap{k-l}\jap{s}^{2/3} \le \jap{l}\).

i) By \eqref{intineq}, 
\[
\int_{s/2}^{s}\abs{\Ical_{k, l}^{i_1, i_2, i_3}(\t)} d\t 
\le C\d^{\frac{7}{4}}\e^2\jap{s}^{-N+I+1}(\log\jap{s} + \log\jap{l}) \frac{1}{\jap{l}}\frac{\log\jap{k-l}}{\jap{k-l}^3}.
\]
Summing over \(\jap{l}\le \jap{s}\), we have 
\begin{align*}
\sum_{\jap{l}\le \jap{s}}\int_{s/2}^{s}\abs{\Ical_{k, l}^{i_1, i_2, i_3}(\t)} d\t
&\le C\d^{\frac{7}{4}}\e^2\jap{s}^{-N+I+1}
\Big(\log\jap{s} \sum_{\jap{l}\le \jap{s}}\frac{1}{\jap{l}}
+ \sum_{\jap{l}\le \jap{s}}\frac{\log\jap{l}}{\jap{l}}\Big)
\sum_{k-l}\frac{\log\jap{k-l}}{\jap{k-l}^3} \\
&\le C\d^{\frac{7}{4}}\e^2\jap{s}^{-N+I+1}(\log\jap{s})^2.
\end{align*}
This is sufficient for the desired conclusion. 

ii) In this case from \eqref{T5compltY} we have 
\begin{align*}
\abs{\Ical_{k, l}^{i_1, i_2, i_3}(\t)} &\le C\d^{\frac{7}{4}}\e^2\jap{s}^{-N+I}
\frac{\jap{l}}{\jap{ks-l\t}}\frac{1}{\jap{k-l}^3} \\
&\le C\d^{\frac{7}{4}}\e^2\jap{s}^{-N+I+1}
\frac{\jap{l}}{\jap{ks-l\t}} \frac{1}{\jap{l}^{3/2}}\frac{1}{\jap{k-l}^{3/2}}.
\end{align*}
Taking an integral with \eqref{intineq}, we have 
\begin{align*}
\sum_{k\neq 0, l}\int_{s/2}^{s}\abs{\Ical_{k, l}^{i_1, i_2, i_3}(\t)} d\t
&\le C\d^{\frac{7}{4}}\e^2\jap{s}^{-N+I+1}\log\jap{s} \sum_l\frac{\log\jap{l}}{\jap{l}^{3/2}}\sum_{k-l}\frac{\log\jap{k-l}}{\jap{k-l}^{3/2}} \\
&\le C\d^{\frac{7}{4}}\e^2\jap{s}^{-N+I+1}\log\jap{s},
\end{align*}
which is the desired inequality. 

iii) For the situation with \(\jap{s}^{2/3}\jap{k-l}\le \jap{l}\), we have \(\abs{l} \ge \frac{1}{2}\jap{s}^{2/3}\).
To utilize this we sub-divide the \(\t\)-interval once more.
When \(s-\t \le 4\jap{s}^{1/3}\), we use the fact that the length of interval is short, together with \eqref{T5compltsloss}.
\begin{equation}\label{temp:shorttau}
\begin{aligned}
\sum_{\jap{s}^{2/3}\jap{k-l}\le \jap{l}}\int_{s-4\jap{s}^{1/3}}^s \abs{\Ical_{k, l}^{i_1, i_2, i_3}(\t)} \, d\t
&\lesssim \sum_{\jap{s}^{2/3}\jap{k-l}\le \jap{l}} \int_{s-4\jap{s}^{1/3}}^s \d^{7/4}\e^2\jap{s}^{-N+I+1} \frac{a_l}{\jap{l}}\frac{b_{k-l}}{\jap{k-l}} \, d\t \\
&\lesssim \int_{s-4\jap{s}^{1/3}}^s \d^{7/4}\e^2\jap{s}^{-N+I+1} \sum_{\abs{l}\ge \frac{1}{2}\jap{s}^{2/3}}\frac{a_l}{\jap{l}} \sum_{k-l}\frac{b_{k-l}}{\jap{k-l}} \, d\t. 
\end{aligned}
\end{equation}
The second sum is convergent, and the first one has a better bound followed from the condition on \(l\): 
\[
\sum_{\abs{l}\ge \frac{1}{2}\jap{s}^{2/3}}\frac{a_l}{\jap{l}}
\le \left(\sum_l a_l^2\right)^{\frac{1}{2}}\bigg(\sum_{\abs{l}\ge \frac{1}{2}\jap{s}^{2/3}}\frac{1}{\jap{l}^2}\bigg)^{\frac{1}{2}}
\lesssim \jap{s}^{-1/3}.
\]
Therefore, we have 
\[
\sum_{\jap{s}^{2/3}\jap{k-l}\le \jap{l}} \int_{s-4\jap{s}^{1/3}}^s \abs{\Ical_{k, l}^{i_1, i_2, i_3}(\t)} \, d\t
\lesssim \d^{7/4}\e^2\jap{s}^{-N+I + 1}
\]
This is enough for our desired estimate. \\
On the other interval, \(s-\t \ge 4\jap{s}^{1/3}\), we have the following bound for the weight \(\frac{\jap{l}}{\jap{ks-l\t}}\): since \(ks-l\t = (k-l)s + l(s-\t)\), as long as \(l\neq 0\),
\begin{align*}
\abs{l(s-\t)} &\ge \frac{1}{2}\jap{s}^{2/3}\jap{k-l}4\jap{s}^{1/3} \ge 2\abs{(k-l)s}, \\
\frac{\jap{l}}{\jap{ks-l\t}} &\le \frac{2\jap{l}}{\abs{l(s-\t)}} \le \frac{4}{s-\t}.
\end{align*}
Note that even if \(l = 0\) the final inequality is true. 
Now we can use \eqref{T5compltY} to get the following bound:
\begin{equation}\label{temp:weightbound}
\begin{aligned}
\sum_{\jap{s}^{2/3}\jap{k-l}\le \jap{l}}\int_{s/2}^{s-4\jap{s}^{1/3}} \abs{\Ical_{k, l}^{i_1, i_2, i_3}(\t)} \, d\t
&\lesssim \d^{7/4}\e^2\jap{s}^{-N+I+1}\int_{s/2}^{s-4\jap{s}^{1/3}} \frac{1}{s-\t}\sum_{l}\frac{a_l(\t)}{\jap{l}}\sum_{k-l}\frac{1}{\jap{k-l}^3}  \, d\t \\
&\lesssim \d^{7/4}\e^2\jap{s}^{-N+I + 1}\int_{s/2}^{s-4\jap{s}^{1/3}} \frac{1}{s-\t} \, d\t \\
&\lesssim \d^{7/4}\e^2\jap{s}^{-N+I+1}\log\jap{s}.
\end{aligned}
\end{equation}
From \eqref{temp:shorttau}, and \eqref{temp:weightbound}, we conclude that 
\[
\sum_{\jap{s}^{2/3}\jap{k-l}\le \jap{l}}\int_{s/2}^{s} \abs{\Ical_{k, l}^{i_1, i_2, i_3}(\t)} \, d\t
\lesssim \d^{7/4}\e^2\jap{s}^{-N+I+1}\log\jap{s}.
\]
All cases have covered in i), ii), iii), so we have the desired estimate 
\[
\sum_{k\neq 0, l}\int_{s/2}^{s} \abs{\Ical_{k, l}^{i_1, i_2, i_3}(\t)} \, d\t
\lesssim \d^{7/4}\e\jap{s}^{-N+I}.
\]

\step{3}
On \([0, s/2]\), the \(\t\)-decay cannot be converted to \(s\)-decay, hence we need to use excessive number of \(Y_{k, -k\t}\) to get sufficient decay in \(s\).
This was done in the proof of Lemma \ref{lem:T5comp}, together with obtaining an integrability in \(\t\). 
We show that the sum over \(k\neq 0, l\) of the time integral of each RHS in the second part of Lemma \ref{lem:T5comp} is controlled by \(\d^{7/4}\e\jap{s}^{-N+I}\).

For the first case \eqref{T5compst}, we have the following estimate.
\begin{equation}\begin{aligned}\label{temp:forT5compst}
&\sum_{k\neq 0, l}\int_0^{s/2} \d^{\frac{7}{4}}\e^2\jap{s}^{-N+I+1}\jap{\t}^{-1}\frac{1}{k^2}a_l(\t) b_{k-l}(\t) \, d\t \\
&\lesssim \d^{\frac{7}{4}}\e^2\jap{s}^{-N+I+1}\int_0^{s/2} \jap{\t}^{-1}\sum_{k\neq 0, l}\frac{1}{k^2}a_l(\t) b_{k-l}(\t) \, d\t \\
&\lesssim \d^{\frac{7}{4}}\e^2\jap{s}^{-N+I+1}\log\jap{s} \sum_{k\neq 0}\frac{1}{k^2} \\
&\lesssim \d^{7/4}\e\jap{s}^{-N+I}.
\end{aligned}\end{equation}
Note that the second inequality is due to Cauchy--Schwarz and the fact that \(\sum_l a_l^2, \sum_{k-l} b_{k-l}^2 \le 1\), and at the last step we used \(\e\jap{s}\log\jap{s} \lesssim 1\).

For the second case \eqref{T5compstN2}, we have the following inequality: 
\[
\frac{1}{k}\min\left(\frac{1}{\jap{l}}, \frac{1}{\jap{k-l}}\right)
\le \frac{1}{k^2}\cdot (\jap{l} + \jap{k-l}) \min\left(\frac{1}{\jap{l}}, \frac{1}{\jap{k-l}}\right)
\le \frac{2}{k^2},
\]
hence we can proceed exactly as in the \eqref{temp:forT5compst} to get the desired bound.

For the third case \eqref{T5compstN1}, we need to argue as in \eqref{temp:easiest}: 
\begin{align*}
&\sum_{k\neq 0, l}\int_0^{s/2} \d^{\frac{7}{4}}\e^2\jap{s}^{-1}\jap{\t}^{-1}\frac{a_l(\t)}{\jap{l}}\frac{b_{k-l}(\t)}{\jap{k-l}} \, d\t \\
&\lesssim \d^{\frac{7}{4}}\e^2\jap{s}^{-1}\int_0^{s/2} \jap{\t}^{-1}\sum_{l}\frac{a_l(\t)}{\jap{l}}\sum_{k-l}\frac{b_{k-l}(\t)}{\jap{k-l}} \, d\t \\
&\lesssim \d^{7/4}\e^2\jap{s}^{-1}\log\jap{s},
\end{align*}
which is satisfactory to give the upper bound \(\d^{7/4}\e\jap{s}^{-2}\).

The last one is the worst case, which forces us to utilize the weight \(\jap{ks-l\t}^{-1}\).
When \(\abs{ks}\ge 2\abs{l\t}\), we have \(\jap{ks-l\t}^{-1}\le \jap{ks}^{-1}\). 
Else, we have \(\frac{1}{2}\le \frac{l\t}{ks}\), and this lets us gain one \(s\) by paying one \(\t\).
This allows us the following estimate. 
\begin{align*}
&\sum_{k\neq 0, l}\int_0^{s/2} \d^{\frac{7}{4}}\e^2\jap{s}^{-N+2}\jap{\t}^{-1} \frac{1}{k^2}\frac{1}{\jap{ks-l\t}}\left(\frac{1}{\jap{l}^2}+\frac{1}{\jap{k-l}^2}\right) \, d\t \\
&\lesssim \sum_{k\neq 0, l}\int_0^{s/2} \d^{\frac{7}{4}}\e^2\jap{s}^{-N+1}\jap{\t}^{-1} \frac{1}{k^3}\left(\frac{1}{\jap{l}^2}+\frac{1}{\jap{k-l}^2}\right) \, d\t \\
&\qquad + 
\sum_{k\neq 0, l}\int_0^{s/2} \d^{\frac{7}{4}}\e^2\jap{s}^{-N+1} \frac{1}{k^3}\frac{l}{\jap{ks-l\t}}\left(\frac{1}{\jap{l}^2}+\frac{1}{\jap{k-l}^2}\right) \, d\t \\
&\lesssim \d^{7/4}\e^2\jap{s}^{-N+1}\log\jap{s} \sum_{k\neq 0, l}\frac{1}{k^3}\left(\frac{1}{\jap{l}^2}+\frac{1}{\jap{k-l}^2}\right) \\
&\qquad 
+ \d^{7/4}\e^2\jap{s}^{-N+1}\log\jap{s} \sum_{k\neq 0, l} \frac{\log\jap{k}}{k^3}\left(\frac{1}{\jap{l}^2}+\frac{1}{\jap{k-l}^2}\right)\log\jap{l}.
\end{align*}
The sums are convergent, so we have \(\d^{7/4}\e\jap{s}^{-N}\) on RHS, which is satisfactory.

Considering these four cases, we conclude that \(\sum_{k\neq 0, l} \int_0^{s/2} \abs{k\Ical_{k, l}^{i_1, i_2, i_3}(\t)}d\t \lesssim \d^{7/4}\e\jap{s}^{-N+I}\), which is the desired bound for the second term in \eqref{T5byI}.

\end{proof}

With these controls, now we can complete the proof of \eqref{dsqNlow}.
\begin{proof}[Proof of \eqref{dsqNlow}]
Lemma \ref{lem:T1b} and Lemma \ref{lem:T5b} control \(\Tcal_{1, b}(s, R)\) and \(\Tcal_{5, b}(s, R)\) in the decomposition of \(\rd_s \qbf_N\) in Proposition \ref{prop:dsqNdecomp}. 
For \(\Tcal_{2, b}(s, R)\), \(\Tcal_{3, b}(s, R)\), and the first term of \(\Tcal_{4, b}(s, R)\), for \(k \neq 0\) these are strictly easier than \(\Tcal_{5, b}(s, R)\), considering the Fourier decomposition \eqref{TibFT}.
For the second term of \(\Tcal_{4, b}(s, R)\), we lose one more \(\jap{s}\) from \(\rd_H\), but this is compensated by \(\rd_s H\) which gives us \(\jap{s}^{-N}\).

For the zero mode, because there is no \(e^{-ik(s-\t)\O}\) term, nothing grows in time, and the decay follows from \(\rd_s A\), \(\rd_s \bar{\CC}\) and \(\rd_s H\).
For instance, for \(\Tcal_{4, b}(s, R)\): using \eqref{drintodell} as before, 
\begin{align*}
&\abs{\rd_R^I \sum_l \int_{-\infty}^\infty \int_0^\infty \int_0^s 
\rd_H(\wh{(p_b)}_l\wh{(\rd_b f)}_{-l}) \cdot \rd_s H \bar{\CC} \, d\t d\ell dw} \\
&\lesssim \sum_{i_1 + i_2 + i_3 + i_4 \le I+1} \sum_l \int_{-\infty}^\infty \int_0^\infty \int_0^s
\abs{\wh{(\rd_R^{i_1}p_b)}_l}\abs{\wh{(\rd_b\rd_{Q_T}^{i_2}Z^{i_3} f)}_{-l}} \cdot \abs{\rd_s \rd_{R, w, \ell}^{i_4} H} \CC \, d\t d\ell dw \\
&\lesssim \sum_{i_1 + i_2 + i_3 + i_4 \le I+1}\int_0^s
\norm{\rd_R^{i_1}p_b}_{L^\infty}\norm{\rd_b\rd_{Q_T}^{i_2}Z^{i_3} f}_{L^\infty}\norm{\rd_s \rd_{R, w, \ell}^{i_4} H}_{L^\infty} \big(\sum_l a_l(\t) b_{-l}(\t)\big) d\t \\
&\lesssim \sum_{i_1 + i_2 + i_3 + i_4 \le I+1}
\d^{5/2}\e^3\jap{s}^{-N + \max{i_4-2, 0}}\int_0^s\jap{\t}^{-N + \max{i_1-2, 0} + 1}d\t, \\
\end{align*}
which is far smaller than we need. 
The zero mode for \(\Tcal_{2, b}(s, R)\) and \(\Tcal_{3, b}(s, R)\) can be controlled in exactly the same way. 

Therefore, we have \(\rd_R^I \Tcal_{a, b}(s, R) \lesssim \d\e\jap{s}^{-N+I}\) for all \(a = 1, 2, 3, 4, 5\) and \(b = 1, 2\), which proves the desired estimate for \(\rd_s \qbf_N\).
\end{proof}

Now, we turn to the control of \(\qbf_N-\bar{\qbf}_N\). We need to consider up to \(N-1\)-th derivatives. 
\begin{proof}[Proof of \eqref{qNqNbarlow}]
We follow the decomposition given in Proposition \ref{prop:qqbardecomp}, and control each term separately.

For the first term, we prove that 
\begin{equation}\label{temp:nonzeromode}
\rd_R^I\Dcal_{1, b}(s, R), \rd_R^I\Dcal_{1, b}(T, R) \lesssim \d^{7/4}\e\jap{s}^{-N+I},
\end{equation}
meaning that we do not expect the terms to be cancelled. 
It suffices to show the estimate for \(\Dcal_{1, b}(s, R)\) as the one for \(\Dcal_{1, b}(T, R)\) can be obtained in the same way. \\
As we did for \(\Tcal_{5, b}\), we decompose \(\Dcal_{1, b}\) into \(\Ical^{i_1, i_2, i_3}_{k, l}\) terms: 
\begin{align*}
\abs{\rd_R^I\Dcal_{1, b}} 
&\lesssim 
\sum_{j_1+j_2+j_3 \le I} \sum_{k\neq 0, l} \abs{k^{j_1}\int_{-\infty}^\infty\int_0^\infty \int_0^s e^{ikQ_T}\wh{(\rd_R^{j_2}p_b)}_l \cdot \rd_R^{j_3}(e^{-ik(s-\t)\O} \wh{(\rd_b f)}_{k-l})\CC(s) \, d\t d\ell dw} \\
&\lesssim 
\sum_{i_1+i_2+i_3 \le I} \sum_{k\neq 0, l} \abs{\int_{-\infty}^\infty\int_0^\infty \int_0^s e^{ikQ_T}\wh{(\rd_R^{i_1}p_b)}_l \cdot (e^{-ik(s-\t)\O} \wh{(\rd_b\rd_{Q_T}^{i_2}Z^{i_3} f)}_{k-l})\CC(s) \, d\t d\ell dw} \\
&\lesssim \sum_{i_1+i_2+i_3 \le I} \sum_{k\neq 0, l}\int_0^s \abs{\Ical_{k, l}^{i_1, i_2, i_3}(\t, s, R)} \, d\t.
\end{align*}
This term is strictly smaller than the one we have for \(\Tcal_{5, b}\) in \eqref{T5byI}, as the sum of index is \(I\) on \(\t\in [s/2, s]\), and there is no \(k\) factor in front of the integral on \(\t\in [0, s/2]\).
Therefore, as long as \(I\le N-2\), we already know that the right hand side is bounded by \(\d^{7/4}\e\jap{s}^{-N+I}\) from the proof of Lemma \ref{lem:T5b}.

When \(I = N-1\), we can directly estimate \(\Dcal_{1, b}\) as follows. 
\begin{align*}
\abs{\rd_R^{N-1}\Dcal_{1, b}}
&\lesssim \sum_{i_1+i_2+i_3 \le N-1} \sum_{k\neq 0, l} \int_0^s 
\norm{\wh{(\rd_R^{i_1}p_b)}_l}_{L^\infty}\norm{\wh{(\rd_b \rd_{Q_T}^{i_2}Z^{i_3} f)}_{k-l}}_{L^\infty}\CC(s) \, d\t \\
&\lesssim \sum_{i_1+i_2+i_3 \le N-1} \sum_{k\neq 0, l} \frac{1}{\jap{l}\jap{k-l}} \int_0^s 
\norm{\wh{(\rd_R^{1+i_1}p_b)}_l}_{L^\infty}\norm{\wh{(\rd_b \rd_{Q_T}^{1+i_2}Z^{i_3} f)}_{k-l}}_{L^\infty}\CC(s) \, d\t \\
&\lesssim \begin{cases}
\displaystyle
\sum_{i_1+i_2+i_3 \le N-1} \d^{7/4}\e^2\int_0^s \jap{s}^{-N + \max{i_1-1, 0} + 1} \, d\t &\text{ if } i_2 + i_3 \le N-2, \\
\displaystyle
\sum_{i_1+i_2+i_3 \le N-1} \d^{7/4}\e^2\int_0^s \jap{s}^{-N + 2} \, d\t &\text{ if } i_2 + i_3 = N-1,
\end{cases} \\
&\lesssim \d^{7/4}\e\jap{s}^{-1}.
\end{align*}
This proves that \(\abs{\rd_R^I\Dcal_{1, b}} \lesssim \d^{7/4}\e\jap{s}^{-N+I}\).

We turn to \(\Dcal_{2, b}(s, T, R)\). After taking \(\rd_R^I\) and using \eqref{drintodell}, we can write \(\rd_R^I \Dcal_{2, b}(s, T, R)\) in the following way. 
\begin{align*}
&\abs{\rd_R^I \Dcal_{2, b}(s, T, R)} \\
&\lesssim \sum_{i_1 + i_2 + i_3 \le I} \sum_l \int_{-\infty}^\infty \int_0^\infty \int_0^s
\bigg|\wh{(\rd_R^{i_1}p_b)}_l\wh{(\rd_b Z^{i_2}f)}_{-l}(H, M)\rd_R^{i_3}\left(\frac{\pi}{R^2A^2}\bar{\CC}(s)\right) \\
&\hspace{45mm}
- \wh{(\rd_R^{i_1}p_b)}_l\wh{(\rd_b Z^{i_2}f)}_{-l}(H_T, M)\rd_R^{i_3}\left(\frac{\pi}{R^2\Abar^2}\bar{\CC}(T)\right)\bigg|\, d\t d\ell dw \\
&\lesssim \sum_{i_1 + i_2 + i_3 \le I} \int_{-\infty}^\infty \int_0^\infty \int_0^s
\abs{\sum_l \wh{(\rd_R^{i_1}p_b)}_l\wh{(\rd_bZ^{i_2}f)}_{-l}(H, M)}
\abs{\rd_R^{i_3}\Big(\frac{\pi}{R^2A^2}\bar{\CC}(s)-\frac{\pi}{R^2\Abar^2}\bar{\CC}(T)\Big)} \\
&\hspace{10mm}
+ \sum_l\abs{\wh{(\rd_R^{i_1}p_b)}_l\wh{(\rd_bZ^{i_2}f)}_{-l}(H, M) - \wh{(\rd_R^{i_1}p_b)}_l\wh{(\rd_bZ^{i_2}f)}_{-l}(H_T, M)}
\abs{\rd_R^{i_3}\Big(\frac{\pi}{R^2\Abar^2}\bar{\CC}(T)\Big)} \, d\t d\ell dw.
\end{align*}
We have to estimate the difference of products, hence we need four inequalities. For \(i_1 +i_2 + i_3\le N-1\),
\begin{align*}
\abs{\rd_R^{i_3} \Big(\frac{\bar{\CC}(s)}{A^2}-\frac{\bar{\CC}(T)}{\Abar^2}\Big)}
&\lesssim \d^{3/4}\e\jap{s}^{-N+\max{i_3-2, 0}}, \\
\abs{\rd_R^{i_3} \Big(\frac{\bar{\CC}(T)}{\Abar^2}\Big)} &\lesssim 1, \\
\abs{\sum_l \wh{(\rd_R^{i_1}p_b)}_l\wh{(\rd_bZ^{i_2}f)}_{-l}(H, M)}
&\lesssim \norm{\rd_R^{i_1}p_b}_{L^\infty}\norm{\rd_bZ^{i_2}f}_{L^\infty} \\
&\lesssim \d^{7/4}\e^2\jap{\t}^{-N + \max{i_1-2, 0} + 1}, \\
\sum_l\abs{\wh{(\rd_R^{i_1}p_b)}_l\wh{(\rd_bZ^{i_2}f)}_{-l}(H, M) - \wh{(\rd_R^{i_1}p_b)}_l\wh{(\rd_bZ^{i_2}f)}_{-l}(H_T, M)}
&\lesssim \d^{5/2}\e^3\jap{s}^{-N}\jap{\t}^{-1}.
\end{align*}
The last inequality is due to the following calculation:
\begin{align*}
\text{ LHS } &\le \abs{H-H_T}\abs{\int_0^1 
\sum_l \rd_H(\wh{(\rd_R^{i_1}p_b)}_l\wh{(\rd_bZ^{i_2}f)}_{-l})(H_T + p(H-H_T), M) dp} \\
&\lesssim \abs{H-H_T}\int_0^1 \sum_l 
\norm{\rd_R^{1+i_1}p_b}\norm{\rd_bZ^{i_2}f} + \norm{\rd_R^{i_1}p_b}\norm{\rd_b\rd_HZ^{i_2}f} \, dp \\
&\lesssim \d^{5/2}\e^3\jap{s}^{-N}\jap{\t}^{-1}.
\end{align*}
Therefore, we have
\begin{align*}
\abs{\rd_R^I \Dcal_{2, b}(s, T, R)}
&\lesssim \sum_{i_1 + i_2 + i_3 \le I} \int_0^s \d^{5/2}\e^3\jap{s}^{-N + \max{i_3-2, 0}}\jap{\t}^{-1} d\t \\
&\lesssim \d^{5/2}\e^3\jap{s}^{-N+I}\log\jap{s},
\end{align*}
which is far better than we need. 

For \(\Dcal_{3, b}(s, T, R)\), we can use the decay of \(p_b\); in case if \(b = 2\) we use \eqref{dHintoell} to \(\rd_H\) on \(f\).
This allows us to estimate as 
\begin{align*}
\abs{\rd_R^I \Dcal_{3, b}(s, T, R)}
&\lesssim \sum_{i_1 + i_2 \le I+1}\sum_l\int_{-\infty}^{\infty}\int_0^\infty \int_s^T
\abs{\wh{(\rd_R^{i_1}p_b)}_l}\abs{\wh{(Z^{i_2}f)}_{-l}} \bar{\CC}(T) \, d\t d\ell dw \\
&\lesssim \sum_{i_1 + i_2 \le I+1}\int_s^T \d^{7/4}\e^2\jap{\t}^{-N + \max{i_1-2, 0}} d\t \\
&\lesssim \d^{7/4}\e\jap{s}^{-N+I}.
\end{align*}
This completes the proof. 
\end{proof}

Therefore, we have completed the proof of Proposition \ref{prop:qN}.
Together with the linear estimates in Proposition \ref{prop:qL}, we get estimates \eqref{dsqlow} and \eqref{qqbarlow}, which completes the set of inequalities in Proposition \ref{prop:Density}.


\section{Estimates on the Metric Components}\label{sec:Elliptic}
Now we estimate the metric components \(\a, A, \b, K\) and their derivatives.
Using the estimates on the density components which are the source of the ODE formulation of Einstein's equation \eqref{absODE}, we can recover the bootstrap assumptions for the metric components \eqref{BA0}--\eqref{BAHK1}.

We first do some reduction. From the mean curvature equation \eqref{eq:MeanCurv}, we can write \(K\) in terms of \(\a\) and \(\b\) as 
\[
K = \frac{2R}{3\a}\left(\frac{\b}{R}\right)' + \frac{\Psi}{3},
\]
so the estimate on \(K\) \eqref{BAK0}, \eqref{BAK1}, \eqref{BAK2}, \eqref{BAHK0}, and \eqref{BAHK1} follows from their counterparts on \(\a, A, \b\), together with the size condition on \(\Psi-\Psi_0\) \eqref{nonlinthm-initAK}. \\
For \eqref{BA2}, for \(I = N+2\) the result follows from \eqref{BA0}, hence it suffices to show this estimate up to \(I\le N+1\).

Therefore, we need to show five estimates: two of them are decaying, and the other three are non-decaying. 
For each of them we should use the corresponding control on density variables: for \eqref{BA0} we use \eqref{qlow}, for \eqref{BA1} we use \eqref{dsqlow}, for \eqref{BA2} with \(I\le N+1\) we use \eqref{qqbarlow}, for \eqref{BAH0} we use \eqref{qtop}, and for \eqref{BAH1} we use \eqref{dsqtop}.
The rest of this section is the proof for Theorem \ref{thm:BT}.

\subsection{Estimates for higher order derivatives}\label{subsec:metrictop}
We improve the estimates \eqref{BA0}, \eqref{BAH0}, and \eqref{BAH1}, which are the non-decaying estimates for the metric components and their derivatives.
In the meantime, we display the estimates that we use to control the metric components, which will be used in the next subsections with the same format, but with different estimates on the density variables.

\begin{proof} [Improve \eqref{BA0} and \eqref{BAH0}]
We wrote Einstein's equation in the form of \eqref{absODE}, and we can use the estimates on the density variables to control the source term of this ODE system. 
Because we know that the Schwarzschild solution \(\vec{X}_1 = (A_0, A_0', \a_0, \a'_0, K_0, \b_0)\) is the solution of this ODE system with zero source term, we will show that the size of the source term \(\vec{S} = (\r, \tr T, j, S_R, \Psi-\Psi_0)\) is small, and use Lemma \ref{lem:ODEstab} to get the desired estimates.
Furthermore, since we know that the source term \(\vec{S}\) is supported in \([m/4, R_{\textrm{out}}]\), it suffices to solve the ODE from \(R_{\textrm{out}}\) toward inside. 
This makes the range of \(R\) that we estimate the solution to be compact.

The estimates we proved in Proposition \ref{prop:Density} will be used now. 
\eqref{qlow} proves that \(\norm{\vec{S}}_{C^I} \lesssim \d\e\) for \(I\le N\), and \eqref{qtop} proves \(\norm{\vec{S}}_{C^I} \lesssim \d\e\jap{s}\) for \(I = N+1\).
Therefore, by Lemma \ref{lem:ODEstab}, we have 
\[
\norm{\vec{X}-\vec{X}_1}_{C^{N+1}} \lesssim \d\e, \qquad
\norm{\vec{X}-\vec{X}_1}_{C^{N+2}} \lesssim \d\e\jap{s}.
\]
This is sufficient to improve \eqref{BA0} and \eqref{BAH0}, except that we are short by one derivative for \(\b\). 
For this we use the equation \eqref{eq:MeanCurv} again; this allows us to control the missing derivative of \(\b\) in terms of the already controlled quantities.
\end{proof}

\begin{proof}[Improve \eqref{BAH1}]
Now we need to estimate \(\rd_s\) of the metric coefficient.
For this, we differentiate the equation \(\vec{X}'(R) = \vec{F}(\vec{X}(R), \vec{S}(R), R)\) in \(s\): we get
\begin{equation}\label{dsODE}
\vec{Y}'(R) = \vec{F}_{\vec{X}}(\vec{X}(R), \vec{S}(R), R)\cdot \vec{Y}(R) + \vec{F}_{\vec{S}}(\vec{X}(R), \vec{S}(R), R)\cdot \rd_s\vec{S}(R),
\end{equation}
where \(\vec{Y} = \rd_s \vec{X}\). 
The Schwarzschild solution is time-invariant, so the above equation in this case reduces to 
\[
0 = \vec{F}_{\vec{X}}(\vec{X}_1(R), 0, R)\cdot 0 + \vec{F}_{\vec{S}}(\vec{X}_1(R), 0, R)\cdot 0.
\]
Considering \(\vec{Y}\) as a variable and \(\vec{X}, \vec{S}, \rd_s \vec{S}\) as the sources, we have the following estimate: 
\[
\norm{\rd_s X}_{C^{I+1}} \lesssim \norm{\vec{X}-\vec{X}_1}_{C^I} + \norm{\vec{S}}_{C^I} + \norm{\rd_s \vec{S}}_{C^I}.
\] 
For \eqref{BAH1} we need to consider the case \(I = N-1\). \eqref{qlow}, \eqref{dsqtop} and the result above shows that the RHS can be controlled by \(\d\e\). 
Note that \(\rd_s\Psi = 0\), so this does not contribute to RHS.  
This, together with one gain of derivative for \(\b\) by \eqref{eq:MeanCurv}, improves our bootstrap assumption \eqref{BAH1}.
\end{proof}

\subsection{Estimates for lower order derivatives: \(\rd_s\) of \(\a, A, \b\)}\label{subsec:metricdsa}
We need to improve the decaying estimate \eqref{BA1}. 

\begin{proof}[Improve \eqref{BA1}]
We can rewrite \eqref{dsODE} as follows: 
\[
\vec{Y}'(R) = \vec{G}^{(1)}(R)\cdot\vec{Y}(R) + \vec{G}^{(2)}(R)\cdot\rd_s\vec{S}(R).
\]
We already know that \(\vec{G}^{(1)}\) and \(\vec{G}^{(2)}\) are smooth and bounded in the range of the variables we have.
Therefore, we can use Gronwall's inequality with induction: for \(I = 0, 1\) from Gronwall's inequality we get 
\[
\abs{\frac{d}{dR}\vec{Y}(R)} \le C\abs{\vec{Y}(R)} + C\abs{\rd_s \vec{S}(R)},
\]
\[
\norm{\vec{Y}}_{C^0} = 
\norm{\rd_s(\a, \a', A, A', \b, K)}_{C^0} \lesssim \norm{\rd_s\vec{S}}_{C^0} \lesssim \d\e\jap{s}^{-N},
\]
by \eqref{dsqlow}. 
Note that the variable \(R\) we consider moves in a compact range, hence all the implicit constant remains bounded. 
This, together with the derivative gain for \(\b\) from \eqref{eq:MeanCurv}, improves the estimates for \(I = 0, 1\). 
Plugging this to the equation, we get 
\[
\norm{\vec{Y}'}_{C^0} = \norm{\rd_R\rd_s(\a, \a', A, A', \b, K)}_{C^0} \lesssim \d\e\jap{s}^{-N}.
\]
This improves the estimates for \(I = 2\).

Now, say we have desired estimates for \(I < l\), where \(3\le l \le N\). 
Then, by differentiating this ODE \(l-2\) times, we get the following: Note that \(\vec{G}^{(1)}\) and \(\vec{G}^{(2)}\) have sufficient regularity for this. 
\[
\vec{Y}^{(l-1)}(R) = \vec{G}^{(1)}(R)\cdot\vec{Y}^{(l-2)}(R) + \vec{G}^{(2)}(R)\cdot\rd_s\vec{S}^{(l-2)}(R) + \text{lower order terms}.
\]
Here, lower order terms contains \(\vec{Y}^{(i)}\) and \(\rd_s\vec{S}^{(i)}\) for \(i < l-2\), which are already controlled by the induction hypothesis, or has better decay than the term for \(i = l-2\). 
Hence, we get
\[
\norm{\rd_R^{l-1}\rd_s(\a, \a', A, A', \b, K)}_{C^0} \lesssim \norm{\rd_R^{l-2}\rd_s\vec{S}}_{C^0} + \text{lower order terms} \lesssim \d\e\jap{s}^{-N+l-2}.
\]
This, together with derivative gain for \(\b\) with \eqref{eq:MeanCurv} completes the proof for the case when \(I = l\).
Hence, we get the improved estimates of \eqref{BA1}.
Note that when we estimate the term with \(l = N\), we use \(\rd_R^{N-2}\rd_s\vec{S}\), where we can still use \eqref{dsqlow}.
\end{proof}

\subsection{Estimates for lower order derivatives:  \(\a-\bar{\a}, A-\Abar, \b-\bar{\b}\)}\label{subsec:metricaabar}

\begin{proof}[Improve \eqref{BA2}]
Because \((\a, A, \b, K)\) and \((\bar{\a}, \Abar, \bar{\b}, \Kbar)\) is the solution at two different time \(s\) and \(T\), we need to subtract two ODEs: 
\[
(\vec{X}_s-\vec{X}_T)' = \vec{F}(\vec{X}_s, \vec{S}_s, R) - \vec{F}(\vec{X}_T, \vec{S}_T, R).
\]
Gronwall's inequality gives us 
\begin{align*}
\abs{\frac{d}{dR}(\vec{X}_s-\vec{X}_T)(R)} &\le \norm{\vec{F}_{\vec{X}}}_{C^0}\abs{(\vec{X}_s-\vec{X}_T)(R)}
+ \norm{\vec{F}_{\vec{S}}}_{C^0}\abs{(\vec{S}_s-\vec{S}_T)(R)}, \\
\norm{\vec{X}_s-\vec{X}_T}_{C^0} &\lesssim \norm{\vec{S}_s-\vec{S}_T}_{C^0} 
\lesssim \d\e\jap{s}^{-N},
\end{align*}
which already covers the case for \(I = 0, 1\). Putting this back to the equation, we get the estimate for \(I = 2\). 

For \(I\ge 3\), we again perform induction. 
Say the estimate is true for \(I < l\), for given \(3\le l \le N+1\).
Then, by differentiating the equation \(l-2\) times, we get the estimate
\begin{align*}
\norm{(\vec{X}_s-\vec{X}_T)^{(l-1)}}_{C^0} 
&\lesssim \norm{(\vec{X}_s-\vec{X}_T)^{(l-2)}}_{C^0} + \norm{(\vec{S}_s-\vec{S}_T)^{(l-2)}}_{C^0} + \text{lower order terms} \\
&\lesssim \d\e\jap{s}^{-N+l-2}.
\end{align*}
This, together with derivative gain for \(\b\) with \eqref{eq:MeanCurv} gives the estimates for the case \(I = l\).
Note that for \(l = N+1\) case we use the estimate \eqref{qqbarlow} on \(\rd_R^{N-1}((\r, \tr T, j, S_R)-(\bar{\r}, \overline{\tr T}, \jbar, \bar{S}_R))\).
Hence, the proof of improved version of \eqref{BA2} is completed.
\end{proof}

This completes the proof of Theorem \ref{thm:BT}, which was the main bootstrap argument.

\section{Continuity Argument: Proof of the Main Theorem}\label{sec:PutEverything}
\begin{proof}[Proof of Theorem \ref{thm:nonlin}]
    This is a standard continuity argument. 
    Let \(B\subset [0, T_f]\) be defined as follows: 
    \begin{align*}
    B \coloneqq \{T_B\in [0, T_f]: &\text{ the solution exists on } [0, T_B) \\
    &\text{ and satisfies the bootstrap assumptions on }  0 < s < T < T_B\}.    
    \end{align*}
    We first prove that \(B\) is nonempty and clopen in the subspace topology on \([0, T_f]\).
    \begin{itemize}
        \item \(B\) is nonempty: This is due to the local well-posedness of the system, which guarantees that there exists a small time interval \([0, T_B]\) on which the solution exists and satisfies the bootstrap assumptions. 
        \item \(B\) is closed: This is clear from the continuity of the quantities involved in the bootstrap assumptions \eqref{BA0}--\eqref{BAHK1}.
        \item \(B\) is open: Let \(T_B \in B\). Then, we can use the local existence theorem at time \(T_B\): The metric support condition follows from the boundary condition that we imposed, and the support condition on \(f\) follows from Proposition \ref{prop:suppf}.
        The smallness condition follows from the conclusion of Theorem \ref{thm:BT} and Theorem \ref{thm:Vlasov}.
        Therefore, there is \(T_B' > T_B\) such that the solution with enough regularity exists on \([0, T_B')\).
        Because we improved the constant in the bootstrap assumption by Theorem \ref{thm:BT} and Theorem \ref{thm:Vlasov}, the solution satisfies the bootstrap assumptions on \([0, T_B')\), again by continuity. 
        There is one subtlety as we choose two times \(s < T\); in the case when \(T < T_B\) there is nothing to prove. 
        Even when \(s < T_B < T\) or \(T_B < s < T\), by the continuity of the solution, choosing \(T_B'\) to be closer to \(T_B\) allows us to have \eqref{BA2} and \eqref{BAK2}.
        This completes the proof. 
    \end{itemize}
    Therefore, \(B = [0, T_f]\). Together with Theorem \ref{thm:BT}, Proposition \ref{prop:Density}, and Theorem \ref{thm:Vlasov}, this shows that the existence of the solution and (2), (3), (4), (5) in Theorem \ref{thm:nonlin} hold.
    The estimate (1) in Theorem \ref{thm:nonlin} follows from the fact that the change of variables map \((s, R, w, \ell)\mapsto (t, Q_{T_f}, H, M)\) is regular, which is the content of Proposition \ref{prop:regDAAV}.
    This completes the proof of the main theorem.
\end{proof}

\bibliographystyle{plain}
\bibliography{refs}

\end{document}